\documentclass[english,10pt]{amsart}

\usepackage[english]{babel}

\usepackage{amscd,amssymb,amsfonts,amsmath}
\usepackage{tikz-cd}

\usepackage{bm}
\usepackage{graphics}
\usepackage{epsfig}
\usepackage{mathrsfs}
\usepackage{xcolor}
\DeclareMathAlphabet{\mathscrbf}{OMS}{mdugm}{b}{n}

\definecolor{violet}{rgb}{0.0,0.2,0.7}
\definecolor{rouge2}{rgb}{0.8,0.0,0.2}
\usepackage{hyperref}
\hypersetup{
    unicode=false,          
    pdftoolbar=true,        
    pdfmenubar=true,        
    pdffitwindow=false,     
    pdfstartview={FitH},    
    pdftitle={},    
    pdfauthor={},     
    colorlinks=true,       
   linkcolor=violet,          
    citecolor=rouge2,        
    filecolor=black,      
    urlcolor=cyan}           

\usepackage[text={6.7in,9.2in},centering]{geometry}

\makeatletter
\renewcommand\subsection{\@startsection{subsection}{2}%
  \z@{.5\linespacing\@plus.7\linespacing}{-.5em}%
  {\normalfont\sffamily}}  
\makeatother

\newcommand{\Pic}{\textup{Pic}}

\newcommand{\Proj}{\textup{Proj}}

\newcommand{\Frobabs}{\textup{F}_{\textup{abs}}}

\newcommand{\Hol}{\textup{Hol}}

\renewcommand{\phi}{\varphi}
\newcommand{\into}{\hookrightarrow}
\newcommand{\map}{\dashrightarrow}

\newcommand{\wt}{\widetilde}

\newcommand{\wb}{\overline}

\renewcommand{\le}{\leqslant}
\renewcommand{\ge}{\geqslant}

\newcommand{\bA}{\textbf{A}}
\newcommand{\bB}{\textup{\textbf{A}}}
\newcommand{\bC}{\textup{\textbf{C}}}
\newcommand{\bD}{\textbf{D}}
\newcommand{\bE}{\textbf{E}}

\newcommand{\bH}{\textbf{H}}

\newcommand{\bK}{\textbf{K}}

\newcommand{\bQ}{\mathbb{Q}}

\newcommand{\bS}{\textbf{S}}
\newcommand{\bT}{\textbf{T}}
\newcommand{\bU}{\textbf{U}}

\newcommand{\bX}{\textbf{X}}

\newcommand{\bZ}{\textbf{Z}}
\renewcommand{\bB}{\textbf{B}}
\newcommand{\bP}{\textbf{P}}

\newcommand{\sA}{\mathscr{A}}
\newcommand{\sB}{\mathscr{B}}

\newcommand{\sE}{\mathscr{E}}

\newcommand{\sG}{\mathscr{G}}
\newcommand{\sH}{\mathscr{H}}
\newcommand{\sI}{\mathscr{I}}

\newcommand{\sL}{\mathscr{L}}
\newcommand{\sM}{\mathscr{M}}
\newcommand{\sN}{\mathscr{N}}
\newcommand{\sO}{\mathscr{O}}

\newcommand{\sbfE}{\mathscrbf{E}}
\newcommand{\sbfG}{\mathscrbf{G}}

\newtheorem{thm}{Theorem}[section]

\newtheorem{lemma}[thm]{Lemma}
\newtheorem{cor}[thm]{Corollary}

\newtheorem{prop}[thm]{Proposition}

\newtheorem*{thm*}{Theorem}
\theoremstyle{definition}
\newtheorem{defn}[thm]{Definition}

\newtheorem{notation}[thm]{Notation}

\newtheorem{defn-thm}[thm]{Definition-Theorem} 
\newtheorem{defn-lemma}[thm]{Definition-Lemma}

\theoremstyle{remark}
\newtheorem{claim}[thm]{Claim}
\newtheorem{fact}[thm]{Fact}

\newtheorem*{not-and-def}{Notation and definitions}
\newtheorem{assumption}[thm]{Assumption} 
\newtheorem{rem}[thm]{Remark}

\newtheorem{exmp}[thm]{Example}

\numberwithin{equation}{section}

\def\factor#1.#2.{\left. \raise 2pt\hbox{$#1$} \right/\hskip -2pt\raise -2pt\hbox{$#2$}}

\begin{document} 

\title[Codimension one foliations with trivial canonical class on singular spaces]{Codimension one foliations with numerically trivial canonical class on singular spaces}

\author{St\'ephane \textsc{Druel}}

\address{St\'ephane Druel: Institut Fourier, UMR 5582 du CNRS, Universit\'e Grenoble Alpes, CS 40700, 38058 Grenoble cedex 9, France} 

\email{stephane.druel@math.cnrs.fr}



\begin{abstract}
In this article, we describe the structure of codimension one foliations with canonical singularities and numerically trivial canonical class on varieties with terminal singularities, extending a result of Loray, Pereira and Touzet to this context.
\end{abstract}

\maketitle
{\small\tableofcontents}

\section{Introduction}

In the last decades, much progress has been made in the classification of complex projective varieties.
The general viewpoint is that complex projective varieties $X$ with mild singularities should be classified according to the behavior of their canonical class $K_X$.
Similar ideas can be applied in the context of foliations on complex projective varieties.
If $\sG$ is a foliation on a  complex projective variety, we define its canonical class to be 
$K_{\sG}=-c_1(\sG)$. In analogy with the case of projective varieties, one expects the numerical properties of $K_{\sG}$
to reflect geometric aspects of $\sG$ (see \cite{brunella}, \cite{mcquillan08}, \cite{spicer}, \cite{cascini_spicer}, \cite{fano_fols}, \cite{codim_1_del_pezzo_fols}, \cite{fano_fols_2}, \cite{codim_1_mukai_fols}, \cite{lpt}, \cite{touzet}).

The Beauville-Bogomolov decomposition theorem asserts that any compact K\"ahler manifold with numerically trivial canonical bundle admits an \'etale cover that decomposes into a product of a torus, and irreducible,
simply-connected Calabi-Yau, and holomorphic symplectic manifolds (see \cite{beauville83}).
In \cite{touzet}, Touzet obtained a foliated version of the Beauville-Bogomolov decomposition theorem
for codimension one regular foliations with numerically trivial canonical bundle on compact K\"ahler manifolds.
The statement below follows from \cite[Th\'eor\`eme 1.2]{touzet} and \cite[Lemma 5.9]{bobo}.

\begin{thm*}[Touzet]
Let $X$ be a complex projective manifold, and let
$\sG$ be a regular codimension one foliation on $X$ with $K_\sG\equiv 0$.
Then one of the following holds.
\begin{enumerate}
\item There exists a $\mathbb{P}^1$-bundle structure $\phi\colon X \to Y$ onto a complex projective manifold $Y$ with $K_Y\equiv 0$, and $\sG$ induces a flat holomorphic connection on $\phi$.
\item There exists an abelian variety $A$ as well as a simply connected projective manifold $Y$ with $K_Y\equiv 0$, and a finite \'etale cover $f: A \times Y \to X$ such that $f^{-1}\sG$ is the pull-back of a codimenion 1 linear foliation on $A$.
\item There exists a smooth complete curve $B$ of genus at least $2$ as well as a complex projective manifold $Y$ with $K_Y\equiv 0$, and a finite \'etale cover $f \colon B \times Y \to X$ such that $f^{-1}\sG$ is
induced by the projection morphism $B \times Y \to B$.
\end{enumerate}
\end{thm*}

Loray, Pereira and Touzet recently described the structure of codimension one foliations with canonical singularities 
(we refer to Section \ref{section:singularities} for this notion) and numerically trivial canonical class on complex projective manifolds in \cite{lpt}. However, with the development of the minimal model program, it became clear that
singularities arise as an inevitable part of higher dimensional life. In this article, we extend their result to the singular setting. Our first main result is the following.

\begin{thm}\label{thm_intro:main}
Let $X$ be a normal complex projective variety with canonical singularities, and let $\sG$ be a codimension one
foliation on $X$ with canonical singularities. Suppose furthermore that $K_\sG\sim_{\mathbb{Q}}0$.
Then one of the following holds.
\begin{enumerate}
\item There exist a smooth complete curve $C$, a complex projective variety $Y$ with canonical singularities and 
$K_Y \sim_\mathbb{Z}0$, as well as a quasi-\'etale cover $f \colon Y \times C \to X$ such that $f^{-1}\sG$ is induced by the projection $Y \times C \to C$.
\item There exist complex projective varieties $Y$ and $Z$ with canonical singularities, as well as a quasi-\'etale cover 
$f \colon Y \times Z \to X$ and a foliation $\sH\cong \sO_Z^{\, \dim Z -1}$ on $Z$ 
such that $f^{-1}\sG$ is the pull-back of $\sH$ via the projection $Y \times Z \to Z$. In addition, we have
$K_Y \sim_\mathbb{Z} 0$, $Z$ is an equivariant compactification of a commutative algebraic group of dimension at least $2$,
and $\sH$ is induced by a codimension one Lie subgroup.
\end{enumerate}
\end{thm}

\begin{rem}
In the setup of Theorem \ref{thm_intro:main}, suppose that $Z$ is rational, and set $n:=\dim Z$. Then $Z$ is an equivariant compactification of $(\mathbb{G}_m)^{n}$ or $(\mathbb{G}_m)^{n-1}\times \mathbb{G}_a$ (see proof of Lemma \ref{lemma:no_residue}). In either case, $Z$ is a toric variety by \cite{arzhantsev_kotenkova}.
\end{rem}

We also show that abundance holds provided that $X$ is terminal. Theorem \ref{thm_intro:main} together with Theorem \ref{thm_intro:abundance} then give the structure of codimension one foliations with canonical singularities and numerically trivial canonical class on varieties with terminal singularities.

\begin{thm}\label{thm_intro:abundance}
Let $X$ be a normal complex projective variety, and let $\sG$ be a codimension one
foliation on $X$ with canonical singularities and $K_\sG\equiv 0$. Suppose in addition that either $X$ has terminal singularities, or that $X$ has canonical singularities and $K_\sG$ is Cartier. Then $K_\sG$ is torsion.
\end{thm}

If $\dim X=3$, Theorem \ref{thm_intro:abundance} is a special case of \cite[Theorem 1.7]{cascini_spicer}.

\medskip

As a consequence of Theorem \ref{thm_intro:main}, we describe the structure of weakly regular (we refer to Section \ref{section:regular_foliations} for this notion) codimension one foliations with numerically trivial canonical class, extending \cite[Th\'eor\`eme 1.2]{touzet} to this context.

\begin{cor}\label{cor_intro}
Let $X$ be a normal complex projective variety with canonical singularities, and let $\sG$ be a weakly regular codimension one foliation on $X$. Suppose furthermore that either $K_\sG$ is Cartier and $K_\sG\equiv 0$, or $K_\sG\sim_{\mathbb{Q}}0$.
Then one of the following holds.
\begin{enumerate}
\item There exist a complex projective manifold $Y$ with $K_Y\equiv 0$, a $\mathbb{P}^1$-bundle 
$\phi\colon Z \to Y$, and a quasi-\'etale cover $f\colon Z \to X$ such that
$f^{-1}\sG$ induces a flat holomorphic connection on $\phi$.
\item There exists an abelian variety $A$ as well as a simply connected projective manifold $Y$ with $K_Y\equiv 0$, and a finite \'etale cover $f: Y \times A \to X$ such that $f^{-1}\sG$ is the pull-back of a codimenion 1 linear foliation on $A$.
\item There exist a smooth complete curve $C$, a complex projective variety $Y$ with canonical singularities and 
$K_Y \sim_\mathbb{Z}0$, as well as a quasi-\'etale cover $f \colon Y \times C \to X$ such that $f^{-1}\sG$ is induced by the projection $Y \times C \to C$.
\end{enumerate}
\end{cor}

If $\sG$ is a regular foliation on a complex manifold $X$ and $L$ is a compact leaf with finite holonomy group, then the 
holomorphic version of the local Reeb stability theorem asserts that there 
exist an invariant open analytic neighborhood $U$ of $L$ and an 
unramified Galois cover $f_1 \colon U_1 \to U$ such that the pull-back $f_1^{-1}\sG_{|U}$ of $\sG_{|U}$ to 
$U_1$ is induced by proper submersion $U_1\to S$. The proofs of our main results rely on the following global version of Reeb stability theorem.

\begin{thm}\label{thm_intro:global_reeb_stability}
Let $X$ be a normal complex projective variety with klt singularities, and let $\sG$ be an algebraically integrable foliation on $X$ with canonical singularities. Suppose that $K_\sG\equiv 0$.
Then there exist complex projective varieties $Y$ and $Z$ with klt singularities and a quasi-\'etale cover 
$f \colon Y \times Z \to X$ such that $f^{-1}\sG$ is induced by the projection $Y \times Z \to Y$.
\end{thm}

\subsection*{Outline of the proof} The main steps for the proofs of both Theorem \ref{thm_intro:main} and Theorem \ref{thm_intro:abundance} are as follows. As we will see below, our work owes a lot to the general strategy introduced in \cite{lpt}.

\medskip

Let $X$ be a normal complex projective variety with terminal singularities, and let $\sG$ be a codimension one
foliation on $X$ with canonical singularities and $K_\sG\equiv 0$. 
An analogue of the Bogomolov vanishing theorem says that $X$ has numerical dimension 
$\nu(X) \le 1$ (see Lemma \ref{lemma:bogomolov}).  

\medskip

Following \cite{lpt}, we show that either $\sG$ is closed under $p$-th powers for almost all primes $p$, or $\sG$ is given by a closed rational $1$-form (see Proposition \ref{prop:p_closed_or_not}) with value in a flat line bundle, whose zero set has codimension at least two. In the latter case, one checks that $\nu(X)\le 0$. If $\nu(X)=0$, one then proves that 
$X$ is covered by an abelian variety, and that $\sG$ is induced by a linear foliation (see Proposition \ref{prop:closed__rational_1_form_K_pseff}). If $\nu(X)=-\infty$, then one first reduces to the case where $\sG$ is defined by a closed rational $1$-form (see Proposition \ref{prop:flatness_torsion}). One then shows easily that $\sG$ is as case (2) of Theorem \ref{thm_intro:main} (see Theorem \ref{thm:closed_rational_1_form}). In particular, abundance holds in this case.
 
\medskip 
 
Suppose from now on that $\sG$ is closed under $p$-th powers for almost all primes $p$. We will prove that
$\sG$ is algebraically integrable, confirming the generalization to foliations by Ekedahl, Shepherd-Barron and Taylor of the classical Grothendieck-Katz conjecture in this special case. Theorem \ref{thm_intro:main} then follows from Theorem \ref{thm_intro:global_reeb_stability}. Moreover, abundance holds for algebraically integrable foliations with canonical singularities and numerically trivial canonical class by Proposition \ref{prop:abundance_alg_int}, as an easy consequence of a theorem of Floris (\cite[Theorem 1.2]{floris}).

Suppose that $\nu(X)=-\infty$. In this case, we show that it is enough to prove the statement under the additional assumptions that $X$ is $\mathbb{P}^1$-bundle over an abelian variety $A$ and $\sG$ is a flat connection on $X \to A$. This follows from the Minimal Model Program togehter with \cite[Theorem I]{GGK} that says that the fundamental group of the smooth locus of a projective klt variety with numerically trivial canonical class and zero augmented irregularity does not admit any finite-dimensional linear representation with infinite image. By a result of Andr\'e (\cite[Theorem 7.2.2]{andre}), we see that we can also suppose that $\sG$ and $X \to A$ are defined over a number field. The statement then follows from a theorem of Bost (\cite[Theorem 2.9]{bost}) who proved the Ekedahl-Shepherd-Barron-Taylor conjecture for flat invariant connections 
on principal bundles with linear solvable structure groups defined over number fields.

Suppose now that $\nu(X)=0$. Using Theorem \ref{thm_intro:global_reeb_stability}, we first show that we may assume that there is no positive-dimensional algebraic subvariety tangent to $\sG$ passing through a general point of $X$ (see Proposition \ref{prop:splitting_algebraic_transcendental}). Then,
running a MMP, one easily reduces to the case where $K_X$ is torsion. A weak version of the singular analogue of the Beauville-Bogomolov decomposition theorem due to Kawamata (\cite[Proposition 8.3]{kawamata85}) implies that, perhaps after passing to a quasi-\'etale cover, $\sG$ is a linear foliation on an abelian variety. \cite[Theorem 2.3]{bost} together with 
\cite[Proposition 3.6]{esbt} imply that $\sG$ is also algebraically integrable in this case.

Suppose finally that $\nu(X)=1$. The proof in this case is much more involved. We use the assumption that $\sG$ is closed under $p$-th powers for almost all primes $p$ to conclude that it is weakly regular. On the other hand, Touzet described in \cite{touzet_conpsef} the structure of codimension one foliations on complex projective manifolds with pseudo-effective conormal bundle. As a consequence, either the conormal sheaf $\sN_\sG$ satisfies 
$\kappa\big(-c_1(\sN_\sG)\big)=\nu\big(-c_1(\sN_\sG)\big)=1$ and $\sG$ is algebraically integrable, or 
$\kappa\big(-c_1(\sN_\sG)\big)=-\infty$ and $\nu\big(-c_1(\sN_\sG)\big)=1$. In the latter case, $\sG$ is induced by a codimension one tautological foliation on a quotient of a polydisc $\mathbb{D}^N$ for some $N\ge 2$
by an arithmetic irreducible lattice $\Gamma\subset\textup{PSL}(2,\mathbb{R})^N$. Using the fact that $\sG$ is weakly regular,
we then show that the image of the map $X \to \mathbb{D}^N/\Gamma$ is curve, yielding a contradiction (see Theorem \ref{thm:algebraic_integrability_nu_un_regular}).

\medskip

These steps are addressed throughout the paper, and are collected together in Section \ref{section:proofs}. 

\subsection*{Structure of the paper}Section \ref{section:notation} gathers notation, known results and global conventions that will be used throughout the paper. In section \ref{section:foliations}, we recall the definitions and basic properties of foliations. In section \ref{section:singularities}, we establish a number of properties of foliations with canonical singularities. In particular, we analyze the behaviour of foliations with canonical singularities under finite covers and $\mathbb{Q}$-factorial terminalization. We also address those with algebraic leaves. In particular, we show that abundance holds for algebraically integrable foliations with canonical singularities and numerically trivial canonical class (see Proposition \ref{prop:abundance_alg_int}).
Section \ref{section:regular_foliations} is devoted to weakly regular foliations. We first establish basic properties. We then give criteria for a foliation with trivial canonical class to be weakly regular (see Propositions \ref{prop:criterion_regularity} and \ref{prop:criterion_regularity_2}). We end this section with the local structure of rank one weakly regular foliations on surfaces with quotient singularities. Sections \ref{section:regular_foliations_algebraic_leaves} and \ref{section:holonomy} prepare for the proof of Theorem \ref{thm_intro:global_reeb_stability}. It is well-known that an algebraically integrable regular foliation on a complex projective manifold is induced by a morphism onto a normal projective variety. In section \ref{section:regular_foliations_algebraic_leaves}, we extend this result to weakly regular foliations with canonical singularities on mildly singular varieties (see Theorem \ref{thm:regular_foliation_morphism}).
Section \ref{section:reeb} is mostly taken up by the proof of Theorem \ref{thm_intro:global_reeb_stability}.
Sections \ref{section:algebraic_integrability_1} and \ref{section:algebraic_integrability_2} prepare for the proof of our main results. In particular, we confirm the Ekedahl-Shepherd-Barron-Taylor conjecture for mildly singular codimension one foliations with trivial canonical class first on projective varieties with $\nu(X)=-\infty$ in Section \ref{section:algebraic_integrability_1}, and then on those with $\nu(X)\ge 0$ in Section \ref{section:algebraic_integrability_2}. In section \ref{section:closed}, we describe codimension one foliations with numerically trivial canonical class defined by closed (twisted) rational $1$-forms.
With these preparations at hand, the proofs of Theorems \ref{thm_intro:main} and
\ref{thm_intro:abundance} and the proof of Corollary \ref{cor_intro} which we give in Section \ref{section:proofs} become reasonably short.

\subsection*{Acknowledgements} We would like to thank Frank Loray, Jorge  V.  Pereira  and
Fr\'ed\'eric Touzet for answering our questions and for extremely helpful conversations concerning the paper \cite{lpt}. 
We would also like to thank Jean-Pierre Demailly and 
Mihai P\u{a}un for useful discussions. During the preparation of the paper, we benefited from a
visit to the Department of Mathematics, Statistics, and Computer Science at the University of Illinois at Chicago.

The author was partially supported by the ALKAGE project (ERC grant Nr 670846, 2015$-$2020)
and the Foliage project (ANR grant Nr ANR-16-CE40-0008-01, 2017$-$2020).

\section{Notation, conventions, and used facts}\label{section:notation}

\subsection{Global Convention}
Throughout the paper a \textit{variety} is a reduced and irreducible scheme separated and of finite type over a field.

Given a scheme $X$, we denote by $X_\textup{reg}$ its smooth locus.

Suppose that $k=\mathbb{C}$. We will use the notions of terminal, canonical, klt, and lc singularities for pairs
without further explanation or comment and simply refer to \cite[Section 2.3]{kollar_mori} for a discussion and for their precise definitions.

\subsection{$\mathbb{Q}$-factorializations and $\mathbb{Q}$-factorial terminalizations}\label{say:q_factorialization}

\begin{defn}
Let $X$ be a normal complex quasi-projective variety with klt singularities.
A $\mathbb{Q}$-\textit{factorialization} is a small birational projective morphism $\beta\colon Z \to  X$,
where $Z$ is $\mathbb{Q}$-factorial with klt singularities. 
\end{defn}

\begin{fact}\label{fact:existence_factorialization}
The existence of $\mathbb{Q}$-factorializations is established in \cite[Corollary 1.37]{kollar_kovacs_singularities}.
Note that we must have $K_{Z}\sim_\mathbb{Q}\beta^*K_X$.
\end{fact}

\begin{defn}
Let $X$ be a normal complex quasi-projective variety with canonical singularities. A \emph{$\mathbb{Q}$-factorial terminalization} of $X$ is a birational crepant projective morphism 
$\beta \colon Z \to X$ where $Z$ is $\mathbb{Q}$-factorial with terminal singularities.  
\end{defn}

\begin{fact}\label{fact:existence_terminalization}
The existence of $\mathbb{Q}$-factorial terminalizations is established in \cite[Corollary 1.4.3]{bchm}.
\end{fact}

\subsection{Projective space bundle}
If $\sE$ is a locally free sheaf of finite rank on a variety $X$, 
we denote by $\mathbb{P}(\sE)$ the variety $\textup{Proj}_X\big(\textup{S}^\bullet\sE\big)$,
and by $\sO_{\mathbb{P}(\sE)}(1)$ its tautological line bundle. 

\subsection{Stability}
The word \textit{stable} will always mean \textit{slope-stable with respect to a
given movable curve class}. Ditto for \textit{semistable} and \textit{polystable}.

\subsection{Reflexive hull}
Given a normal variety $X$, $m\in \mathbb{N}$, and coherent sheaves $\sE$ and $\sG$ on $X$, write
$\sE^{[m]}:=(\sE^{\otimes m})^{**}$, $\textup{S}^{[m]}\sE:=(\textup{S}^m\sE)^{**}$, $\det\sE:=(\Lambda^{\textup{rank} \,\sE}\sE)^{**}$, and 
$\sE \boxtimes \sG := (\sE \otimes\sG)^{**}.$
Given any morphism $f \colon Y \to X$, write 
$f^{[*]}\sE:=(f^*\sE)^{**}.$

\subsection{Reflexive K\"{a}hler differentials and pull-back morphisms}\label{subsection:pull-back_morphims} 
Given a normal variety $X$, we denote the sheaf of K\"{a}hler differentials by
$\Omega^1_X$. 
If $0 \le p \le \dim X$ is any integer, write
$\Omega_X^{[p]}:=(\Omega_X^p)^{**}$.
The tangent sheaf $(\Omega_X^1)^*$ will be denoted by $T_X$.

If $D$ is a reduced effective divisor on $X$ we denote by 
$(X,D)_{\textup{reg}}$ the open set where $(X,D)$ is log smooth. We write $\Omega_X^{[p]}(\textup{log}\, D)$ for
the reflexive sheaf on $X$ whose restriction to $U:=(X,D)_{\textup{reg}}$ is the sheaf of logarithmic differential forms 
$\Omega_U^{p}(\textup{log}\, D_{|U})$. We will refer to it as the sheaf of \textit{reflexive logarithmic $p$-forms}.
Suppose that $X$ is smooth and let $t$ be a defining equation for $D$ on some open set $X^\circ$.
Let $\alpha$ be a rational $p$-form on $X$. Then $\alpha$ is a reflexive logarithmic $p$-form on $X^\circ$
if and only if $t\alpha$ and $td\alpha$ are regular on $X^\circ$ (see \cite{saito}).

Suppose that $k=\mathbb{C}$.

If $f \colon Y \to  X$ is any morphism between varieties, we denote the
standard pull-back maps of K\"{a}hler differentials by
$$df \colon f^*\Omega_X^p \to \Omega_Y^p\quad\textup{and}\quad df \colon 
H^0(X,\Omega_X^p) \to H^0(Y,\Omega_Y^p).$$ 
Reflexive differential forms do not generally 
satisfy the same universal properties as K\"{a}hler differentials.
However, it has been shown in \cite{greb_kebekus_kovacs_peternell10} and \cite{kebekus_pull_back}
that many of the functorial properties do
hold if we restrict ourselves to klt spaces.

\begin{thm}[{\cite[Theorem 1.3]{kebekus_pull_back}}]
Let $f \colon Y \to X$ be a morphism of normal varieties. Suppose that $X$ is klt. Then there exist pull-back morphisms 
$$d_{\textup{refl}}f\colon f^*\Omega_Y^{[p]} \to \Omega_X^{[p]}\quad\textup{and}\quad d_{\textup{refl}}f\colon 
H^0\big(X,\Omega_X^{[p]}\big) \to H^0\big(Y,\Omega_Y^{[p]}\big)$$
that agree with the usual pull-back morphisms of K\"ahler differentials wherever this makes sense.
\end{thm}

More precisely, the pull-back morphism for reflexive forms satisfies the following universal property (see \cite[Proposition 6.1]{kebekus_pull_back}). Let $f \colon Y \to X$ be any morphism of klt spaces. Given a commutative diagram

\begin{center}
\begin{tikzcd}
&&&& Z\ar[d, "{\beta, \textup{ resolution of singularities}}"]\\
V \ar[rrrru, bend left=15, "\alpha"]\ar[rr, "{g,\textup{ dominant}}"'] && Y_{\textup{reg}} \ar[rr, "f_{|Y_{\textup{reg}}}"'] && X, \\ 
\end{tikzcd}
\end{center}

\noindent where $V$ is smooth, we have 
$$dg \circ d_{\textup{refl}}\big(f_{|Y_{\textup{reg}}}\big) = d\alpha \circ d_{\textup{refl}}\beta.$$

\subsection{Pull-back of Weil divisors}\label{definition:pull-back}
Let $\psi\colon X \to Y$ be a dominant equidimensional morphism of normal varieties, and let $D$ be a Weil $\mathbb{Q}$-divisor on $Y$. The \textit{pull-back $\psi^*D$ of $D$} is defined as follows. We define 
$\psi^*D$ to be the unique $\mathbb{Q}$-divisor on $X$ whose restriction to 
$\psi^{-1}(Y_{\textup{reg}})$ is $(\psi_{|\psi^{-1}(Y_{\textup{reg}})})^*(D_{|Y_{\textup{reg}}})$. This construction agrees with the usual pull-back if $D$ is $\mathbb{Q}$-Cartier.

We will need the following easy observation.

\begin{lemma}\label{lemma:pull-back_Q_Cartier}
Let $\psi\colon X \to Y$ be a projective and dominant morphism of normal varieties,
and let $D$ be a Weil divisor on $Y$. Suppose in addition that $\psi$ is equidimensional. If $\psi^*D$ is $\mathbb{Q}$-Cartier, then so is $D$.
\end{lemma}

\begin{proof}
The statement is local on $Y$, hence we may shrink $Y$ and assume that $Y$ is affine.
By \cite[Theorem 6.3]{gabber_liu_lorenzini}, there exists a subvariety $Z \subseteq X$ such that 
$\psi_{|Z}\colon Z \to Y$ is finite and surjective. 
Replacing $X$ by the normalization of $Z$, we may assume without loss of generality that $\psi$ is a finite morphism. 
Recall that $\psi^*D$ is $\mathbb{Q}$-Cartier by assumption. Shrinking $Y$ again, if necessary, we may also assume that there exist a positive integer $m$ and a rational function $t$ on $X$ such that $\textup{div}\,t=mf^*D$. One then readily checks that 
$\textup{div}\,\textup{N}_{X/Y}(t)=m (\deg\,f) D$, where $\textup{N}_{X/Y}(t)$ denotes the norm of $t$. This shows that $D$ is $\mathbb{Q}$-Cartier (see 
\cite[Lemma 5.16]{kollar_mori} for a somewhat related result).
\end{proof}

\subsection{Numerical dimension}
Let $D$ be a $\mathbb{Q}$-divisor on a complex projective manifold $X$, and let $A$ be an ample divisor on $X$.
Following Nakayama (see \cite[Definition V.2.5]{nakayama04}), we
set 
$$
\sigma(D,A):=\max \left\{ k \in \mathbb{Z}_{\geq 0} \left|\; \varlimsup_{m \to \infty} \frac{h^0\big(X, \mathcal{O}_X(\lfloor mD \rfloor + A)\big)}{m^k}>0 \right.\right\}
$$
if $h^0(X, \mathcal{O}_X(\lfloor mD \rfloor+ A)) \neq 0$ for some $m \ge 1$, and 
$\sigma(D,A):=-\infty$ otherwise. The numerical dimension of $D$ is defined as
$$
\nu(D):=\max \bigl\{ \sigma(D,A) \left|\; A \textup{ ample divisor on X} \right.\bigr\}.
$$
Let now $D$ be a $\mathbb{Q}$-Cartier $\mathbb{Q}$-divisor on a normal projective variety $X$, and let $\beta \colon Z \to X$ be a resolution of singularities. The \emph{numerical dimension} of $D$ is defined as 
$\nu(D):=\nu(\beta^*D)$. Then $\nu(D)$ is independent of the resolution, and 
it depends only on the numerical class of $D$ by \cite[Proposition V.2.7]{nakayama04}. 
We refer to  \cite{nakayama04} for more detailed properties.

\begin{rem}\label{rem:pseff_versus_numerical dimension}
Recall that a Weil $\mathbb{Q}$-divisor $D$ on a normal projective variety $X$ is said to be \textit{pseudo-effective} if, for any big $\mathbb{Q}$-divisor $B$ on $X$ and any rational number 
$\varepsilon > 0$, there exists an effective $\mathbb{Q}$-divisor $E$ on $X$ such that
$D + \varepsilon B \sim_\mathbb{Q} E$. If $D$ is $\mathbb{Q}$-Cartier, then $D$ is 
pseudo-effective if and only if $\nu(D)\ge 0$ by definition.
\end{rem}

\subsection{Quasi-\'etale morphisms} We will need the following definition.

\begin{defn}
A \textit{cover} is a finite and surjective morphism of normal varieties.

A morphism $f\colon Y \to X$ between normal varieties is called a \textit{quasi-\'etale morphism}
if $f$ is finite and \'etale in codimension one. 
\end{defn}

\begin{rem}\label{rem:quasi_etale_smooth_locus}
Let $f\colon Y \to X$ be a quasi-\'etale cover. By the Nagata-Zariski purity theorem, $f$ branches only on the singular set of $X$. In particular, we have $f^{-1}(X_{\textup{reg}}) \subseteq Y_{\textup{reg}}$. 
\end{rem}

The following elementary fact will be use throughout the paper.

\begin{fact}\label{fact:quasi_etale_cover_and_singularities}
Let $f\colon Y \to X$ be a quasi-\'etale cover between normal complex varieties. If $K_X$ is Cartier (resp. $\mathbb{Q}$-Cartier), then 
$K_Y\sim_\mathbb{Z}f^*K_X$ is Cartier (resp. $\mathbb{Q}$-Cartier) as well. 
If $X$ is terminal (resp. canonical, klt) then so is $Y$ by \cite[Proposition 3.16]{kollar97}.
\end{fact}

\subsection{Augmented irregularity}
The irregularity of normal complex projective varieties is
generally not invariant under quasi-\'etale maps. The notion of augmented irregularity addresses this issue (see \cite[Definition 3.1]{gkp_bo_bo}).
\begin{defn}\label{defn:augmented_irregularity}
Let $X$ be a normal complex projective variety. We denote the irregularity of $X$ by
$q(X) := h^1( X,\, \sO_X )$ and define the \emph{augmented irregularity} as
$$
\wt q(X) := \max \bigl\{q(Y) \mid Y \to X \ \text{a
quasi-\'etale cover} \bigr\} \in \mathbb N \cup \{ \infty\}.
$$
\end{defn}

\begin{rem}\label{rem:irregularity_bir_invariant}
By a result of Elkik (\cite{elkik}), canonical singularities are rational. It follows that the irregularity is a birational invariant of complex projective varieties with canonical singularities. 
\end{rem}

\begin{rem}\label{remark:augmented_irregularity_finite}
The augmented irregularity of
canonical varieties with numerically trivial canonical class is
finite. This follows easily from \cite[Proposition 8.3]{kawamata85}. 
\end{rem}

The following result often reduces the study of varieties with
trivial canonical class to those with $\wt q(X) = 0$ (see also \cite[Proposition 8.3]{kawamata85}).

\begin{thm}[{\cite[Corollary 3.6]{gkp_bo_bo}}]\label{thm:kawamata_abelian_factor}
Let $X$ be a normal complex projective variety with canonical
singularities. Assume that $K_X$ is numerically trivial. Then there exists
an abelian variety $A$, as well as a normal projective variety $Y$ with $K_{Y}\sim_\mathbb{Z} 0$ and $\wt q (Y)=0$, and a quasi-\'etale cover
$ A \times Y \to X$.
\end{thm}

\subsection{Automorphism group} Let $X$ be a complex projective variety and let $\textup{Aut}^\circ(X)$ be the 
neutral component of the automorphism group $\textup{Aut}(X)$ of $X$; $\textup{Aut}^\circ(X)$ is an algebraic group of finite type with $\dim \textup{Aut}^\circ(X) = h^0(X,T_X)$. 
By a theorem of Chevalley, $\textup{Aut}^\circ(X)$ has a largest connected affine normal subgroup
$G$. Further, the
quotient group
$\textup{Aut}^\circ(X)/G$ is an abelian variety.
By \cite[Theorem 14.1]{uenoLN439}, if $G$ is non-trivial, then $X$ is uniruled. 
In particular, if $X$ is canonical and $K_X \equiv 0$, then $\textup{Aut}^\circ(X)$ is an abelian variety.
Lemma \ref{lemma:automorphism_group_zero_can_class} below extends this observation to klt spaces.

\begin{lemma}\label{lemma:automorphism_group_zero_can_class}
Let $X$ be a complex projective variety with klt singularities. Assume that $K_X$ is numerically trivial. Then 
the neutral component $\textup{Aut}^\circ(X)$ of the automorphism group $\textup{Aut}(X)$ of $X$ is an abelian variety.
\end{lemma}

\begin{proof}
By \cite[Corollary V 4.9]{nakayama04}, $K_X$ is torsion. Let $m$ its Cartier index, and let $f \colon Y \to X$ be the index one canonical cover, which is quasi-\'etale (\cite[Definition 2.52]{kollar_mori}). 
By Fact \ref{fact:quasi_etale_cover_and_singularities}, the variety $Y$ is then klt. Moreover, we have $K_Y\sim_\mathbb{Z} 0$ by construction, and therefore $Y$ has canonical singularities. On the other hand,
$$Y \cong \textup{Spec}_X\bigoplus_{i=0}^{m-1}\sO_X(-iK_X),$$ 
and hence there is an injective morphism of algebraic groups
$\textup{Aut}^\circ(X) \subseteq \textup{Aut}^\circ(Y)$. The lemma then follows from \cite[Theorem 14.1]{uenoLN439} applied to $Y$ as explained above. 
\end{proof}

\begin{rem}\label{remark:augmented_irregularity_vector_fields}
Let $X$ be a complex projective variety with klt singularities. Suppose that $K_X \sim_\mathbb{Z} 0$, and set $n:=\dim X$. Then
$$
\begin{array}{ccll}
\dim \textup{Aut}^\circ(X) & = & h^0(X,T_{X})\\
& = & h^0\big(X,\Omega_{X}^{[n -1]}\big) & \text{ since $\Omega_{X}^{[n -1]}\cong T_{X}$}\\
& = & h^{n-1}(X,\sO_{X})  
& \text{ by Hodge symmetry for klt spaces (see \cite[Proposition 6.9]{gkp_bo_bo})} \\
& =  & q(X) & \text{ by Serre duality using the assumption that $K_X \sim_\mathbb{Z} 0$}.\\
\end{array}
$$
\end{rem}

\section{Foliations}\label{section:foliations}

In this section, we have gathered a number of results and
facts concerning foliations which will later be used in the proofs.

\begin{defn}

A \emph{foliation} on  a normal variety $X$ over a field $k$ is a coherent subsheaf $\sG\subseteq T_X$ such that
\begin{enumerate}
\item $\sG$ is closed under the Lie bracket, and
\item $\sG$ is saturated in $T_X$. In other words, the quotient $T_X/\sG$ is torsion-free.
\end{enumerate}

The \emph{rank} $r$ of $\sG$ is the generic rank of $\sG$.
The \emph{codimension} of $\sG$ is defined as $q:=\dim X-r$. 

\medskip

The \textit{canonical class} $K_{\sG}$ of $\sG$ is any Weil divisor on $X$ such that  $\sO_X(-K_{\sG})\cong \det\sG$. 

\medskip

Suppose that $k=\mathbb{C}$. 

Let $X^\circ \subset X_{\textup{reg}}$ be the open set where $\sG_{|X_{\textup{reg}}}$ is a subbundle of $T_{X_{\textup{reg}}}$. 
A \emph{leaf} of $\sG$ is a maximal connected and immersed holomorphic submanifold $L \subset X^\circ$ such that
$T_L=\sG_{|L}$. A leaf is called \emph{algebraic} if it is open in its Zariski closure.

The foliation $\sG$ is said to be \emph{algebraically integrable} if its leaves are algebraic.

\end{defn}

We will use the following notation.

\begin{notation}
Let $\psi \colon X \to Y$ be a dominant equidimensional morphism of normal varieties. 

Write $K_{X/Y}:=K_X-\psi^*K_Y$. We will refer to it as the \emph{relative canonical divisor of $X$ over $Y$}.

Set $$R(\psi)=\sum_{D} \big(\psi^*D-{(\psi^*D)}_{\textup{red}}\big)$$
where $D$ runs through all prime divisors on $Y$. We will refer to it
as the \emph{ramification divisor of $\psi$}.
\end{notation}

\begin{exmp}\label{example:canonical_class_foliation}
Let $\psi \colon X \to Y$ be a dominant equidimensional morphism of normal varieties, and let $\sG$ be the foliation on $X$ induced by $\psi$. A straightforward computation shows that
$$K_\sG\sim_\mathbb{Z}K_{X/Y}-R(\psi).$$
\end{exmp}

\subsection{Foliations defined by $q$-forms} \label{q-forms}
Let $\sG$ be a codimension $q$ foliation on an $n$-dimensional normal variety $X$.
The \emph{normal sheaf} of $\sG$ is $\sN:=(T_X/\sG)^{**}$.
The $q$-th wedge product of the inclusion
$\sN^*\into \Omega^{[1]}_X$ gives rise to a non-zero global section 
 $\omega\in H^0(X,\Omega^{q}_X\boxtimes \det\sN)$
 whose zero locus has codimension at least two in $X$. 
Moreover, $\omega$ is \emph{locally decomposable} and \emph{integrable}.
To say that $\omega$ is locally decomposable means that, 
in a neighborhood of a general point of $X$, $\omega$ decomposes as the wedge product of $q$ local $1$-forms 
$\omega=\omega_1\wedge\cdots\wedge\omega_q$.
To say that it is integrable means that for this local decomposition one has 
$d\omega_i\wedge \omega=0$ for every  $i\in\{1,\ldots,q\}$. 
The integrability condition for $\omega$ is equivalent to the condition that $\sG$ 
is closed under the Lie bracket.

Conversely, let $\sL$ be a reflexive sheaf of rank $1$ on $X$, and let
$\omega\in H^0(X,\Omega^{q}_X\boxtimes \sL)$ be a global section
whose zero locus has codimension at least two in $X$.
Suppose that $\omega$  is locally decomposable and integrable.
Then  the kernel
of the morphism $T_X \to \Omega^{q-1}_X\boxtimes \sL$ given by the contraction with $\omega$
defines 
a foliation of codimension $q$ on $X$. 
These constructions are inverse of each other.

\subsection{Foliations described as pull-backs} \label{pullback_foliations}

Let $X$ and $Y$ be normal varieties, and let $\varphi\colon X\map Y$ be a dominant separable rational map that restricts to a morphism $\varphi^\circ\colon X^\circ\to Y^\circ$,
where $X^\circ\subset X$ and  $Y^\circ\subset Y$ are smooth open subsets.

Let $\sG$ be a codimension $q$ foliation on $Y$. Suppose that the restriction $\sG^\circ$ of $\sG$ to $Y^\circ$ is
defined by a twisted $q$-form
$\omega_{Y^\circ}\in H^0(Y^\circ,\Omega^{q}_{Y^\circ}\otimes \det\sN_{\sG^\circ}).$
Then $\omega_{Y^\circ}$ induces a non-zero twisted $q$-form 
$$\omega_{X^\circ}:= d\phi^\circ(\omega_{Y^\circ})\in 
H^0\big(X^\circ,\Omega^{q}_{X^\circ}\otimes (\varphi^\circ)^*({\det\sN_\sG}_{|Y^\circ})\big)$$ which defines a codimension $q$ foliation $\sE^\circ$ on $X^\circ$. 
\emph{The pull-back $\varphi^{-1}\sG$ of $\sG$ via $\varphi$} is the foliation on $X$ 
whose restriction to $X^\circ$ is $\sE^\circ$.

\medskip

The following observation is rather standard.
We include a proof here for the reader's convenience.

\begin{lemma}\label{lemma:pull_back_fol_and_finite_cover}
Let $f \colon Y\to X$ be a quasi-finite dominant morphism of normal complex varieties, and let $\sG$ be a foliation on $X$ of rank $1 \le r \le \dim X -1$. Let also $B$ be a codimension one irreducible component of the branch locus of $f$.
\begin{enumerate}
\item Suppose that $B$ is $\sG$-invariant. Let $D$ be any irreducible component of $f^{-1}(B)$, and let $m$ denotes the ramification index of $f$ along $D$. Then the natural map $\det f^{[*]}\sN_\sG^* \to \det \sN_{f^{-1}\sG}^*$
vanishes at order $m-1$ along $D$.
\item If $B$ is not $\sG$-invariant, then the map $\det f^{[*]}\sN_\sG^* \to \det \sN_{f^{-1}\sG}^*$
is an isomorphism at a general point in $f^{-1}(B)$.
\end{enumerate}
\end{lemma}

\begin{proof}
Set $q:=n-r$.
Replacing $X$ by a dense open set $X^\circ$ with complement of codimension at least two in $X$ and $Y$ by $f^{-1}(X^\circ)$, we may assume without loss of generality that $X$ and $Y$ are smooth, that the branch locus of $f$ is a smooth hypersurface $B \subset X$, 
and that $\sG$ is a regular foliation. We may also assume that the ramification divisor $D:=f^{-1}(B)$ is smooth.
Given a point $x \in B$, there are analytic coordinates $(x_1,\ldots,x_n)$ centered at $x$ such that $B$ is defined by $x_1=0$
and such that $f$ is given by $(y_1,\ldots,y_n) \mapsto (y_1^m,y_2\ldots,y_n)$ for some integer $m \ge 1$ and some local analytic coordinates $(y_1,\ldots,y_n)$ centered at a point $y$ in $f^{-1}(x)$.

Suppose first that $B$ is $\sG$-invariant. Then $\sG$ is given by a local $q$-form $dx_1\wedge \alpha_1+x_1\alpha_2$, where 
$\alpha_1 \in \wedge^{q-1}\big(\sO_{X,x}dx_2\oplus\cdots\oplus\sO_{X,x}dx_n\big)$ is nowhere vanishing
and $\alpha_2 \in \wedge^{q}\big(\sO_{X,x}dx_1\oplus\cdots\oplus\sO_{X,x}dx_n\big)$.
It follows that $$df(dx_1\wedge \alpha_1+x_1\alpha_2)=x_1^{m-1}\big(mdx_1\wedge df(\alpha_1)+x_1df(\alpha_2)\big)$$ vanishes at order $m-1$ along $D$.

If $B$ is not $\sG$-invariant, then $\sG$ is given by a nowhere vanishing $q$-form $\alpha_3 \in \wedge^{q}\big(\sO_{X,x}dx_2\oplus\cdots\oplus\sO_{X,x}dx_n\big)$.
But $df(\alpha_3)$ is obviously a nowhere vanishing $q$-form. This completes the proof of the lemma.
\end{proof}

\subsection{Projectable foliations}\label{projectable_foliations}
Let $\pi\colon X \to Y$ be a dominant separable morphism between normal varieties, and let $\sG$ be a foliation on $X$.
We say that $\sG$ is \textit{projectable under} $\pi$ if there exists a saturated distribution $\sH \subseteq T_Y$ such that
the natural map $$\sG_{|\pi^{-1}(Y_{\textup{reg}})} \to \big({\pi_{|\pi^{-1}(Y_{\textup{reg}})}}^{*}\sH_{|\pi^{-1}(Y_{\textup{reg}})}\big)^{\textup{sat}}$$ is an isomorphism, where $\big({\pi_{|\pi^{-1}(Y_{\textup{reg}})}}^{*}\sH_{|\pi^{-1}(Y_{\textup{reg}})}\big)^{\textup{sat}}$
denotes the saturation of 
${\pi_{|\pi^{-1}(Y_{\textup{reg}})}}^{*}\sH_{|\pi^{-1}(Y_{\textup{reg}})}$
in $$(\pi^{*}T_Y)_{|\pi^{-1}(Y_{\textup{reg}})} \cong 
{\pi_{|\pi^{-1}(Y_{\textup{reg}})}}^{*}T_{Y_{\textup{reg}}}.$$
One then readily checks that $\sH$ is a foliation on $Y$. We refer the reader to \cite[Section 2.7]{druel_bbcd2}
for a more detailed explanation. 

\subsection{Invariant subvarieties} Let $X$ be a normal variety, 
let $Y \subseteq X$ be a closed subvariety, and let
$\partial$ be a derivation on $X$. 
Say that $Y$ is  
\textit{invariant under} $\partial$ if $\partial(\sI_Y)\subseteq \sI_Y$.

Let $\sG \subseteq T_X$ be a foliation on $X$.  Say that $Y$ is  
\textit{invariant under} $\sG$ if for any local section $\partial$ of $\sG$ over some open subset $U$ of $X$, 
$Y \cap U$ is invariant under $\partial$. 
To prove that $Y$ is invariant under $\sG$ it is enough to show that  $Y \cap U$ of $Y$ is invariant under $\sG_{|U}$ for some open set $U\subseteq X$ such that $Y\cap U$ is dense in $Y$.
If $X$ and $Y$ are smooth and $\sG \subseteq T_X$ is a subbundle, then $Y$ is invariant under $\sG$ if and only if $\sG_{|Y} \subseteq T_Y \subseteq {T_X}_{|Y}$. 

By \cite[Theorem 5]{seidenberg67}, the singular locus of a normal complex variety is invariant under any derivation. 
Other examples of invariant subsets are provided by the following easy result.

\begin{lemma}\label{lemma:singular_set_invariant}
Let $\sG \subseteq T_X$ be a foliation of rank $r\ge 1$ on a complex manifold $X$.
Then any component of the singular locus of $\sG$ is invariant under $\sG$.
\end{lemma}

\begin{proof}
We argue by induction on $r\ge 1$.

If $r=1$, then the singular set of $\sG$ is obviously invariant under $\sG$. 

Suppose from now on that $r \ge 2$.
The statement is local on $X$, hence we may shrink $X$ and assume that $X$ is affine.
Let $S$ be a component of the singular locus of $\sG$, and let also $x \in S$ be a general point.
We may also assume that 
$\partial(\sI_{\{x\}})\not\subseteq \sI_{\{x\}}$
for some $\partial \in H^0(X,\sG) \subseteq H^0(X,T_X)$. But then $\sO_X\partial\subset T_X$ is a regular foliation at $x$ of rank one. By a theorem of Frobenius, there exist an open neighborhood $U$ of $x$ with respect to the analytic topology and an isomorphism of analytic varieties
$U\cong \mathbb{D} \times W$ such that $\sO_X\partial\subset T_X$ is induced on $U$ by the projection $\phi\colon U\cong \mathbb{D} \times W \to W$, where $\mathbb{D}$ is the complex open unit disk
and $W$ is a germ of smooth analytic variety. In particular, there is a foliation $\sH$ on $W$ such that 
$\sG_{|U} = \phi^{-1}\sH$, and hence $S \cap U = \phi^{-1}(T)$ for some component $T$ of the singular set of $\sH$.
By induction, we conclude that $T$ is invariant under $\sH$. This immediately implies that $S$ is invariant under $\sG$, completing the proof of the lemma.
\end{proof}

\subsection{The family of leaves}\label{family_leaves} We refer the reader to \cite[Remark 3.12]{codim_1_del_pezzo_fols} for a more detailed explanation.
Let $X$ be a normal complex projective variety, and let $\sG$ be an algebraically integrable foliation on $X$.
There is a unique normal complex projective variety $Y$ contained in the normalization 
of the Chow variety of $X$ 
whose general point parametrizes the closure of a general leaf of $\sG$
(viewed as a reduced and irreducible cycle in $X$).
Let $Z \to Y\times X$ denotes the normalization of the universal cycle.
It comes with morphisms

\begin{center}
\begin{tikzcd}
Z \ar[r, "\beta"]\ar[d, "\psi"] & X \\
 Y &
\end{tikzcd}
\end{center}

\noindent where $\beta\colon Z\to X$ is birational and, for a general point $y\in Y$, 
$\beta\big(\psi^{-1}(y)\big) \subseteq X$ is the closure of a leaf of $\sG$.
The morphism $Z \to Y$ is called the \emph{family of leaves} and $Y$ is called
the \emph{space of leaves} of $\sG$.

Suppose furthermore that $K_\sG$ is $\bQ$-Cartier. 
There is a canonically defined effective Weil $\bQ$-divisor $B$ on $Z$ such that 
$$K_{\beta^{-1}\sG}+B\sim_\mathbb{Z}K_{Z/Y}-R(\psi)+B \sim_\mathbb{Q} \beta^* K_\sG,$$
where $R(\psi)$ denotes the ramification divisor of $\psi$. Note that $B$ is $\beta$-exceptional since $\beta_*K_{\beta^{-1}\sG}\sim_\mathbb{Z}K_\sG$.

The following property holds in addition. 
Let $m$ be a positive integer and let $X ^\circ \subseteq X$ be a dense open set such that 
$\sO_{X^\circ}\big(m{K_\sG}_{|X^\circ}\big)\cong \sO_{X^\circ}$.
Let $f^\circ \colon X_1^\circ \to X^\circ$ be the associated cyclic cover, which is quasi-\'etale (see \cite[Definition 2.52]{kollar_mori}). Finally, let $Z_1^\circ$ be the normalization of the product $Z^\circ \times_{X^\circ} X_1^\circ$, where 
$Z^\circ:=\beta^{-1}(X^\circ)$. 
If $\beta_1^\circ\colon Z_1^\circ \to X_1^\circ$ 
and $g^\circ\colon Z_1^\circ \to Z^\circ$
denote the natural morphisms, then there exists an effective $\beta_1^\circ$-exceptional divisor $B_1^\circ$ on $Z_1^\circ$ such that $$K_{(\beta_1^\circ)^{-1}(f^\circ)^{-1}(\sG_{|X^\circ})}+B_1^\circ\sim_\mathbb{Z} (\beta_1^\circ)^* K_{(f^\circ)^{-1}(\sG_{|X^\circ})}.$$ 
Let $C_1^\circ$ denotes the non $(\beta_1^\circ)^{-1}(f^\circ)^{-1}(\sG_{|X^\circ})$-invariant part of the ramification divisor $R(g^\circ)$ of $g^\circ$. Using Lemma \ref{lemma:pull_back_fol_and_finite_cover} applied to $g^\circ$ and $(\beta^{-1}\sG)_{|Z^\circ}$, we easily obtain 
$(g^\circ)^*\big(B_{|Z^\circ}\big) \sim_\mathbb{Q} B_1^\circ + C_1^\circ$. By the negativity lemma, 
we get 
$$(g^\circ)^*\big(B_{|Z^\circ}\big) = B_1^\circ + C_1^\circ.$$

\subsection{Bertini-type results} 
The present subsection is devoted to the following auxiliary result.

\begin{prop}\label{prop:bertini}
Let $X$ be a normal complex projective variety with $\dim X \ge 2$, and let $\sG\subseteq T_X$ be a foliation of rank $2 \le r \le \dim X -1$. If
$H \in |\sL|$ is a general 
member of a basepoint-free linear system corresponding to $\sL\in\textup{Pic}(X)$, then $\sG_{|H}\subseteq {T_X}_{|H}$ and $\sG_H:={\sG}_{|H}\cap T_H$ is a foliation on $H$. In addition, the following holds.
\begin{enumerate}
\item Suppose that $H$ is transverse to $\sG$ at a general point in $X$. Then $\sG_H$ has rank $r-1$, and there exists an effective divisor $B$ on $H$ such that 
$$K_{\sG_H}\sim_\mathbb{Z}(K_\sG+H)_{|H}-B.$$ 
Moreover, if $B_1$ is a prime divisor on $H$, then $B_1 \subseteq \textup{Supp}\,B$ if and only if $\sG$ is tangent to $H$ at a general point of $B_1$.
\item Suppose that there is a dense open set $X^\circ \subseteq X_{\textup{reg}}$ with complement of codimension at least two satisfying the following property. For each $x \in X^\circ$, $\sG$ is regular at $x$, and
there exist $H_1 \in |\sL|$ and $H_2 \in |\sL|$ passing through $x$ with $H_1 \neq H_2$ such that any 
member of $\langle H_1,H_2\rangle$ 
is transverse to $\sG$ at $x$. Then $$K_{\sG_H}\sim_\mathbb{Z}(K_\sG+H)_{|H}.$$
\item Suppose finally that $\sL$ is very ample. Then we have $$K_{\sG_H}\sim_\mathbb{Z}(K_\sG+H)_{|H}.$$
\end{enumerate}
\end{prop}

\begin{proof}
By \cite[Th\'eor\`eme 12.2.1]{ega28}, $\sG_{|H}\subseteq {T_X}_{|H}$ and $\sG_H={\sG}_{|H}\cap T_H$ are saturated in $T_H$. 
On the other hand, $\sG_H$ is obviously closed under the Lie bracket.
This proves that 
$\sG_H$ is a foliation.

Set $\sN:=T_X/\sG$ and $q=n-r$, and let $\omega\in H^0(X,\Omega^{q}_X\boxtimes \det\sN)$
be a non-zero twisted $q$-form defining $\sG$. Let also $\omega_H\in H^0(H,\Omega^{q}_H\boxtimes \det\sN_{|H})$
be the induced $q$-form, and let $B$ be the maximal effective divisor on $H$ such that 
$\omega_H\in H^0\big(H,\Omega^{q}_H\boxtimes \det\sN_{|H}\boxtimes\sO_H(-B)\big)$.
Then $\omega_H$ is a twisted $q$-form defining $\sG_H$ since the sheaf $\sG_H$ is reflexive by \cite[Remark 2.3]{fano_fols}.
A straightforward computation then shows that
$K_{\sG_H}\sim_\mathbb{Z}(K_\sG+H)_{|H}-B$ 
using the adjunction formula $K_H\sim_\mathbb{Z} (K_X+H)_{|H}$. This shows Item (1).

Now, Item (2) follows from Item (1) and an easy dimension count (see proof of \cite[Lemma 2.9]{fol_index}), while Item (3) is an immediate consequence of Item (2). 
\end{proof}

\section{Singularities of foliations}\label{section:singularities}

There are several notions of singularities for foliations. The notion of \textit{reduced}
foliations has been used in the birational classification of foliations by curves on surfaces (see [Bru04]). More recently, notions of singularities coming  from  the  minimal  model  program  have  shown  to be very  useful  when  studying  birational  geometry  of foliations. We refer the reader to \cite[Section I]{mcquillan08} for an in-depth discussion. 
Here we only recall the notion of
canonical foliation following McQuillan (see \cite[Definition I.1.2]{mcquillan08}).

\begin{defn}\label{definition:canonical_singularities}
Let $\sG$ be a foliation on a normal complex variety $X$. Suppose that $K_\sG$ is $\mathbb{Q}$-Cartier.
Let $\beta\colon Z \to X$ be a projective birational morphism. 
Then there are uniquely defined rational numbers $a(E,X,\sG)$ such that
$$
K_{\beta^{-1}{\sG}}\sim_\mathbb{Q}\beta^*K_{\sG}+ \sum_E a(E,X,\sG)E,
$$
where $E$ runs through all exceptional prime divisors for $\beta$.
The rational numbers $a(E,X,\sG)$ do not depend on the birational morphism $\beta$,
but only on the valuations associated to the $E$. 
We say that $\sG$ is \textit{canonical} if, for all $E$ exceptional over $X$, $a(E,X,\sG) \ge 0$.
\end{defn}

\subsection{Elementary properties}
In this subsection we analyze the behaviour of canonical singularities with respect to birational maps, finite covers, and projections.

\begin{lemma}\label{lemma:singularities_birational_morphism}
Let $\beta\colon Z \to X$ be a birational projective morphism of normal complex varieties, and let $\sG$ be a foliation on $X$.
Suppose that $K_\sG$ is $\mathbb{Q}$-Cartier.
\begin{enumerate}
\item Suppose that $K_{\beta^{-1}\sG}\sim_\mathbb{Q} \beta^*K_\sG+E$ for some effective $\beta$-exceptional 
$\mathbb{Q}$-divisor on $Z$. If 
$\beta^{-1}\sG$ is canonical, then so is $\sG$.
\item If $K_{\beta^{-1}\sG}\sim_\mathbb{Q} \beta^*K_\sG$, then $\sG$ is canonical if and only if so is $\beta^{-1}\sG$.
\end{enumerate}
\end{lemma}

\begin{proof}We can write $K_{\beta^{-1}\sG}\sim_\mathbb{Q} \beta^*K_\sG+E$ for some $\beta$-exceptional 
$\mathbb{Q}$-divisor on $Z$.

Let $Z_1$ be a normal complex variety, and let $\beta_1\colon Z_1 \to Z$ be a birational projective morphism. Write 
$$K_{(\beta\circ\beta_1)^{-1}\sG}\sim_\mathbb{Q}(\beta\circ\beta_1)^*\sG+E_1$$ for some 
$\mathbb{Q}$-divisor $E_1$ on $Z_1$ with support contained in 
$\textup{Exc}\,\beta\circ\beta_1$.
Then 
$$K_{(\beta\circ\beta_1)^{-1}\sG}\sim_\mathbb{Q} \beta_1^*K_{\beta^{-1}\sG}+E_1-\beta_1^*E.$$
Note that $E_1-\beta_1^*E$ is supported on $\textup{Exc}\,\beta_1$. Indeed, we have $(\beta_1)_*E_1-E=(\beta_1)_*\big(E_1-\beta_1^*E\big)\sim_\mathbb{Q}0$ since $(\beta_1)_*K_{(\beta\circ\beta_1)^{-1}\sG}\sim_\mathbb{Z} K_{\beta^{-1}\sG}$. On the other hand,
$(\beta_1)_*E_1-E$ is $\beta$-exceptional, and hence we must have $(\beta_1)_*E_1-E=0$. This immediately implies that 
$E_1-\beta_1^*E$ is $\beta_1$-exceptional.

If $E$ is effective and $\beta^{-1}\sG$ is canonical, then $E_1-\beta_1^*E$ is effective, and hence so is $E_1$. This proves Item (1).

If $E=0$, then we see that $\sG$ is canonical if and only if so is $\beta^{-1}\sG$, proving Item (2). 
\end{proof}

\begin{lemma}\label{lemma:canonical_quasi_etale_cover}
Let $f\colon X_1 \to X$ be a finite cover of normal complex varieties, and let $\sG$ be a foliation on $X$ with
$K_\sG$ $\mathbb{Q}$-Cartier. Suppose that any codimension one component of the branch locus of $f$ is 
$\sG$-invariant. If $\sG$ is canonical, then so is $f^{-1}\sG$.
\end{lemma}

\begin{proof}
Set $\sG_1:=f^{-1}\sG$. By Lemma \ref{lemma:pull_back_fol_and_finite_cover}, we have $K_{\sG_1}\sim_\mathbb{Z}f^*K_\sG$. In particular, $K_{\sG_1}$ is 
$\mathbb{Q}$-Cartier. Let $\beta_1\colon Z_1 \to X_1$ be a birational projective morphism, and let $E_1$ be a $\beta_1$-exceptional prime divisor on $Z_1$. By \cite[Theorem 3.17]{kollar97},
there exists a projective birational morphism 
$\beta\colon Z \to X$ a well as a commutative diagram
\begin{center}
\begin{tikzcd}
Z_1 \ar[r, "g"]\ar[d, "\beta_1"'] & Z \ar[d, "\beta"] \\
X_1 \ar[r, "f"'] & X
\end{tikzcd}
\end{center}
\noindent such that $E:=g(E_1)$ is a $\beta$-exceptional prime divisor on $Z$.
Let $m$ denotes the ramification index of $g$ along $E_1$.
By Lemma \ref{lemma:pull_back_fol_and_finite_cover}, if $E$ is $\beta^{-1}\sG$-invariant, then 
$a(E_1,X_1,\sG_1)=m a(E,X,\sG)$, and  $a(E_1,X_1,\sG_1)=m a(E,X,\sG)+m-1$ otherwise. In particular, if $\sG$ is canonical, then so is $f^{-1}\sG$, proving the lemma.
\end{proof}

The following example shows that the converse is not true in general.

\begin{exmp}Let $G$ be a finite subgroup of $\textup{GL}(2,\mathbb{C})$ that does not contain any quasi-reflections, and set
$X:= \mathbb{A}^{2}/G$. Suppose that $G \not\subset \textup{SL}(2,\mathbb{C})$, so that $X$ is not canonical.  
Let $Y$ be a normal variety and consider the foliation $\sG$ on $X \times Y$ induced by the projection 
$X \times Y \to Y$. Let also $f \colon \mathbb{A}^{2} \times Y \to X \times Y$ be the quasi-\'etale cover induced by the projection morphism $\mathbb{A}^{2} \to \mathbb{A}^{2}/G=X$. Then $\sG$ is not canonical while 
$f^{-1}\sG$ is (see Example \ref{example:product} below).
\end{exmp}

\begin{lemma}\label{lemma:pull_back_foliation_singularities}
Let $Y$ and $Z$ be normal complex projective varieties, and let $\sH$ be the foliation on $Y$. Denote by
$\psi \colon Y \times Z \to Y$ the projection, and set $\sG:=\psi^{-1}\sH$. If $\sG$ is canonical, then so does $\sH$.
\end{lemma}

\begin{proof}
Suppose that $\sG$ is canonical. Let $\beta\colon Y_1 \to Y$ be a projective birational morphism with $Y_1$ normal, and let 
$F_1 \cong Y_1 $ be a fiber of the projection $Y_1 \times Z \to Y \times Z \to Z$. 
Denote by $\gamma\colon Y_1 \times Z \to Y \times Z$ the natural morphism, and set $F:=\gamma(F_1) \cong Y$.
One then readily checks that $K_\sG\sim_\mathbb{Z}K_{Y\times Z/Y}+\psi^*K_\sH$
and
$K_{\gamma^{-1}\sG}\sim_\mathbb{Z}K_{Y_1\times Z/Y_1}+\psi_1^*K_{\beta^{-1}\sH}$, where 
$\psi_1\colon Y_1\times Z \to Y_1$ denotes the projection.
It follows that ${K_\sG}_{|F}\sim_\mathbb{Z}K_{\sH}$ and ${K_{\gamma^{-1}\sG}}_{|F_1}\sim_\mathbb{Z} K_{\beta^{-1}\sH}$.
In particular, $K_\sH$ is $\mathbb{Q}$-Cartier.
By assumption, $K_{\gamma^{-1}\sG}\sim_\mathbb{Q}\gamma^*K_\sG+E$ for some effective $\gamma$-exceptional $\mathbb{Q}$-divisor, and hence $K_{\beta^{-1}\sH}\sim_\mathbb{Q}\beta^*K_\sH + E_{|F}$. This proves the lemma. 
\end{proof}

\subsection{$\mathbb{Q}$-factorial terminalization} In this paragraph, we analyze the behaviour of canonical singularities with respect to $\mathbb{Q}$-factorial terminalizations.

We will need the following auxiliary result.

\begin{lemma}\label{lemma:extension_reflexive_1_forms}
Let $X$ be a normal complex quasi-projective variety with klt singularities, let $1 \le p \le \dim X$ be an integer,
and let $\sA \subseteq\Omega_X^{[p]}$ be a saturated reflexive subsheaf of rank one.
Suppose that $\sA^{[m]}$ is a line bundle for some positive integer $m$.
Let $\beta \colon Z \to X$ be a resolution of singularities with exceptional set $E$, and assume that $E$ is a divisor with simple normal crossings.
Let $\sB \subseteq \Omega_Z^{p}$ denotes the saturation of $\beta^{[*]}\sA \subseteq \Omega_Z^{p}$, and let $E_1$ denotes 
the reduced divisor on $Z$ whose support is
the union of all irreducible components $E'$ of $E$ such that $\sB$ is not saturated in $\Omega_Z^{p}\big(\textup{log}\,E\big)$ at general points of $E'$.
Then, there exist an effective $\beta$-exceptional divisor $E_2$ and a rational number
$0 \le \varepsilon <1$ such that $\beta^*c_1(\sA)+E_2\sim_\mathbb{Q} c_1(\sB)+\varepsilon E_1$.
\end{lemma}

\begin{rem}
In the setup of Lemma \ref{lemma:extension_reflexive_1_forms}, \cite[Theorem 4.3]{greb_kebekus_kovacs_peternell10} shows that there is an embedding $\beta^{[*]}\sA \subseteq \Omega_Z^{p}$.
\end{rem}

\begin{rem}\label{rem:exceptional_set}
In the setup of Lemma \ref{lemma:extension_reflexive_1_forms}, suppose furthermore that $1 \le p \le \dim X -1$, and that $\sA$ is the conormal sheaf of a foliation $\sG$ on $X$. Then $\textup{Supp}\,E_1$ is the union of all irreducible components of $E$ that are invariant under $\beta^{-1}\sG$.
\end{rem}

\begin{proof}[Proof of Lemma \ref{lemma:extension_reflexive_1_forms}]The proof if very similar to that of 
\cite[Theorem 7.2]{greb_kebekus_kovacs_peternell10}, and so we leave some easy details to the reader.

Note that $\sB$ is a line bundle by \cite[Proposition 1.9]{hartshorne80}.
To prove the statement, it suffices to show that there exists a rational number
$0 \le \varepsilon <1$ such that, if $\sigma$ is any section of $\sA^{[m]}$ over some open set $U \subseteq X$, then the rational section $(\beta_{|U})^*\sigma$
of $\sB^{\otimes m}$ is regular on $\beta^{-1}(U)\setminus\textup{Supp}\,E_1$ and has poles of order at most $m\varepsilon$ along 
$(\textup{Supp}\,E_1) \cap \beta^{-1}(U)$.

The statement is local on $X$, hence we may shrink $X$, and assume that 
$\sA^{[\otimes m]} \cong \sO_X$.
Let $g\colon Y \to X$ be the associated cyclic cover (see \cite[Definition 2.52]{kollar_mori}),
and let $T$ denotes the normalization of the fiber product $Y\times_X Z$ with natural morphisms
$\gamma\colon T\to Y$ and 
$f\colon T \to Z$.
Note that $Y$ is klt by Fact \ref{fact:quasi_etale_cover_and_singularities} and that $g^{[*]}\sA\cong \sO_Y$. 

Let $\sigma \in H^0\big(X,\sA^{[\otimes m]}\big)$ be a nowhere vanishing section, and
consider the pull-back $\beta^*\sigma$, which is a rational section of 
$$\sB^{\otimes m} \subseteq \textup{S}^{m}\Omega_{Z}^p,$$ 
possibly with poles along $E$. 
Applying \cite[Theorem 4.3]{greb_kebekus_kovacs_peternell10} to $\gamma$ and using the fact that 
$g^{[*]}\sA$ is locally free, we see that there is an embedding 
$$\gamma^*\big(g^{[*]}\sA^{[m]}\big) \cong \big(\gamma^* g^{[*]}\sA\big)^{\otimes m}\subseteq \textup{S}^{m}\Omega_T^{[p]}.$$
This immediately implies that 
$$(\beta\circ f)^*\sigma \in H^0\big(T,\textup{S}^{m}\Omega_T^{[p]}\big).$$

Set $n:=\dim X$. Let $z \in \textup{Supp}\,E$ be a general point, and
let $(z_1,\ldots,z_n)$ be local coordinates on some open neighborhood $U$ of $z$ in $Z$ such that 
$z_1=0$ is a local equation of $E_1$. Let $t\in T$ such that $f(t)=z$, and 
let $(t_1,\ldots,t_n)$ be local coordinates on some open neighborhood $V$ of $t$ in $T$ with $f(V)\subseteq U$. We may assume without loss of generality that $f$ is given
by $(t_1,\ldots,t_n) \mapsto (t_1^{k_1},t_2,\ldots,t_n)$ on $V$. 

Suppose first that $\sB$ is not saturated in $\Omega_Z^{p}\big(\textup{log}\,E\big)$ at $z$. Shrinking $U$, if necessary, we may assume that $\sB_{|U}$ is generated by 
$dz_1\wedge\alpha_1+z_1\alpha_2$ 
where $\alpha_1 \in \wedge^{p-1}\big(\sO_Udz_2\oplus\cdots\oplus\sO_Udz_n\big)$ is nowhere vanishing
and $\alpha_2 \in \Omega_U^p=\wedge^{p}\big(\sO_Udz_1\oplus\cdots\oplus\sO_Udz_n\big)$.
Then ${\sB^{\otimes m}}_{|U}\subseteq {\textup{S}^{m}\Omega_Z^p}_{|U}$ is generated by $(dz_1\wedge\alpha_1+z_1\alpha_2)^{\otimes m}$ and
${f^*\sB^{\otimes m}}_{|V} \subseteq {\textup{S}^{m}\Omega_T^{[p]}}_{|V}$ by $t_1^{m(k_1-1)}\big(k_1dt_1\wedge df(\alpha_1)+t_1df(\alpha_2)\big)^{\otimes m}$.
One then readily checks that $\beta^*\sigma$ has a pole of order at most $m(1-\frac{1}{k_1})$ along $z_1=0$
since $(\beta\circ f)^*\sigma$ is regular on $V$.

If $\sB$ is saturated in $\Omega_Z^{p}\big(\textup{log}\,E\big)$ at $z$, then we may assume that $\sB_{|U}$ is generated by 
a nowhere vanishing $p$-form $\alpha_1$ with 
$\alpha_1 \in \wedge^{p}\big(\sO_Udz_2\oplus\cdots\oplus\sO_Udz_n\big)$. Arguing as above, one concludes that 
$\beta^*\sigma$ is regular in codimension one on $U$.

Write $f^* E_1
=\sum_{i\in I}k_i F_i$
where the $F_i$ are prime divisors on $T$ and $k_i$ are positive integers, and let $0\le \varepsilon<1$ be a rational number such that $\varepsilon \ge 1-\frac{1}{k_{i}}$ for every indices $i\in I$. We conclude that $\beta^*\sigma$ (viewed as a rational section of the line bundle $\sB^{\otimes m}$)
is regular on $X \setminus \textup{Supp}\,E_1$ and has poles of order at most $m\varepsilon$ along $\textup{Supp}\,E_1$.
This finishes the proof of the lemma.
\end{proof}

The following is an easy consequence of Lemma \ref{lemma:extension_reflexive_1_forms}.

\begin{prop}\label{prop:numerical_dimension}
Let $X$ be a normal complex projective variety with klt singularities, and let $\sG$ be a codimension one foliation on $X$
such that $c_1(\sN_\sG)$ is $\mathbb{Q}$-Cartier.
Let $\beta \colon Z \to X$ be a resolution of singularities with exceptional set $E$, and assume that $E$ is a divisor with simple normal crossings.
Let also $E_1$ denotes the reduced divisor on $Z$ whose support is the union of all irreducible components of $E$ that are invariant under $\beta^{-1}\sG$. 
There exists a rational number $0\le \varepsilon <1$ such that 
$$\kappa\big(-c_1(\sN_\sG)\big) = \kappa\big(-c_1(\sN_{\beta^{-1}\sG})+\varepsilon E_1\big)
\quad \text{and} \quad 
\nu\big(-c_1(\sN_\sG)\big) = \nu\big(-c_1(\sN_{\beta^{-1}\sG})+\varepsilon E_1\big).$$ In particular, $-c_1(\sN_\sG)$ is pseudo-effective if and only if so is $-c_1(\sN_{\beta^{-1}\sG})+\varepsilon E_1$. 
\end{prop}

\begin{proof}
Applying Lemma \ref{lemma:extension_reflexive_1_forms} and using Remark \ref{rem:exceptional_set}, we see that there exist an effective $\beta$-exceptional divisor $E_2$ and a rational number $0\le \varepsilon <1$ such that 
$-\beta^*c_1(\sN_\sG)+E_2 \sim_\mathbb{Q}-c_1(\sN_{\beta^{-1}\sG})+\varepsilon E_1$.
By \cite[Lemma II.3.11]{nakayama04}, we have 
$\kappa\big(-\beta^*c_1(\sN_\sG)+E_2\big) = \kappa\big(-c_1(\sN_\sG)\big)$ since $E_2$ is effective and $\beta$-exceptional.
We also have $\nu\big(-\beta^*c_1(\sN_\sG)+E_2\big) = \nu\big(-c_1(\sN_\sG)\big)$ by \cite[Proposition V.2.7]{nakayama04}. 
This proves the proposition. 
\end{proof}

The following results often reduce the study of mildly singular foliations with numerically trivial canonical class on varieties with canonical singularities to those on varieties with terminal singularities.

\begin{prop}\label{prop:terminalization_canonical_singularities}
Let $X$ be a normal complex projective variety with canonical singularities, and let $\sG$ be a foliation on $X$. 
Suppose that $\sG$ is canonical and that $K_\sG$ is Cartier.
Let $\beta \colon Z \to X$ be a $\mathbb{Q}$-factorial terminalization of $X$. Then $\beta^{-1}\sG$ is canonical with $K_{\beta^{-1}\sG}\sim_\mathbb{Z}\beta^*K_\sG$.  
\end{prop}

\begin{proof}
Note that $c_1(\sN_\sG)$ is $\mathbb{Q}$-Cartier since $c_1(\sN_\sG) \sim_\mathbb{Z} K_\sG - K_X$.
Recall also that $K_Z\sim_\mathbb{Q}\beta^*K_X$.
Applying Lemma \ref{lemma:extension_reflexive_1_forms} to a resolution of $Z$, we see that there exist effective $\beta$-exceptional $\mathbb{Q}$-diviors $E_1$ and $E_2$ with $E_1$ reduced, and a rational number $0 \le \varepsilon <1$ such that $c_1\big(\sN_{\beta^{-1}\sG}\big)+E_2\sim_\mathbb{Q} \beta^*c_1\big(\sN_\sG\big)+\varepsilon E_1$. On the other hand, we have $K_{\beta^{-1}\sG}\sim_\mathbb{Z}\beta^*K_\sG+F$ for some effective integral $\beta$-exceptional Weil divisor $F$ since 
$\sG$ is canonical and $K_\sG$ is Cartier. Since $K_{\beta^{-1}\sG}\sim_\mathbb{Q} K_Z+c_1\big(\sN_{\beta^{-1}\sG}\big)$ and $K_\sG \sim_\mathbb{Q} K_X+c_1\big(\sN_\sG\big)$, we must have $\varepsilon E_1 = F + E_2$ by the negativity lemma. It follows that $F=0$, and hence $K_{\beta^{-1}\sG}\sim_\mathbb{Z}\beta^*K_\sG$. By Lemma \ref{lemma:singularities_birational_morphism}, we see that $\beta^{-1}\sG$ is canonical, completing the proof of the proposition.
\end{proof}

Example \ref{example:ruled_surface} below shows that Proposition \ref{prop:terminalization_canonical_singularities} is wrong if one drops the assumption that $K_\sG$ is Cartier.

\begin{exmp}\label{example:ruled_surface}
Let $E$ be an elliptic curve, and set $X_1:= E \times \mathbb{P}^1$.
Let $G$ be a cyclic group of order $2$ acting on $E$ by $x \mapsto -x$ and acting on 
$\mathbb{P}^1$ by $(x:y)\mapsto (y:x)$, where $(x:y)$ are homogeneous coordinates on $\mathbb{P}^1$.
Set $X:=X_1/G$, and denote by $f \colon X_1\to X$ the projection map, which is two-to-one quasi-\'etale cover. We obtain a rational surface containing eight rational double points. 
Consider the foliation $\sG$ on $X$ given by the morphism $X \to \mathbb{P}^1/G\cong \mathbb{P}^1$. The canonical divisor $K_\sG$ is not Cartier, but $2 K_\sG$ is. Let $\beta\colon Z \to X$ be the blow-up of the eight singular points, and denote by $E_i$ the $\beta$-exceptional divisors. Then $\sG$ and $\beta^{-1}\sG$ are canonical but
$K_{\beta^{-1}\sG}\sim_\mathbb{Q}\beta^*K_\sG+\frac{1}{2}\sum_i E_i$.
\end{exmp}

\begin{lemma}
Let $X$ be a normal complex projective variety with canonical singularities, and let $\sG$ be a foliation on $X$. 
Suppose that $\sG$ is canonical and that $\det\sN_\sG$ is Cartier.
Let $\beta \colon Z \to X$ be a $\mathbb{Q}$-factorial terminalization of $X$. Then $\beta^{-1}\sG$ is canonical with $K_{\beta^{-1}\sG}\sim_\mathbb{Z}\beta^*K_\sG$.  
\end{lemma}

\begin{proof}
Recall that $K_Z\sim_\mathbb{Q}\beta^*K_X$.
By \cite[Theorem 4.3]{greb_kebekus_kovacs_peternell10}, there is an effective $\beta$-exceptional Weil divisor $E$ on $Z$ such that $c_1(\sN_{\beta^{-1}\sG})\sim_\mathbb{Z}\beta^*c_1(\sN_\sG)-E$. Since $K_{\beta^{-1}\sG}\sim_\mathbb{Q} K_Z+c_1\big(\sN_{\beta^{-1}\sG}\big)$ and $K_\sG\sim_\mathbb{Q} K_X+c_1\big(\sN_\sG\big)$, we must have
$K_{\beta^{-1}\sG}\sim_\mathbb{Q} \beta^*K_\sG-E$. It follows that $E=0$ since $\sG$ is canonical by assumption. Applying Lemma \ref{lemma:singularities_birational_morphism}, we see that $\beta^{-1}\sG$ is canonical. This finishes the proof of the lemma.
\end{proof}

\subsection{Algebraically integrable foliations} In this paragraph, we address algebraically integrable foliations with canonical singularities.

\begin{lemma}\label{lemma:canonical_foliation_versus_lc_pairs}
Let $X$ be a normal complex projective variety, and let $\sG$ be an algebraically integrable foliation on $X$. Suppose that $\sG$ is canonical. 
Let $\psi\colon Z \to Y$ be the family of leaves, and let $\beta\colon Z \to X$ be the natural morphism (see \ref{family_leaves}). 
Then the following holds.
\begin{enumerate}
\item The foliation $\beta^{-1}\sG$ is canonical with $K_{\beta^{-1}\sG}\sim_\mathbb{Q}\beta^* K_\sG$.
\item There exists a dense open set $Y^\circ\subseteq Y$ such that 
$Z$ has canonical singularities over $Y^\circ$.
In particular, a general fiber of $\psi$ has canonical singularities.
\item If $P$ is a prime divisor on $Y$ and $C:=(\psi^*P)_{\textup{red}}$, then the pair $(Z,C)$ is log canonical over an open neighborhood of the generic point of $P$. 
\end{enumerate}
\end{lemma}

\begin{proof}
Recall from subsection \ref{family_leaves} that there is an effective Weil $\bQ$-divisor $B$ on $Z$ such that 
\begin{equation}\label{eq:canonical_bundle_formula_1}
K_{\beta^{-1}\sG}+B\sim_\mathbb{Q} \beta^* K_{\sG}.
\end{equation} 
On the other hand, since $\sG$ is canonical by assumption, there exists an effective Weil $\mathbb{Q}$-divisor $E$ on $Z$ such that
\begin{equation}\label{eq:canonical_bundle_formula_2}
K_{\beta^{-1}\sG}\sim_\mathbb{Q}\beta^*K_{\sG}+E.
\end{equation}
From equations \eqref{eq:canonical_bundle_formula_1} and \eqref{eq:canonical_bundle_formula_2}, we obtain
\begin{equation*}
B+E\sim_\mathbb{Q} 0.
\end{equation*}
This immediately implies that $B=0$ and $E=0$, and shows that $K_{\beta^{-1}\sG}\sim_\mathbb{Q}\beta^* K_\sG$. 
By Lemma \ref{lemma:singularities_birational_morphism}, $\beta^{-1}\sG$ is then canonical, proving Item (1).

Let $\beta_1\colon Z_1 \to Z$ be a resolution of singularities, and set $\psi_1:=\psi\circ\beta_1$.
By Example \ref{example:canonical_class_foliation}, we have 
$K_{\beta^{-1}\sG}\sim_\mathbb{Z}K_{Z/Y}-R(\psi)$ and $K_{\beta_1^{-1}\beta^{-1}\sG}\sim_\mathbb{Z}K_{Z_1/Y}-R(\psi_1)+F_1$, where $R(\psi)$ and $R(\psi_1)$ denote the ramification divisors of $\psi$ and $\psi_1$ respectively, and $F_1$ is a $\beta_1$-exceptional $\mathbb{Q}$-divisor on $Z_1$ such that $\psi_1(\textup{Supp}\,F_1)$ has codimension at least two in $Y$. In particular, $K_{Z/Y}-R(\psi)$ is $\mathbb{Q}$-Cartier.
Since $\beta^{-1}\sG$ is canonical, there exists an effective 
$\beta_1$-exceptional
$\mathbb{Q}$-divisor $E_1$ on $Z_1$ such that
\begin{equation}\label{eq:canonical_bundle_formula_3}
K_{\beta_1^{-1}\beta^{-1}\sG}\sim_\mathbb{Q}\beta_1^*K_{\beta^{-1}\sG}+E_1.
\end{equation}
Set also $Y^\circ:=Y\setminus \psi_1\big(\textup{Supp}\,R(\psi_1)\cup \textup{Supp}\,F_1\big)$, $Z^\circ:=\psi^{-1}(Y^\circ)$, and
$Z_1^\circ:=\psi_1^{-1}(Y^\circ)$.
Equation \eqref{eq:canonical_bundle_formula_3} then gives
$$K_{Z_1^\circ}\sim_\mathbb{Q} \beta_1^*K_{Z^\circ}+{E_1}_{|Z_1^\circ}.$$
This shows that $Z^\circ$ has canonical singularities. Item (2) follows easily.

To prove Item (3), we may assume that $\beta\colon Z_1\to Z$ be a resolution of $(Z,C)$.  
Shrinking $Y$, if necessary, we may assume that $Y$ is smooth, and that $\psi_1$ is also equidimensional.
Then $K_{\beta_1^{-1}\beta^{-1}\sG}=K_{Z_1/Y}-R(\psi_1)$. 
We may also assume without loss of generality that either $R(\psi_1)=0$ or
$\psi_1\big(\textup{Supp}\,R(\psi_1)\big)=P$. It follows that
$C=\psi^*P - R(\psi)$ and that $C_1:=(\psi_1^*P)_{\textup{red}}=\psi_1^*P - R(\psi_1)$. Equation \eqref{eq:canonical_bundle_formula_3} now yields
$$K_{Z_1}+C_1 \sim_\mathbb{Q} \beta^*(K_Z+C) + E_1.$$
Since $C_1$ is reduced, 
we conclude that $(Z,C)$ is log canonical, completing the proof of the lemma.
\end{proof}

\begin{rem}\label{rem:family_leaves_cartier}
In the setup of Lemma \ref{lemma:canonical_foliation_versus_lc_pairs} (1), suppose in addition that $K_\sG$ is Cartier. Then 
$K_{\beta^{-1}\sG}\sim_\mathbb{Z}\beta^* K_\sG$.
\end{rem}

The converse is also true if $\dim Y =1$ by Proposition \ref{prop:canonical_singularities_alg_int_cd1_foliation} below.

\begin{prop}\label{prop:canonical_singularities_alg_int_cd1_foliation}
Let $X$ be a normal complex projective variety, and let $\psi \colon X \to Y$ be a surjective morphism onto a smooth projective curve. Denote by $\sG$ the foliation induced by $\psi$. Let $\beta\colon X_1 \to X$ be a resolution of singularities, and denote by $R(\psi_1)$ the ramification divisor of $\psi_1:=\psi\circ\beta$.
If $X$ has canonical singularities over $Y \setminus \psi_1\big(\textup{Supp}\,R(\psi_1)\big)$ and, for any point $P$ in $\psi_1\big(\textup{Supp}\,R(\psi_1)\big)$, the pair $\big(X,(\psi^*P)_{\textup{red}}\big)$ is log canonical over an open neighborhood of $P$, then $\sG$ is canonical.
\end{prop}

\begin{proof}
Let $R(\psi)$ denotes the ramification divisor of $\psi$, and denote by $B(\psi)$ the reduced divisor on $Y$ with support $\psi\big(\textup{Supp}\, R(\psi)\big)$. By Example \ref{example:canonical_class_foliation}, we have $K_\sG=K_{X/Y}-R(\psi)$.
Note that $K_\sG$ is $\mathbb{Q}$-Cartier since $R(\psi)=\psi^*B(\psi)- \big(\psi^*B(\psi)\big)_{\textup{red}}$ and $K_X+\big(\psi^*B(\psi)\big)_{\textup{red}}$ is $\mathbb{Q}$-Cartier by assumption.

Let $\beta\colon X_1 \to X$ be a resolution of singularities, and let $E$ be the $\beta$-exceptional divisor on $X_1$ such that 
\begin{equation}\label{eq:canonical_bundle_formula}
K_{\beta^{-1}\sG}\sim_\mathbb{Q}\beta^*K_\sG+E.
\end{equation}
Denote by $B(\psi_1)$ the reduced divisor on $Y$ with support $\psi_1\big(\textup{Supp}\, R(\psi_1)\big)$. 
Set $C:=\big(\psi^*B(\psi_1)\big)_{\textup{red}}$ and $C_1:=\big(\psi_1^*B(\psi_1)\big)_{\textup{red}}$. Then Equation \eqref{eq:canonical_bundle_formula} above yields
$$K_{X_1}+C_1=\beta^*(K_X+C)+E.$$
Since $X$ has canonical singularities over $Y^\circ:=Y \setminus \psi_1\big(\textup{Supp}\, R(\psi_1)\big)$, we see that $E$ is effective over $Y^\circ$.
On the other hand, since the pair $(X,C)$ is log canonical over some open neighborhood of $\textup{Supp}\, B(\psi)$, we conclude that $E$ is effective over some open neighborhood of $\textup{Supp}\, B(\psi)$. This proves that $E$ is effective, completing the proof of the proposition.
\end{proof}

\begin{exmp}\label{example:product}
Let $Y$ and $Z$ be normal projective varieties, and let $\sG$ be the foliation on $X : = Y \times Z$ induced by the projection 
$Y \times Z \to Y$. Then $\sG$ is canonical if and only if $Z$ has canonical singularities.
Indeed, if $\sG$ is canonical, then $Z$ has canonical singularities by Lemma \ref{lemma:canonical_foliation_versus_lc_pairs} above. Suppose that $Z$ has canonical singularities, and let $\beta_1\colon Z_1 \to Z$ be a resolution of singularities. Let also $Y_1 \to Y$ be a resolution of $Y_1$. Let $\sG_1$ be the foliation on $Y_1 \times Z$ induced by the projection 
$Y_1 \times Z \to Y_1$, and denote by $\gamma_1\colon X_1:=Y_1 \times Z \to Y \times Z=X$ the natural morphism. Notice that 
$K_\sG$ is $\mathbb{Q}$-Cartier, and that $K_{\sG_1} \sim_\mathbb{Q} \gamma_1^*K_\sG$. Thus, by Lemma \ref{lemma:singularities_birational_morphism}, it suffices to show that $\sG_1$ is canonical.
Let $\sG_2$ be the foliation on $X_2:=Y_1 \times Z_1$ induced by the projection 
$Y_1 \times Z_1 \to Y_1$, and denote by $\gamma_2\colon X_2=Y_1 \times Z_1 \to Y_1 \times Z=X_1$ the natural morphism.
Since $Z$ has canonical singularities, $K_{\sG_2} \sim_\mathbb{Q} \gamma_2^*K_{\sG_1} + E_2$ for some effective and 
$\gamma_2$-exceptional Weil $\mathbb{Q}$-divisor.
Now, $\sG_2$ is canonical since it is a regular foliation (see Lemma \ref{lemma:regular_versus_canonical}). In particular, if 
$\gamma_3 \colon X_3 \to X_2$ is any projective birational morphism with $X_3$ normal, then 
$K_{\gamma_3^{-1}\sG_2} \sim_\mathbb{Q} \gamma_3^*K_{\sG_2}+E_3$ 
for some effective and 
$\gamma_3$-exceptional Weil $\mathbb{Q}$-divisor. It follows that 
$K_{\gamma_3^{-1}\sG_2} \sim_\mathbb{Q}(\gamma_2\circ\gamma_3)^*K_{\sG_1}+ \gamma_3^* E_2 + E_3$. This shows that $\sG$ is canonical.
\end{exmp}

The following result, which will be crucial for the proof of Theorem \ref{thm:regular_foliation_morphism}, extends \cite[Lemma 2.12]{fol_index} to the singular setting. 

\begin{prop}\label{proposition:dicritical_versus_canonical}
Let $\sG$ be a foliation of rank $r \ge 1$ on a normal complex projective variety $X$. 
Suppose that $\sG$ is algebraically integrable and that $K_\sG$ is $\mathbb{Q}$-Cartier.
Let $\psi\colon Z \to Y$ be the family of leaves, and let $\beta\colon Z \to X$ be the natural morphism (see \ref{family_leaves}). 
Let also $B$ be the effective $\beta$-exceptional $\mathbb{Q}$-divisor on $Z$ such that $K_{\beta^{-1}\sG}+B\sim_{\mathbb{Q}}\beta^*K_\sG$.
If $E$ is a $\beta$-exceptional prime divisor on $Z$ such that $\psi(E)=Y$, then 
$E \subseteq \textup{Supp}\,B$.
\end{prop}

\begin{rem}
Proposition \ref{proposition:dicritical_versus_canonical} says that $\sG$ is not canonical along $\beta(E)$.
\end{rem}

\begin{proof}[Proof of Proposition \ref{proposition:dicritical_versus_canonical}]
We argue by induction on $r\ge 1$. 
Let $E$ be a $\beta$-exceptional prime divisor on $Z$, and assume that $\psi(E)=Y$.

Suppose first that $r=1$. Let $m$ be a positive integer and let $X ^\circ \subseteq X$ be a dense open set such that 
$\sO_{X^\circ}\big(m{K_\sG}_{|X^\circ}\big)\cong \sO_{X^\circ}$. Suppose in addition that 
$\beta(E)\cap X^\circ\neq \emptyset$.
Let $f^\circ \colon X_1^\circ \to X^\circ$ be the associated cyclic cover, which is quasi-\'etale (see \cite[Definition 2.52]{kollar_mori}). Finally, let $Z_1^\circ$ be the normalization of the product $Z^\circ \times_{X^\circ} X_1^\circ$, where 
$Z^\circ:=\beta^{-1}(X^\circ)$. If $\beta_1^\circ\colon Z_1^\circ \to X_1^\circ$ 
and $g^\circ\colon Z_1^\circ \to Z^\circ$
denote the natural morphisms, then there exists an effective $\beta_1^\circ$-exceptional divisor $B_1^\circ$ on $Z_1^\circ$ such that $K_{(\beta_1^\circ)^{-1}(f^\circ)^{-1}(\sG_{|X^\circ})}+B_1^\circ \sim_\mathbb{Z} 0$. Moreover, 
$(g^\circ)^*\big(B_{|Z^\circ}\big) - B_1^\circ $ is effective (see subsection \ref{family_leaves}).
Since $\psi(E)=Y$ by assumption, we see that $E$ is not invariant under $\beta^{-1}\sG$.
Let $E_1^\circ$ be a prime divisor on $Z_1^\circ$ such that $g^\circ(E_1^\circ)=E \cap Z^\circ$.
Notice that $E_1^\circ$ is not invariant under $(\beta_1^\circ)^{-1}(f^\circ)^{-1}(\sG_{|X^\circ})$ and that 
$E_1^\circ$ is $\beta_1^\circ$-exceptional.
By \cite[Theorem 5]{seidenberg67}, the singular locus of $X_1^\circ$ is invariant under any derivation on $X_1^\circ$. Thus, if $\beta_1^\circ(E_1^\circ)$ is contained in $X_1^\circ\setminus (X_1^\circ)_{\textup{reg}}$ then $E_1^\circ \subseteq \textup{Supp}\,B_1^\circ$ by Lemma \ref{lemma:prime_invariant} below.
If $\beta_1^\circ(E_1^\circ) \cap (X_1^\circ)_{\textup{reg}} \neq\emptyset$, then $\beta_1^\circ(E_1^\circ)$ is contained in the singular locus of $(f^\circ)^{-1}\big(\sG_{|X_1^\circ}\big)$, and hence
$E_1^\circ \subseteq \textup{Supp}\,B_1^\circ$ by Lemma \ref{lemma:prime_invariant} again and Lemma \ref{lemma:singular_set_invariant}. In either case, since $(g^\circ)^*\big(B_{|Z^\circ}\big) - B_1^\circ $ is effective,
we see that $E \subseteq \textup{Supp}\,B$.
This proves the proposition when $r=1$.

Suppose from now on that $r \ge 2$. We may assume without loss of generality that $X \subseteq \mathbb{P}^N$ for some positive integer $N$.

Let $H \subset X$ be a general hyperplane section. We may assume that $H$ and $G:=\beta^{-1}(H)$
are normal varieties (see \cite[Th\'eor\`eme 12.2.4]{ega28}). 
Set $\gamma:=\beta_{|G}\colon G \to H$ and
$\sE:={\sG}_{|H}\cap T_H$, so that $\gamma^{-1}\sE=\beta^{-1}\sG_{|G} \cap T_{G}$.
Applying \cite[Th\'eor\`eme 12.2.4]{ega28} again, we see that general fibers of $\psi_{|G}\colon G \to Y$ are integral by general choice of $H$. This implies that  
$\psi_{|G}\colon G \to Y$ is the family of leaves of $\sE$. 
Since the restriction of $\beta$ to any fiber of $\psi$ is finite, we have
$\dim \beta(E) \ge r-1 \ge 1$. In particular, $E_{|G}$ is a non-zero divisor on $G$.

Using Proposition \ref{prop:bertini} and the formula $K_{\beta^{-1}\sG}+B\sim_{\mathbb{Q}}\beta^*K_\sG$, we obtain
\begin{equation}\label{eq:restriction}
K_\sE\sim_\mathbb{Z} {K_\sG}_{|H}+H_{|H}
\end{equation}
and
\begin{equation}\label{eq:canonical_hyperplane}
K_{\gamma^{-1}\sE}\sim_\mathbb{Z}
{K_{\beta^{-1}\sG}}_{|G}+G_{|G}-B_G 
\sim_\mathbb{Q} \gamma^*K_\sE-B_G-B_{|G},
\end{equation}
for some effective $\gamma$-exceptional divisor $B_G$ on $G$. 
By induction, we must have $$\textup{Supp}\,E_{|G} \subseteq \textup{Supp}\big(B_G+B_{|G}\big).$$
Given a general fiber $F$ of $\psi$, we may assume that $H \cap F_{\textup{reg}}$ is smooth by Bertini's theorem.
This immediately implies that $\psi_{|G}\big(\textup{Supp}\,B_G\big) \subsetneq Y$ (see Proposition \ref{prop:bertini}).
On the other hand, by general choice of $H$, any irreducible component of $E \cap G$ is mapped onto $Y$ by $\psi_{|G}$.
Therefore, we have 
$\textup{Supp}\,E_{|G} \subseteq \textup{Supp}\, B_{|G}$, and hence 
$$E \subseteq \textup{Supp}\, B.$$
This completes the proof of the proposition.
\end{proof}

\begin{lemma}\label{lemma:prime_invariant}
Let $Z$ and $X$ be normal complex varieties, and let $\beta\colon Z \to X$ be a birational projective morphism.
Let $\sG$ be a foliation of rank one on $X$. Suppose that $K_\sG$ is Cartier, and write 
$K_{\beta^{-1}\sG}\sim_\mathbb{Z}\beta^*K_\sG+B$ for some $\beta$-exceptional divisor $B$ on $Z$.
Let $E \subset Z$ be a prime divisor not contained in $\textup{Supp}\,B$. If 
$\beta(E)$ is contained in a proper closed $\sG$-invariant subvariety $Y \subsetneq X$, then $E$ is invariant under $\beta^{-1}\sG$. 
\end{lemma}

\begin{proof}
The statement is local on $X$, hence we may shrink $X$ and assume that $K_\sG \sim_\mathbb{Z} 0$ and that $X$ is affine,
$X \subseteq \mathbb{A}^N$ for some integer $N \ge 1$.

Let $\partial_Z \in H^0\big(Z,T_Z\boxtimes\sO_Z(B)\big)$ such that $\beta^{-1}\sG\boxtimes\sO_Z(B)=\sO_Z\partial_Z$, and denote by $\partial_X\in H^0(X,T_X)$ 
the derivation on $X$ induced by $\partial_Z$, so that $\sG=\sO_X\partial_X$. By construction, $\partial_Z$ is a regular derivation on $Z \setminus \textup{Supp}\, B$.

Let $f$ be a non-zero regular function on $X$, vanishing on $Y$, such that 
$m:=v_E(f)$ is minimal, where $v_E$ denote the divisorial valuation on the field of rational functions on $X$ induced by $E$.
Let also $g$ be a local equation of $E$ on some open subset $U \subseteq Z\setminus\textup{Supp}(B)$. There exists a function $u$ on $U$ such that
$u_{|E\cap U}\not\equiv 0$ and $f\circ\beta_{|U}=u g^m$. It follows that
$$\partial_Z\big(f\circ\beta\big)_{|U}= g^m {\partial_Z}_{|U}(u)+m u g^{m-1}{\partial_Z}_{|U}(g).$$
On the other hand, we have
$$\partial_Z(f\circ\beta) = \beta\circ\partial_X(f)$$ 
and thus $v_E\big(\partial_Z(f\circ\beta)\big)=v_E\big(\partial_X(f)\big) \ge v_E(f)=m$ by choice of $f$, using the fact that 
$\partial_X(f)$ vanishes on $Y$ by construction. Moreover, $v_E({\partial_Z}_{|U}(u)) \ge 0$ since ${\partial_Z}_{|U}$ is regular on $U$.
This immediately implies that $v_E\big(\partial_Z(g)\big)\ge 1$, proving the lemma.
\end{proof}

The following easy consequences of Lemma \ref{lemma:prime_invariant}
might be of independent interest.

\begin{prop}
Let $Z$ and $X$ be normal complex varieties, and let $\beta\colon Z \to X$ be a proper birational morphism.
Let $\partial_Z \in H^0(Z,T_Z)$, and let $\partial_X\in H^0(X,T_X)$ be the induced derivation on $X$. Suppose that $\partial_X\neq 0$ in codimension $1$.
Let also $E \subset Z$ be a prime divisor. 
If
$\beta(E)$ is contained in the singular locus of $X$, then $E$ is invariant under $\partial_Z$.
\end{prop}

\begin{proof}
By \cite[Theorem 5]{seidenberg67}, the singular locus of $X$ is invariant under $\partial_X$.
Let $B$ be the maximal effective divisor on $Z$ such that $\partial_Z \in H^0\big(Z,T_Z\boxtimes\sO_Z(-B)\big)$. Observe that
$B$ is $\beta$-exceptional since $\partial_X\neq 0$ in codimension $1$ by assumption.
If $E \subseteq \textup{Supp}\, B$, then ${\partial_Z}_{|E} \equiv 0$ and $E$ is obviously invariant under $\partial_Z$. If $E$ is not contained in $\textup{Supp}\, B$, then 
the claim follows from Lemma \ref{lemma:prime_invariant} above applied to the foliation $\sG=\sO_X\partial \subseteq T_X$.
\end{proof}

Recall that an \textit{equivariant resolution} of a normal variety $X$ is a projective birational
morphism $\beta\colon Z \to X$ with $Z$ smooth such that $\beta$ restricts to an isomorphism over the smooth locus of $X$ and such that $\beta_* T_ Z = T_X$.

\begin{cor}\label{cor:equivariant_resolution}
Let $X$ be a normal complex variety, and let $\beta\colon Z \to X$ be an equivariant resolution of $X$.
Then $\beta_* T_ Z (-\textup{log}\, E) = T_X$, where $E$ denotes the union of all prime $\beta$-exceptional divisors.
\end{cor}

\subsection{Singularities of foliations with numerically trivial canonical class}
In general, Definition \ref{definition:canonical_singularities} requires some understanding of the numbers $a(E,X,\sG)$ for all exceptional divisors of all birational modifications of $X$. However, if $K_\sG$ is $\mathbb{Q}$-Cartier and $K_\sG\equiv 0$, then we have the following characterization of canonical singularities, due to Loray, Pereira, and Touzet when $X$ is smooth.

\begin{prop}\label{proposition:canonical_versus_uniruled}
Let $X$ be a normal complex projective variety, and let $\sG$ be a foliation on $X$ with $K_\sG$ $\mathbb{Q}$-Cartier and
$K_\sG\equiv 0$. Then $\sG$ has canonical singularities if and only if $\sG$ is not uniruled. 
\end{prop}

\begin{proof}
The same argument used in the proof of \cite[Corollary 3.8]{lpt} shows that the conclusion of Proposition \ref{proposition:canonical_versus_uniruled} holds. One only needs to replace the use of \cite[Theorem 3.7]{lpt} by Theorem \ref{thm:uniruled_versus_pseudo-effectivity} below.
\end{proof}

\begin{thm}\label{thm:uniruled_versus_pseudo-effectivity}
Let $X$ be a normal complex projective variety, and let $\sG$ be a foliation on $X$ with canonical singularities. Then $\sG$ is uniruled if and only if $K_\sG$ is not pseudo-effective.
\end{thm}

\begin{proof}
Let $\beta\colon Z \to X$ be a resolution of singularities.

Suppose first that $K_\sG$ is not pseudo-effective. Since $\sG$ is canonical, $K_{\beta^{-1}\sG}$ is not pseudo-effective as well. Applying \cite[Theorem 4.7]{campana_paun15} to $\beta^{-1}\sG$, we see that $\sG$ is uniruled.

Suppose now that $\sG$ is uniruled. The same argument used in the proof of \cite[Theorem 3.7]{lpt} applied to 
$\beta^{-1}\sG$
shows that 
$K_{\beta^{-1}\sG}$ is pseudo-effective. Note that the assumption that $\sG$ is canonical is not used here.
This immediately implies that $K_\sG$ is not pseudo-effective, finishing the proof of the theorem.
\end{proof}

\subsection{Abundance for algebraically integrable foliations with numerically trivial canonical class}

Let $X$ be a projective klt variety with numerically trivial canonical class. By a theorem of Nakayama, $K_X$ is torsion. Proposition \ref{prop:abundance_alg_int} below extends Nakayama's result to mildly singular algebraically integrable foliations.

\begin{prop}\label{prop:abundance_alg_int}
Let $X$ be a normal complex projective variety, and let $\sG$ be an algebraically integrable foliation on $X$. Suppose that 
$\sG$ is canonical with $K_\sG\equiv 0$. Then $K_\sG$ is torsion.
\end{prop}

\begin{proof}
Let $\psi\colon Z \to Y$ be the family of leaves, and let $\beta\colon Z \to X$ be the natural morphism (see \ref{family_leaves}). Let also $F$ be a general fiber of $\psi$.
By Lemma \ref{lemma:canonical_foliation_versus_lc_pairs}, $\beta^{-1}\sG$ is canonical with 
$K_{\beta^{-1}\sG}\sim_\mathbb{Q}\beta^* K_\sG$, and moreover, $F$ has canonical singularitites.
Recall from Example \ref{example:canonical_class_foliation} that $K_{\beta^{-1}\sG}=K_{Z/Y}-R(\psi)$, where $R(\psi)$ denotes the ramification divisor of $\psi$.
From the adjunction formula, we conclude that $K_F\sim_\mathbb{Z}{K_Z}_{|F} \equiv 0$, and thus $K_F$ is torsion by \cite[Corollary V 4.9]{nakayama04}. It follows that there exists a $\mathbb{Q}$-divisor $B$ on $Z$ 
with 
$\psi(\textup{Supp}\,B)\subsetneq Y$ such that 
$$K_{Z/Y}-R(\psi)\sim_\mathbb{Z}B.$$
On the other hand, since $K_{Z/Y}-R(\psi)\equiv 0$ by assumption, there exists a $\mathbb{Q}$-divisor on $Y$ such that 
$B = \psi^* D$. This follows easily from \cite[Theorem A.7]{reid_chapters}. 
Therefore, we have
$$K_Z+\psi^{-1}(C)\sim_\mathbb{Q} \psi^*\big(K_Y+C+D\big),$$
where $C$ is the reduced divisor on $Y$ with support $\psi\big(\textup{Supp}\,R(\psi)\big)$.
By Lemma \ref{lemma:pull-back_Q_Cartier}, $D$ is $\mathbb{Q}$-Cartier.
Moreover, we have $D \equiv 0$ by the projection formula.
By Lemma \ref{lemma:canonical_foliation_versus_lc_pairs} applied to $\beta^{-1}\sG$, the discriminant of the klt-trivial fibration $\psi\colon \big(Z,\psi^{-1}(C)\big)\to Y$ is $C$ (we refer the reader to \cite{floris} for the definitions of klt-trivial fibration and discriminant). From \cite[Theorem 1.2]{floris}, we conclude that $D$ is torsion, proving the proposition.
\end{proof}

\section{Weakly regular foliations on singular spaces}\label{section:regular_foliations}

\subsection{Definitions and examples}
There are several notions of regularity for foliations on singular spaces. We first recall the notion of strongly regular foliation following \cite[Definition 1.9]{holmann}.

\begin{defn}
Let $\sG$ be a foliation of rank $r \ge 1$ on a normal complex variety. Say that $\sG$ is \textit{strongly regular
at $x \in X$} if $x$ has an open analytic neighborhood $U$ that is biholomorphic to 
$\mathbb{D}^r\times M$, where $\mathbb{D}$ is the complex open disk and $M$ is a germ of normal complex analytic variety, such that $\sG_{|U}$ is induced by the projection
$\mathbb{D}^r\times M \to M$. Say that $\sG$ is \textit{strongly regular} if $\sG$ is strongly regular at any point $x \in X$.
\end{defn}

\begin{rem}\label{rem:strongly_regular_infinitesimal}
By \cite[Lemma 1.3.2]{bogomolov_mcquillan01}, $\sG$ is strongly regular at $x \in X$ if and only if $\sG$ is locally free in a neighborhood of $x$ and the natural map $\Omega_X^{r} \to \sO_X(K_\sG)$ induced by the $r$-th wedge product
of the inclusion $\sG \into T_X$ is surjective at $x$.
\end{rem}

\begin{exmp}
If $Y$ and $Z$ are normal varieties and $\sG$ is the foliation on $X : = Y \times Z$ induced by the projection 
$Y \times Z \to Y$, then $\sG$ is strongly regular if and only if $Z$ is smooth. 
\end{exmp}

The notion of strong regularity is however not flexible enough to allow for applications.
The following notion of regularity for foliations adresses this issue
(\cite[Definition 3.5]{codim_1_del_pezzo_fols}). 

\begin{defn}\label{defn:regular}
Let $\sG$ be a foliation of rank $r\ge 1$ on a normal complex variety. 
The $r$-th wedge product
of the inclusion $\sG \into T_X$ gives rise to a non-zero map $\sO_X(-K_\sG) \to (\wedge^rT_X)^{**}$. 
We will refer to the dual map $\Omega_X^{[r]} \to \sO_X(K_\sG)$ as the \textit{Pfaff field} associated to $\sG$.

The \textit{singular locus} $S$ of $\sG$ is the closed subscheme of $X$ whose ideal sheaf is the image
of the induced map $\Omega_X^{[r]}\boxtimes \sO_X(-K_\sG) \to \sO_X$, which we will refer to as the \textit{twisted Pfaff field} associated to $\sG$. 

We say that $\sG$ is \textit{weakly regular at} $x \in X$ if $x \not\in S$.
We say that $\sG$ is \textit{weakly regular} if $S=\emptyset$.
\end{defn}

\begin{exmp}
Let $X$ be a normal variety, and consider $\sG=T_X$. Then $\sG$ is weakly regular.
\end{exmp}

A strongly regular foliation is obviously weakly regular in the sense of Definition \ref{defn:regular} above. The converse is true if $X$ is smooth by Frobenius' theorem.
Other examples of weakly regular foliations are provided by the following results.

\begin{prop}\label{prop:generic_smoothness}
Let $X$ be a normal complex variety, and let $\psi \colon X \to Y$ be a dominant morphism onto a variety $Y$. Let  
$\sG$ be the foliation on $X$ induced by $\psi$. 
Then $\sG$ is weakly regular over the generic point of $Y$.
\end{prop}

\begin{proof}
Note that we may assume without loss of generality that $Y$ is smooth.
By \cite[Proposition III.10.6]{hartshorne77}, the tangent map $T\psi\colon T_X \to \psi^*T_Y$ is surjective along a general fiber $F$ of 
$\psi$. The claim then follows from Lemma \ref{lemma:regular_tangent_versus_conormal} below.
\end{proof}

\begin{lemma}\label{lemma:regular_tangent_versus_conormal}
Let $X$ be a normal variety of dimension $n\ge 2$ and let $\sG\subsetneq T_X$ be a foliation of rank $r \ge 1$ on $X$. Set $q:=n-r$. Then $\sG$ is weakly regular if and only if the map 
$\eta\colon (\wedge^q T_X)^{**}\boxtimes\det\sN_\sG^* \to \sO_X$ induced by the $q$-th wedge product of the quotient map 
$T_X \to \sN_\sG$ 
is surjective.
\end{lemma}
\begin{proof}
The wedge product of differential forms on $X_\textup{reg}$ induces an isomorphism
of reflexive sheaves 
$$\Omega_X^{[r]} \cong \big(\wedge^{q}T_X\big)^{**}\boxtimes \sO_X(K_X).$$
The canonical isomorphism $\sO_X(K_\sG) \cong \sO_X(K_X)\boxtimes \det\sN_\sG$ then yields 
$$\Omega_X^{[r]}\boxtimes\sO_X(-K_\sG) \cong \big(\wedge^{q}T_X\big)^{**}\boxtimes \det\sN_\sG^*.$$
One readily checks that the map $\Omega_X^{[r]}\boxtimes\sO_X(-K_\sG) \to \sO_X$ induced by $\eta$ is the twisted Pfaff field associated with $\sG$. This shows the lemma.
\end{proof}

\begin{lemma}\label{lemma:direct_summand_regular}
Let $X$ be a normal variety, and let $\sG$ be a foliation on $X$. Suppose that there exists a distribution $\sE$ on $X$ such that 
$T_X=\sG\oplus \sE$. Then $\sG$ is weakly regular.
\end{lemma}

\begin{proof}
Set $r:=\textup{rank}\,\sG$. Observe that $\det\sG^*\cong \sO_X(K_\sG)$ is a direct summand of 
$\Omega_X^{[r]}$ and that the twisted Pfaff field  
$\Omega_X^{[r]}\boxtimes \sO_X(-K_\sG) \to \sO_X$ associated to $\sG$
is induced by the projection $\Omega_X^{[r]} \to \det\sG^*$.
This immediately implies that $\sG$ is weakly regular.
\end{proof}

\subsection{Elementary properties}
The following lemma says that a weakly regular foliation $\sG$ has mild singularities if $K_\sG$ is Cartier.

\begin{lemma}\label{lemma:regular_versus_canonical}
Let $X$ be a normal complex variety with klt singularities, and let $\sG$ be a foliation on $X$. 
Suppose that $K_\sG$ is Cartier.
If $\sG$ is weakly regular, then it has canonical singularities.
\end{lemma}

\begin{proof}Let $Z$ be a normal variety, and let
$\beta\colon Z \to X$ be a birational projective morphism. 
Let $\eta_X\colon\Omega_X^{[r]} \twoheadrightarrow \sO_X(K_\sG)$ and $\eta_Z\colon \Omega_Z^{[r]} \to \sO_Z(K_{\beta^{-1}\sG})$ be the Pfaff fields associated to $\sG$ and $\beta^{-1}\sG$ respectively.
Recall from paragraph \ref{subsection:pull-back_morphims} that there exists a morphism of sheaves
$d_\textup{refl}\beta\colon\beta^*\Omega_X^{[r]} \to \Omega_Z^{[r]}$ that agrees with the usual pull-back morphism of K\"ahler differentials wherever this makes sense.
One then readily checks that we have a commutative diagram 

\begin{center}
\begin{tikzcd}[row sep=large]
\beta^*\Omega_X^{[r]} \ar[r, "{d_\textup{refl}\beta}"]\ar[d, twoheadrightarrow, "{\beta^*\eta_X}"'] & \Omega_Z^{[r]}\ar[d, "{\eta_Z}"] \\
\beta^*\sO_X(K_\sG) \ar[r] & \sO_Z(K_{\beta^{-1}\sG}).
\end{tikzcd}
\end{center}

\noindent In particular, there is a $\beta$-exceptional effective divisor $E$ on $Z$ such that $K_{\beta^{-1}\sG}=\beta^*K_\sG+E$, proving the lemma.
\end{proof}

\begin{rem}
We will show that the converse is also true if $\sG$ is algebraically integrable with $K_\sG\equiv 0$ (see Corollary \ref{cor:canonical_versus_regular}).
\end{rem}

Example \ref{example:product2} below shows that Lemma \ref{lemma:regular_versus_canonical} is wrong if one drops the assumption that $K_\sG$ is Cartier.

\begin{exmp}\label{example:product2}
Let $Y$ and $Z$ be normal varieties, and let $\sG$ be the foliation on $X : = Y \times Z$ induced by the projection 
$Y \times Z \to Y$. Then $\sG$ is weakly regular by Lemma \ref{lemma:direct_summand_regular}. 
But $\sG$ has canonical singularities if and only if $Z$ has canonical singularities by
Example \ref{example:product}.
\end{exmp}

Next, we analyze the behaviour of weakly regular foliations with respect to smooth morphisms,
quasi-finite maps, and birational modifications.

\begin{lemma}\label{lemma:regular_local_etale}
Let $\pi \colon Y \to X$ be a surjective \'etale morphism of normal varieties, and let $\sG$ be a foliation on $X$. Then $\sG$ is weakly regular if and only if so is $\pi^{-1}\sG$.
\end{lemma}

\begin{proof}
Recall that the pull-back of a reflexive sheaf by a flat morphism is reflexive as well by \cite[Proposition 1.8]{hartshorne80}.
Note also that $\pi^*\big(\Omega_X^{[r]}\boxtimes\sO_X(-K_\sG)\big)\cong \Omega_Y^{[r]}\boxtimes\sO_Y(-\pi^*K_\sG)$ since both 
are reflexive sheaves and agree on $\pi^{-1}(X_\textup{reg})=Y_{\textup{reg}}$. Let $\eta\colon \Omega_X^{[r]}\boxtimes\sO_X(-K_\sG) \to \sO_X$ be the twisted Pfaff field associated to $\sG$. One 
readily checks that the induced map 
$\pi^*\eta\colon  \Omega_Y^{[r]}\boxtimes\sO_Y(-\pi^*K_\sG) \to \sO_Y$
is the twisted Pfaff field associated to $\pi^{-1}\sG$. The lemma follows since $\pi$ is a faithfully flat morphism.
\end{proof}

\begin{prop}\label{prop:regular_quasi_etale}
Let $X$ be a normal complex variety, let $\sG$ be a foliation on $X$, and let $\pi\colon Y \to X$ be a quasi-finite dominant morphism.
Suppose that any codimension one irreducible component of the branch locus of $\pi$ is $\sG$-invariant. Then the following holds.
\begin{enumerate}
\item If $\sG$ is weakly regular, then so is $\pi^{-1}\sG$.
\item Suppose in addition that $\pi$ is finite and surjective. If $\pi^{-1}\sG$ is weakly regular, then so is $\sG$. 
\end{enumerate}
\end{prop}

\begin{proof}Set $r:=\textup{rank}\,\sG$, and let 
$\eta_X\colon \Omega_X^{[r]}\boxtimes\sO_X(-K_\sG) \to \sO_X$ 
and 
$\eta_Y\colon \Omega_Y^{[r]}\boxtimes\sO_Y(-K_{\pi^{-1}\sG}) \to \sO_Y$
be the twisted Pfaff fields associated to $\sG$ and $\pi^{-1}\sG$ respectively.

\medskip

Suppose first that $\sG$ is weakly regular. By Lemma \ref{lemma:pull_back_fol_and_finite_cover}, we have $K_{\pi^{-1}\sG}\sim_\mathbb{Z}\pi^*K_\sG$. This implies that
$$\pi^{[*]}\big(\Omega_X^{[r]}\boxtimes\sO_X(-K_\sG)\big)\cong \pi^{[*]}\Omega_X^{[r]}\boxtimes\sO_Y(-K_{\pi^{-1}\sG}).$$
One then readily checks that we have a commutative diagram

\begin{center}
\begin{tikzcd}[row sep=large]
\pi^*\big(\Omega_X^{[r]}\boxtimes\sO_X(-K_\sG)\big) \ar[r]\ar[d, twoheadrightarrow, "{\pi^*\eta_X}"'] & 
\Omega_Y^{[r]}\boxtimes\sO_Y(-K_{\pi^{-1}\sG})\ar[d, "{\eta_Y}"] \\
\sO_Y \ar[r, equal] & \sO_Y.
\end{tikzcd}
\end{center}

\noindent This shows that $\pi^{-1}\sG$ is a weakly regular foliation, proving Item (1).

\medskip

Suppose from now on that $\pi$ is a finite cover, and that $\pi^{-1}\sG$ is weakly regular.
Let $\gamma \colon Y_1 \to Y$ be a finite cover such that the induced cover $\pi_1 \colon Y_1 \to X$ is Galois, with Galois group $G$. We may also assume that $\pi_1$ is quasi-\'etale away from the branch locus of $\pi$.
By Item (1) applied to $\gamma$, $\pi_1^{-1}\sG$ is weakly regular as well.
Thus, we may assume without loss of generality that $\pi$ is Galois with Galois Group $G$. By Lemma \ref{lemma:pull_back_fol_and_finite_cover}, the tangent map $T\pi$ induces an isomorphism 
$\pi^{-1}\sG \cong (\pi^*\sG)^{**}$. It follows that  
$\sO_Y(K_{\pi^{-1}\sG}) \cong \sO_Y(\pi^*K_\sG)$ as $G$-sheaves. 
Note that $\eta_Y$ is obviously $G$-equivariant.
By \cite[Lemma A.4]{greb_kebekus_kovacs_peternell10} and \cite[Theorem 2]{brion_differential_form}, we have 
$$\Big(\pi_*\big(\Omega_Y^{[r]}\boxtimes\sO_Y(-K_{\pi^{-1}\sG})\big)\Big)^G\cong 
\Omega_X^{[r]}\boxtimes\sO_Y(-K_{\sG}).$$
It follows that the map $\eta_Y^G \colon \Omega_X^{[r]}\boxtimes\sO_Y(-K_{\sG}) \to \sO_X$
induced by $\eta_Y$ is the twisted Pfaff field associated to $\sG$.
From \cite[Lemma A.3]{greb_kebekus_kovacs_peternell10}, we see that $\eta_X=\eta_Y^G$ is surjective. This shows that 
$\sG$ is weakly regular, completing the proof of the proposition.
\end{proof}

The following is an immediate consequence of Proposition \ref{prop:regular_quasi_etale}.

\begin{cor}\label{cor:regular_quasi_etale}
Let $X$ be a normal complex variety, let $\sG$ be a foliation on $X$, and let $\pi\colon Y \to X$ be a quasi-\'etale cover.
Then $\sG$ is weakly regular if and only if so is $\pi^{-1}\sG$.
\end{cor}

\begin{lemma}\label{lemma:regular_bir_crepant_map}
Let $\pi \colon Y \to X$ be a projective birational morphism of normal complex varieties, and let $\sG$ be a foliation on $X$.
Suppose that $K_\sG$ is Cartier, and that $K_{\pi^{-1}\sG}\sim_\mathbb{Z}\pi^*K_\sG$. Suppose furthermore that $X$ has klt singularities.
If $\sG$ is weakly regular, then so is  
$\pi^{-1}\sG$.
\end{lemma}

\begin{proof}
Let $\eta_X\colon\Omega_X^{[r]} \twoheadrightarrow \sO_X(K_\sG)$ and $\eta_Y\colon \Omega_Y^{[r]} \to \sO_Y(K_{\pi^{-1}\sG})$ be the Pfaff fields associated to $\sG$ and $\pi^{-1}\sG$ respectively.
Recall from the proof of Lemma \ref{lemma:regular_versus_canonical} that there is a commutative diagram 

\begin{center}
\begin{tikzcd}[row sep=large]
\pi^*\Omega_X^{[r]} \ar[r, "{d_\textup{refl}\pi}"]\ar[d, twoheadrightarrow, "{\pi^*\eta_X}"'] & \Omega_Y^{[r]}\ar[d, "{\eta_Y}"] \\
\pi^*\sO_X(K_\sG) \ar[r, "{\sim}"] & \sO_Y(K_{\pi^{-1}\sG}).
\end{tikzcd}
\end{center}
This immediately implies that $\pi^{-1}\sG$ is weakly regular, proving the lemma.
\end{proof}

\begin{lemma}\label{lemma:properties:regular}
Let $\pi \colon Y \to X$ be a dominant morphism of normal complex varieties, and let $\sG$ be a foliation on $X$. 
\begin{enumerate}
\item Suppose that $Y = X \times Z$ and that $\pi$ is the projection onto $X$. Then 
$K_{\pi^{-1}\sG}\sim_\mathbb{Z}\pi^*K_\sG+K_{Y/X}$, and $\sG$ is weakly regular if and only if so is $\pi^{-1}\sG$.
\item If $\pi$ is a smooth morphism, then $K_{\pi^{-1}\sG}\sim_\mathbb{Z}\pi^*K_\sG+K_{Y/X}$. Moreover,
$\sG$ is weakly regular if and only if so is $\pi^{-1}\sG$.
\item Suppose that $X$ has klt singularities, and that $\pi$ is a small projective birational map. Suppose on addition that 
$K_\sG$ is $\mathbb{Q}$-Cartier. Then $K_{\pi^{-1}\sG}\sim_\mathbb{Q}\pi^*K_\sG$. Moreover,
if $\sG$ is weakly regular, then so is $\pi^{-1}\sG$.
\end{enumerate}
\end{lemma}

\begin{proof}Set $r:=\textup{rank}\,\sG$, and let $\eta_X\colon \Omega_X^{[r]}\boxtimes\sO_X(-K_\sG) \to \sO_X$ be the twisted Pfaff field associated to $\sG$.

\medskip

Suppose first that $Y = X \times Z$ and that $\pi$ is the projection onto $X$. Denote by $p$ the projection onto $Z$, and set $m:=\dim Z$.
Recall that the pull-back of a reflexive sheaf by a flat morphism is reflexive as well by \cite[Proposition 1.8]{hartshorne80}.
Then $$\pi^{-1}\sG \cong T_{Y/X} \oplus \pi^*\sG \cong p^*T_Z \oplus \pi^*\sG,$$ 
and hence $K_{\pi^{-1}\sG}\sim_\mathbb{Z}\pi^*K_\sG+K_{Y/X}\sim_\mathbb{Z}\pi^*K_\sG+p^*K_Z$.

We have $$\Omega_Y^{[r+m]} \cong \bigoplus_{i+j=r+m}\pi^{*}\Omega_X^{[i]}\boxtimes p^*\Omega_{Z}^{[j]}$$
and the twisted Pfaff field associated to $\pi^{-1}\sG$ is the composed map
\begin{center}
\begin{tikzcd}[cramped]
\Omega_Y^{[r+m]}\boxtimes \sO_X(-K_{\pi^{-1}\sG})
\ar[r, twoheadrightarrow] & 
\Big(\pi^{*}\Omega_X^{[r]}\boxtimes p^*\Omega_{Z}^{[m]}\Big)\boxtimes \sO_X(-K_{\pi^{-1}\sG})
\cong \pi^*\big(\Omega_X^{[r]}\boxtimes\sO_X(-K_\sG)\big)
\ar[r, "\pi^*\eta_X"] & 
\pi^*\sO_X\cong\sO_Y.
\end{tikzcd}
\end{center}
It follows that $\sG$ is weakly regular if and only if so is $\pi^{-1}\sG$
since $\pi$ is a faithfully flat morphism.

\medskip

Suppose now that $\pi$ is smooth. Set $m:=\dim Y - \dim X$,
$X^\circ:=X_{\textup{reg}}$, and $Y^\circ:=\pi^{-1}(X^\circ)\subseteq Y_{\textup{reg}}$. We have an exact sequence
of vector bundles
$$0 \to T_{Y^\circ/X^\circ}\to (\pi^{-1}\sG)_{|Y^\circ} \to (\pi_{|Y^\circ})^*\sG_{|X^\circ} \to 0$$
and thus 
$K_{\pi^{-1}\sG}\sim_\mathbb{Z}\pi^*K_\sG+K_{Y/X}$
since $Y^\circ$ has complement of codimension at least two in $Y$.
We proceed to show that $\sG$ is weakly regular if and only if so is $\pi^{-1}\sG$. The statement is local on $X$ for the \'etale topology
by Lemma \ref{lemma:regular_local_etale}.
Thus, we may assume that 
$Y = X \times \mathbb{A}^m$ and that $\pi$ is given by the projection $Y = X \times \mathbb{A}^m \to X$.
The claim then follows from the previous case.

\medskip

Suppose finally that $X$ has klt singularities, and that $\pi$ is a small projective birational map. Suppose in addition that
$K_\sG$ is $\mathbb{Q}$-Cartier. We clearly have $K_{\pi^{-1}\sG}\sim_\mathbb{Q}\pi^*K_\sG$. 
Replacing $X$ by an open subset, if necessary, we may assume that $K_\sG$ is torsion. It follows that $K_{\pi^{-1}\sG}$ is torsion as well. By Corollary \ref{cor:regular_quasi_etale}, replacing $Y$ and $X$ by the associated cyclic quasi-\'etale covers (see \cite[Definition 2.52]{kollar_mori}), we may also assume that $K_\sG$ is Cartier and that 
$K_{\pi^{-1}\sG}\sim_\mathbb{Z}\pi^*K_\sG$. The statement then follows from Lemma \ref{lemma:regular_bir_crepant_map}.
\end{proof}

\begin{rem}\label{rem:bir_small_cartier}
In the setup of Lemma \ref{lemma:properties:regular} (2), suppose in addition that $K_\sG$ is Cartier. Then 
$K_{\pi^{-1}\sG}\sim_\mathbb{Z}\pi^*K_\sG$.
\end{rem}

\begin{lemma}\label{lemma:regular_family_leaves}
Let $X$ be a normal complex variety with klt singularities, and let $\sG$ be an algebraically integrable foliation on $X$ with canonical singularities. 
Let $\psi\colon Z \to Y$ be the family of leaves, and let $\beta\colon Z \to X$ be the natural morphism (see \ref{family_leaves}). 
If $\sG$ is weakly regular, then so is $\beta^{-1}\sG$.
\end{lemma}

\begin{proof}
By Lemma \ref{lemma:canonical_foliation_versus_lc_pairs}, we have $K_{\beta^{-1}\sG}\sim_\mathbb{Q}\beta^*K_\sG$.
Note that the statement is local on $X$.
Let $m$ be a positive integer and let $X ^\circ \subseteq X$ be a dense open subset such that 
$\sO_{X^\circ}\big(m{K_\sG}_{|X^\circ}\big)\cong \sO_{X^\circ}$. 
Let $f^\circ \colon X_1^\circ \to X^\circ$ be the associated cyclic cover, which is quasi-\'etale (see \cite[Definition 2.52]{kollar_mori}), and let $Z_1^\circ$ be the normalization of the product $Z^\circ \times_{X^\circ} X_1^\circ$, where 
$Z^\circ:=\beta^{-1}(X^\circ)$. Let also $\beta_1^\circ\colon Z_1^\circ \to X_1^\circ$ 
and $g^\circ\colon Z_1^\circ \to Z^\circ$
denote the natural morphisms. Recall from subsection \ref{family_leaves} that 
$K_{(\beta_1^\circ)^{-1}(f^\circ)^{-1}(\sG_{|X^\circ})}\sim_\mathbb{Z} (\beta_1^\circ)^* K_{(f^\circ)^{-1}(\sG_{|X^\circ})}$
and that the support of the ramification divisor $R(g^\circ)$ of $g^\circ$ must be $(\beta_1^\circ)^{-1}(f^\circ)^{-1}(\sG_{|X^\circ})$-invariant. 
By construction, $K_{(f^\circ)^{-1}(\sG_{|X^\circ})}$ is Cartier.
Moreover, $(f^\circ)^{-1}(\sG_{|X^\circ})$ is weakly regular by Corollary \ref{cor:regular_quasi_etale}.
Applying Lemma \ref{lemma:regular_bir_crepant_map}, we see that $(\beta_1^\circ)^{-1}(f^\circ)^{-1}(\sG_{|X^\circ})=(g^\circ)^{-1}\big((\beta^{-1}\sG)_{|Z^\circ}\big)$ is weakly regular.
The statement then follows from Proposition \ref{prop:regular_quasi_etale}.
\end{proof}

\subsection{Criteria for weak regularity} Let $\sG$ be a foliation with numerically trivial canonical class on a complex projective manifold. Suppose that $\sG$ has a compact leaf. Then Theorem 5.6 in \cite{lpt} asserts that $\sG$ is regular and that there exists a foliation on $X$ transverse to $\sG$ at any point in $X$. 
In this paragraph, we extend this result to mildly singular varieties (see Corollary \ref{corollary:compact_leaf_holomorphic_form}). We also show that algebraically integrable foliations with mild singularities and numerically trivial canonical class are weakly regular (see Corollary \ref{cor:canonical_versus_regular}). Finally, we provide another criterion for regularity of foliations (see Proposition \ref{prop:criterion_regularity_2}).

\medskip

We will need the following easy observations.

\begin{lemma}\label{lemma:leaf_pfaff_field}
Let $X$ be a normal complex variety, and let $\sG$ be a foliation of rank $r$ on $X$. Suppose that $K_\sG$ is Cartier, and let $\eta\colon \Omega_X^{[r]} \to \sO_X(K_\sG)$ be the Pfaff field associated to $\sG$. Let $L \subset X$ be a subvariety which is not entirely contained in the union of the singular loci of $X$ and $\sG$. Suppose in addition that $\dim L = r$. Then the following holds.
\begin{enumerate}
\item The variety $L \cap X_\textup{reg}$ is a leaf of $\sG_{|X_\textup{reg}}$ if and only if the composed map ${\Omega_X^{r}}_{|L} \to {\Omega_X^{[r]}}_{|L} \to {\sO_X(K_\sG)}_{|L}$ factors through the natural map ${\Omega_X^r}_{|L} 
\twoheadrightarrow \Omega_L^r$.
\item Suppose that $L \cap X_\textup{reg}$ is a leaf of $\sG$, and let $F$ be the normalization of $L$. Denote by $n \colon F \to X$ the natural morphism. Then there is a commutative diagram
\begin{center}
\begin{tikzcd}
n^*\Omega_X^r \ar[r]\ar[d, "{dn}"'] & n^*\Omega_X^{[r]} \ar[r, "{n^*\eta}"] & n^*\sO_X(K_\sG) \ar[d, equal]\\
\Omega_F^r \ar[r] & \Omega_F^{[r]} \ar[r] & n^*\sO_X(K_\sG).
\end{tikzcd}
\end{center}
\end{enumerate}
\end{lemma}

\begin{proof}
Item (1) follows easily from \cite[Lemma 2.7]{fano_fols} using the fact that ${\sO_X(K_\sG)}_{|L}$ is torsion-free.
Item (2) follows from Item (1) and \cite[Proposition 4.5]{adk08}.
\end{proof}

\begin{lemma}\label{lemma:non-zero_class}
Let $X$ be a normal complex projective variety, and let $H$ be an ample divisor on $X$. For any integer $1 \le r \le \dim X$, the image of 
$c_1(H)^r \in H^r\big(X,\Omega_X^r\big)$ under the natural map 
$H^r\big(X,\Omega_X^r\big) \to H^r\big(X,\Omega_X^{[r]}\big)$ is non-zero.
\end{lemma}

\begin{proof}
In order to prove the lemma, it suffices to consider the case when $r = \dim X$.
Let $\beta\colon Z \to X$ be a resolution of $X$. The image of 
$c_1(H)^{\dim X} \in H^{\dim X}\big(X,\Omega_X^{\,\dim X}\big)$ under the map $H^{\dim X}\big(X,\Omega_X^{\,\dim X}\big) \to H^{\dim X}\big(Z,\Omega_Z^{\,\dim X}\big)$
is non-zero, and hence $c_1(H)^{\dim X}$ is non-zero as well. On the other hand, the kernel and the cokernel of the natural map
$\Omega_X^{\,\dim X} \to \Omega_X^{[{\dim X}]}$ are supported on closed subsets of codimension at least two. It follows that the natural map $H^{\dim X}\big(X,\Omega_X^{\,\dim X}\big) \to H^{\dim X}\big(X,\Omega_X^{[\dim X]}\big)$ is an isomorphism, completing the proof of the lemma.
\end{proof}

\begin{prop}\label{prop:criterion_regularity}
Let $X$ be a normal complex projective variety, and let $\sG$ be a foliation of rank $r$ on $X$ with $K_\sG$ Cartier and $K_\sG \equiv 0$. Let $L \subset X$ be a proper subvariety which is not entirely contained in the union of the singular loci of $X$ and $\sG$. Suppose that 
$L \cap X_{\textup{reg}}$ is a leaf of $\sG_{|X_{\textup{reg}}}$. Let $F$ be the normalization of $L$, and denote by $n \colon F \to X$ the natural morphism. Suppose furthermore that the map $\eta_F\colon \Omega_F^{[r]} \to n^*\sO_X(K_\sG)$ given by Lemma \ref{lemma:leaf_pfaff_field} is an isomorphism, and that $F$ has rational singularities.
Then $\sG$ is weakly regular, and there exists a decomposition $T_X \cong \sG\oplus \sE$
of $T_X$ into involutive subsheaves.
\end{prop}

\begin{proof}
Let $\beta\colon Z \to X$ be an embedded resolution of $L$, and let $T$ be the strict transform of $L$ in $Z$.
Observe that $\beta_{|T}\colon T \to L$ factors through $F \to L$, and denote by $\gamma\colon T \to F$ the induced map. 

Let $H$ be an ample divisor on $X$. Let also $\eta\colon \Omega_X^{[r]} \to \sO_X(K_\sG)$ be the Pfaff field associated to $\sG$.
Consider the image $c$ of 
$c_1(H)^r \in H^r\big(X,\Omega_X^r\big)$
under the composed map 
$$H^r\big(X,\Omega_X^r\big) \to H^r\big(X,\Omega_X^{[r]}\big) \to H^r\big(X,\sO_X(K_\sG)\big).$$
We will show that $c \neq 0$ and that $H^r(n^*)(c)\neq 0$. By Lemma \ref{lemma:leaf_pfaff_field}, we have a commutative diagram

\begin{center}
\begin{tikzcd}
H^r(X,\Omega_X^r) \ar[r]\ar[d, "{H^r(dn)}"'] & H^r\big(X,\Omega_X^{[r]}\big) 
\ar[r, "{H^r(\eta)}"] & H^r\big(X,\sO_X(K_\sG)\big) \ar[d, "{H^r(n^*)}"]\\
H^r(F,\Omega_F^r) \ar[r] & H^r\big(F,\Omega_F^{[r]}\big) \ar[r] & H^r\big(F,n^*\sO_X(K_\sG)\big).
\end{tikzcd}
\end{center}
By assumption, the map $H^r\big(F,\Omega_F^{[r]}\big) \to H^r\big(F,n^*\sO_X(K_\sG)\big)$ is an isomorphism.
On the other hand, by Lemma \ref{lemma:non-zero_class} above, the image of 
$c_1\big(H_{|F}\big)^r \in H^r\big(F,\Omega_F^r\big)$
under the map 
$H^r\big(F,\Omega_F^r\big) \to H^r\big(F,\Omega_F^{[r]}\big)$
is non-zero. This immediately implies that $c\neq 0$ and 
that $H^r(n^*)(c)\neq 0$.

Consider the commutative diagram

\begin{center}
\begin{tikzcd}
H^r\big(Z,\beta^*\sO_X(K_\sG)\big) \ar[r]   & H^r\big(T,(n\circ\gamma)^*\sO_X(K_\sG)\big) \\
H^r\big(X,\sO_X(K_\sG)\big) \ar[u]\ar[r, "{H^r(n^*)}"] &  H^r\big(F,n^*\sO_X(K_\sG)\big)\ar[u].
\end{tikzcd}
\end{center}
Since $F$ has rational singularities, the morphism $H^r\big(F,n^*\sO_X(K_\sG)\big) \to H^r\big(T,(n\circ\gamma)^*\sO_X(K_\sG)\big)$ is an isomorphism. This implies that the image $c_Z$ of $c$ under the map 
$H^r\big(X,\sO_X(K_\sG)\big) \to H^r\big(Z,\beta^*\sO_X(K_\sG)\big)$ is non-zero.

On the other hand, by Hodge symmetry with coefficients in local systems, there are natural isomorphisms 
\begin{multline*}
H^r\big(Z,\beta^*\sO_X(K_\sG)\big) \cong \overline{H^0\big(Z,\Omega_Z^{r}\otimes\beta^*\sO_X(-K_\sG)\big)}
\\
\text{and}\quad
H^r\big(T,(n\circ\gamma)^*\sO_X(K_\sG)\big) \cong \overline{H^0\big(T,\Omega_T^{r}\otimes(n\circ\gamma)^*\sO_X(-K_\sG)\big)}.
\end{multline*}
It follows that $\omega_Z:=\wb{c_Z}\in H^0\big(Z,\Omega_Z^{r}\otimes\beta^*\sO_X(-K_\sG)\big)$ is a (closed) twisted $r$-form that restricts to a non-zero twisted $r$-form $\alpha_T \in H^0\big(T,\Omega_T^{r}\otimes (n\circ\gamma)^*\sO_X(-K_\sG)\big)$.
Let $\omega \in H^0\big(X,\Omega_X^{[r]}\otimes\sO_X(-K_\sG)\big)$ be the (closed) twisted reflexive $r$-form on $X$
induced by $\omega_Z$, and let $\alpha_F \in H^0\big(F,\Omega_F^{r}\otimes n^*\sO_X(-K_\sG)\big)$
be the non-zero twisted reflexive $r$-form on $F$ induced by $\alpha_T$. By construction, the contraction 
$\eta (\omega)$ is a regular function that restricts to the non-zero regular function 
$\eta_F (\alpha_F)$ on $F$. It follows that $\eta (\omega)$ is constant. 
This shows that there is a decomposition
$T_X\cong \sG \oplus \sE$ into involutive subsheaves. Then $\sG$ is weakly regular by Lemma \ref{lemma:direct_summand_regular}, completing the proof of the proposition.
\end{proof}

\begin{cor}\label{corollary:compact_leaf_holomorphic_form}
Let $X$ be a normal complex projective variety with klt singularities, and let $\sG$ be a foliation on $X$. Suppose that $K_\sG$ is Cartier and that 
$K_\sG \equiv 0$. Let $L \subset X$ be a proper subvariety disjoint from the singular locus of $\sG$ and not contained in the singular locus of $X$. 
Suppose that 
$L \cap X_{\textup{reg}}$ is a leaf of $\sG_{|{\textup{reg}}}$. 
Suppose furthermore that the normalization $F$ of $L$ has rational singularities.
Then $\sG$ is weakly regular and there is a decomposition $T_X \cong \sG\oplus \sE$
of $T_X$ into involutive subsheaves.
\end{cor}

\begin{proof}Denote by $n\colon F \to X$ the natural morphism, and let $\eta_F\colon \Omega_F^{[r]} \to n^*\sO_X(K_\sG)$ be the map given by Lemma \ref{lemma:leaf_pfaff_field}.
By Lemma \ref{lemma:leaf_pfaff_field} and paragraph \ref{subsection:pull-back_morphims},
we have a commutative diagram

\begin{center}
\begin{tikzcd}[row sep=large, column sep=large]
n^*\Omega_X^{[r]} \ar[d, "{d_\textup{refl}n}"']\ar[r, "n^*\eta"] & n^*\sO_X(K_\sG) \ar[d, equal]\\
\Omega_F^{[r]} \arrow[r, "{\eta_F}"'] & n^*\sO_X(K_\sG).
\end{tikzcd}
\end{center}
On the other hand, the map $n^*\eta$ is surjective by assumption. This immediately implies that $\eta_F$ is an isomorphism, so that Proposition \ref{prop:criterion_regularity} applies.
\end{proof}

\begin{cor}\label{cor:canonical_versus_regular}
Let $X$ be a normal complex projective variety with klt singularities, and let $\sG$ be an algebraically integrable foliation on $X$ with canonical singularities. Suppose that $K_\sG$ is Cartier and that 
$K_\sG \equiv 0$. 
Then $\sG$ is weakly regular and there is a decomposition $T_X \cong \sG\oplus \sE$
of $T_X$ into involutive subsheaves.
\end{cor}

\begin{proof}
Let $\psi\colon Z \to Y$ be the family of leaves, and let $\beta\colon Z \to X$ be the natural morphism (see \ref{family_leaves}). Let also $F$ be a general fiber of $\psi$. Note that $L:=\beta(F)$ is the closure of a leaf of $\sG$. By Lemma \ref{lemma:canonical_foliation_versus_lc_pairs}, $F$ has canonical singularities and $K_{\beta^{-1}\sG}\sim_\mathbb{Q}\beta^*K_\sG$. In particular, $F$ is the normalization of $L$, and it has rational singularities by \cite{elkik}. Moreover, $K_F \sim_\mathbb{Q} n^*\sO_X(K_\sG)$ by Example \ref{example:canonical_class_foliation} and the adjunction formula, where $n\colon F \to X$ denotes the restriction of $\beta$ to $F$. It follows that the map $\eta_F \colon \Omega_F^{[r]}\cong\sO_F(K_F) \to n^*\sO_X(K_\sG)$ given by Lemma \ref{lemma:leaf_pfaff_field} is an isomorphism, where $r$ denotes the rank of $\sG$. The conclusion then follows from Proposition \ref{prop:criterion_regularity}.
\end{proof}

Example \ref{example:canonical_not_regular} shows that Corollary \ref{cor:canonical_versus_regular} above is wrong if one drops the assumption that $K_\sG\equiv 0$.

\begin{exmp}\label{example:canonical_not_regular}
Let $B$ and $C$ be smooth projective curves, and let $X$ be the blow-up of $B\times C$ at some point. Let also $\sG$ be the foliation on $X$ induced by the natural morphism $X \to B$. Then $\sG$ has canonical singularities but it is not regular.
\end{exmp}

The proof of Proposition \ref{prop:criterion_regularity_2} below makes use of the following result, which might be of independent interest.

\begin{lemma}\label{lemma:criterion_regularity}
Let $X$ be a normal complex projective variety, and let $\sG$ be a foliation of rank $r$ on $X$ with $K_\sG$ Cartier and $K_\sG \equiv 0$. Let $H$ be a very ample Cartier divisor on $X$, and let $\Omega_X^{r} \to \sO_X(K_\sG)$ be map induced by the Pfaff field associated to $\sG$. 
Let $\beta\colon Z \to X$ be a resolution of singularities.
Suppose that the image $c$ of $c_1(H)^r \in H^r\big(X,\Omega_X^{r}\big)$ under the composed map 
$$H^r\big(X,\Omega_X^{r}\big) \to H^r\big(X,\sO_X(K_\sG)\big)\to H^r\big(Z,\beta^*\sO_X(K_\sG)\big)$$ is non-zero. Then $\sG$ is weakly regular, and there exists a decomposition $T_X \cong \sG\oplus \sE$
of $T_X$ into involutive subsheaves.
\end{lemma}

\begin{proof}[Proof of Lemma \ref{lemma:criterion_regularity}]The proof is similar to that of 
\cite[Proposition 2.7.1]{dps5}. 

\medskip

Set $n:=\dim X$. The linear system $|H|$ embeds $X$ into $\mathbb{P}^N$ for some positive integer $N$. 
Denote by $\gamma\colon Z \to \mathbb{P}^N$ the natural map. 
Let $(U_i)_{i\in I}$ be a finite covering of $\mathbb{P}^N$
by open sets such that ${\sO_X(K_\sG)}_{|X_i}\cong \sO_{X_i}$, where $X_i:=U_i \cap X$.
Denote by $v_i \in H^0(X_i,\wedge^rT_{X_i})$ a $r$-field defining $\sG_{|X_i}$, and let
$u_i \in H^0(U_i,\wedge^rT_{U_i})$ such that ${u_i}_{|X_i}=v_i \in H^0(X_i,\wedge^rT_{X_i}) \subset 
H^0\big(X_i,{\wedge^rT_{U_i}}_{|X_i}\big)$. Let $\omega_{\textup{FS}}$ be the Fubini-Study form on $\mathbb{P}^N$, and denote by
$\omega_i$ its restriction to $U_i$.
The pull-back $\eta_i$ on $\gamma^{-1}(U_i)$ of the contraction $u_i \lrcorner \,\omega_i^r$ of $\omega_i^r$ by $v_i$ is a $\wb{\partial}$-closed $(0,r)$-form. Moreover, the $\eta_i$ glue to give a 
$\wb{\partial}$-closed $(0,r)$-form $\eta$ with coefficients in the unitary flat line bundle
$\beta^*\sO_X(K_\sG)$. By construction, $\eta$
represents $c$ if $H^r\big(Z,\beta^*\sO_X(K_\sG)\big)$ is identified with the corresponding Dolbeault cohomology group.

By Hodge symmetry with coefficients in local systems, there exists a holomorphic $r$-form $\alpha$ with values in $\beta^*\sO_X(-K_\sG)$ such that $\{\wb{\alpha}\}=c \in H^r\big(Z,\beta^*\sO_X(K_\sG)\big)$.
In particular, there exists
a $(0,r-1)$-form $\xi$ with values in $\beta^*\sO_X(K_\sG)$ such that 
$\wb{\alpha} = \eta +\wb{\partial}\xi$. Note that $\alpha$ and $\eta$ are harmonic forms with respect to any K\"{a}hler form.
In particular, $\wb{\alpha}$ and $\eta$ are closed.

Set $X^\circ:=X\setminus \beta\big(\textup{Exc}\,\beta\big)$ and let 
$v^\circ \in H^0\big(X^\circ,\wedge^rT_{X^\circ}\otimes\sO_{X^\circ}(K_{\sG_{|X^\circ}})\big)$ be a twisted $r$-field defining
$\sG_{|X^\circ}$. Let also $\alpha^\circ \in H^0\big(X^\circ,\Omega^r_{X^\circ}\otimes\sO_{X^\circ}(-K_{\sG_{|X^\circ}})\big)$
be the twisted $r$-form induced by $\alpha$ on $X^\circ$. Notice that the contraction $\alpha^\circ(v^\circ)$ is a regular function, and hence constant since $X^\circ$ has complement of codimension at least two in $X$. To prove the statement, it suffices to show that $\alpha^\circ(v^\circ)$ is non-zero (see Lemma \ref{lemma:direct_summand_regular}).

Let $\omega$ be the smooth closed semi-positive $(1,1)$-form on $Z$ induced by 
$\omega_{\textup{FS}}$. Since $\wb{\alpha}$ and $\eta$ are closed, we have
$$\int_Z \alpha \wedge \wb{\alpha} \wedge \omega^{n-r}=\int_Z \alpha \wedge \eta \wedge \omega^{n-r}.$$
By the Hodge-Riemann bilinear relations, there exists a complex number $C_1\neq 0$ such that the smooth $(n,n)$-form 
$$C_1{\alpha \wedge \wb{\alpha} \wedge \omega^{n-r}}_{|\beta^{-1}(X^\circ)}=C_1\alpha^\circ \wedge \wb{\alpha^\circ} \wedge \omega_{\textup{FS}}^{n-r}$$
is semi-positive and not identically zero since $\alpha \neq 0$ by assumption. It follows that 
$\int_Z \alpha \wedge \eta \wedge \omega^{n-r} \neq 0$.
On the other hand, a straightforward computation show that 
$$\int_Z \alpha \wedge \eta \wedge \omega^{n-r} = 
\int_{X^\circ} \alpha^\circ \wedge (v^\circ\lrcorner\, {\omega_\textup{FS}}_{|X^\circ}^r) \wedge {\omega_\textup{FS}}_{|X^\circ}^{n-r}=C_2\alpha^\circ(v^\circ),$$
for some complex number $C_2 \neq 0$.
This immediately implies that $\alpha^\circ(v^\circ)$ is non-zero, finishing the proof of the lemma.
\end{proof}

\begin{prop}\label{prop:criterion_regularity_2}
Let $X$ be a normal complex $\mathbb{Q}$-Gorenstein projective variety, and let $\sG$ be a codimension one foliation on $X$ with $K_\sG$ Cartier and $K_\sG\equiv 0$. Suppose that $X$ is smooth in codimension two with rational singularities, and that 
$K_X\cdot H^{\dim X-1}\neq 0$ for some ample Cartier divisor $H$ on $X$. Suppose furthermore that there exists an open 
set $X^\circ\subseteq X_{\textup{reg}}$ with complement of codimension at least three such that
$\sG_{|X^\circ}$ is defined by closed holomorphic $1$-forms with zero set of codimension at least two locally for the analytic topology.
Then $\sG$ is weakly regular, and there exists a decomposition $T_X \cong \sG\oplus \sE$
of $T_X$ into involutive subsheaves.
\end{prop}

\begin{proof}Set $n:=\dim X$, and let $\eta\colon \Omega_X^{n-1} \to \sO_X(K_\sG)$ be map induced by the Pfaff field associated to $\sG$. Let $\beta\colon Z \to X$ be a resolution of singularities. 
The natural map $H^{n-1}\big(X,\sO_X(K_\sG)\big)\to H^{n-1}\big(Z,\beta^*\sO_X(K_\sG)\big)$ is an isomorphism
since $X$ has rational singularities. Thus, by Lemma \ref{lemma:criterion_regularity}, it suffices to show that the image of 
$c_1(H)^{n-1} \in H^{n-1}(X,\Omega_X^{n-1})$ under the map 
$H^{n-1}(\eta)\colon H^{n-1}\big(X,\Omega_X^{n-1}\big) \to H^{n-1}\big(X,\sO_X(K_\sG)\big)$ is non-zero.

Set $\sL:=\det \sN_\sG^*$.
By Lemma \ref{lemma:atiyah_class} below, 
the cohomology class $c_1(\sL_{|X^\circ})\in H^1(X^\circ,\Omega_{X^\circ}^1)$ lies in the image of the natural map
$H^1\big(X^\circ,\sL_{|X^\circ}\big)\to H^1(X^\circ,\Omega_{X^\circ}^1)$. 
Recall that rational singularities are Cohen-Macaulay. It follows that $\sO_X(K_X)$ and 
$\sO_X(K_X-K_\sG)\cong \sL$ are Cohen-Macaulay sheaves.
Then the restriction map $H^1(X,\sL) \to H^1\big(X^\circ,\sL_{|X^\circ}\big)$ is an isomorphism by
\cite[Theorem 1.14]{siu_trautmann}. Applying \cite[Theorem 1.14]{siu_trautmann} again together with 
\cite[Proposition 1.6]{hartshorne80}, we see that the restriction map
$H^1\big(X,\Omega_{X}^{[1]}\big) \to H^1(X^\circ,\Omega_{X^\circ}^1)$ is injective.
It follows that the image 
of $c_1(\sL)\in H^1\big(X,\Omega_X^{1}\big)$ under the natural map 
$c_1(\sL)\in H^1\big(X,\Omega_X^{1}\big) \to H^1\big(X,\Omega_X^{[1]}\big)$
is the image of a class $c \in H^1(X,\sL)$.

The same argument used in the proof of Lemma \ref{lemma:non-zero_class} shows that
$c_1(H)^{n-1} \otimes c_1(\sL)$ maps to a non-zero class $c_1(H)^{n-1} \cup c_1(\sL)\in H^n\big(X,\Omega_X^{[n]}\big)$ 
under the composed map

\begin{center}
\begin{tikzcd}[cramped]
H^{n-1}(X,\Omega_X^{n-1})\otimes H^1(X,\Omega_X^1)
\ar[r, "\bullet\,\cup\,\bullet"] & 
H^{n}(X,\Omega_X^n)
\ar[r] & 
H^{n}\big(X,\Omega_X^{[n]}\big),
\end{tikzcd}
\end{center}
using the assumptions that $K_X\cdot H^{\dim X-1}\neq 0$ and $K_\sG\equiv 0$. On the other hand, one readily checks that $c_1(H)^{n-1} \cup c_1(\sL)$ is the image of $c_1(H)^{n-1} \otimes c$ under the composed map
\begin{center}
\begin{tikzcd}[row sep=large, column sep=large]
H^{n-1}(X,\Omega_X^{n-1})\otimes H^1(X,\sL) \arrow[rr, "{H^{n-1}(\eta)\otimes\textup{Id}}"] &&  
H^{n-1}\big(X,\sO_X(K_\sG)\big) \otimes H^1(X,\sL)\arrow[d, phantom, ""{coordinate, name=Z}]
\arrow[d,"\bullet\,\cup\,\bullet\quad"',rounded corners,
to path=
{ -- ([xshift=2ex]\tikztostart.east)
|- (Z) [near end]
-| ([xshift=-2ex]\tikztotarget.west)\tikztonodes
-- (\tikztotarget)}
] \\
&& H^n(X,\sO_X(K_\sG)\otimes \sL)\cong H^n\big(X,\Omega_X^{[n]}\big).
\end{tikzcd}
\end{center}
This immediately implies that $H^{n-1}(\eta)\big(c_1(H)^{n-1}\big)\neq 0$, completing the proof of the proposition.
\end{proof}

The following result generalizes \cite[Corollary 3.4]{baum_bott70}.

\begin{lemma}\label{lemma:atiyah_class}
Let $X$ be a complex manifold, let $\sG \subset T_X$ be a codimension one foliation, and set $\sL:=\det\sN_\sG^*$. 
Suppose that $\sG$ is defined by closed holomorphic $1$-forms with zero set of codimension at least two locally for the analytic topology. 
Then the cohomology class $c_1(\sL)\in H^1(X,\Omega_X^1)$ lies in the image of the natural map
$H^1\big(X,\sL\big)\to H^1(X,\Omega_X^1)$.
\end{lemma}

\begin{proof}
Note that $\sL$ is a line bundle by \cite[Proposition 1.9]{hartshorne80}.
Let $(U_i)_{i\in I}$ be a covering of $X$ by analytically open sets such that 
$\sG_{|U_i}$ is defined a closed $1$-form $\omega_i$ with zero set of codimension at least two. Then, we can write 
$\omega_i = g_{ij} \omega_j$ on 
$U_{ij}:=U_i\cap U_j$, where $g_{ij}$ is a nowhere vanishing holomorphic function on $U_{ij}$. The cocycle
$[(g_{ij})_{i,j}] \in H^1(X,\sO_X^{\times})$ then satisfies $[(g_{ij})_{i,j}]=[\sL]$
since both classes agree away from the singular set of $\sG$ which has codimension at least two in $X$. Now, since $\omega_i$ and 
$\omega_j$ are closed, we must have $0=dg_{ij}\wedge \omega_j$ on $U_{ij}$. It follows that 
$$dg_{ij} \in H^0\big(U_{ij},\sL_{|U_{ij}}\big) \subset 
H^0\big(U_{ij},\Omega^1_{U_{ij}}\big)$$ since $\sL$ is saturated in $\Omega_X^1$ by 
\cite[Lemma 9.7]{fano_fols}. This easily implies that 
the cohomology class $c_1(\sL)=[(d\log\, g_{ij})_{ij}]\in H^1(X,\Omega_X^1)$
lies in the image of the natural map
$H^1\big(X,\sL\big)\to H^1(X,\Omega_X^1)$, proving the lemma.
\end{proof}

\subsection{Local structure in codimension $2$ of weakly regular rank $1$ folations} In the present section, we describe weakly regular rank $1$ foliations on surfaces with klt singularities.

\medskip

We first show that a weakly regular foliation given by a derivation with zero set of codimension at least two is strongly regular in codimension two.

\begin{prop}\label{prop:regular_foliation_klt_varieties}
Let $X$ be a normal complex variety with klt singularities, and let $\sG \subset T_X$ be a weakly regular foliation of rank one. 
Suppose in addition that $K_\sG$ is Cartier. Then there exists a closed subset $Z \subseteq X$ with $\textup{codim}\,Z \ge 3$ such that $\sG_{|X\setminus Z}$ is strongly regular.
\end{prop}

\begin{proof}Set $n:=\dim X$.
Recall from \cite[Proposition 9.3]{greb_kebekus_kovacs_peternell10} that klt spaces have quotient singularities in codimension two. Thus, we may assume without loss of generality that $X$ has quotient singularities. Given $x \in X$, 
we have $(X,x) \cong (\mathbb{C}^{n}/G,0)$ for some finite subgroup $G$ of $\textup{GL}(n,\mathbb{C})$ that does not contain any quasi-reflections. In particular, the quotient map $\pi\colon \mathbb{C}^{n} \to \mathbb{C}^{n}/G$ is \'etale outside of the singular set. The statement is local on $X$, hence we may shrink $X$ and assume that
there exists $\partial\in H^0(X,T_{X})$ such that $\sG=\sO_X\partial$.
By Proposition \ref{prop:regular_quasi_etale}, $\partial$ induces a nowhere vanishing vector field $\partial_U \in H^0(U,T_U)$ 
on some open $G$-stable neighborhood $U$
of $0 \in \mathbb{C}^{n}$. 
Since $\sG$ is weakly regular, there exists a $G$-invariant holomorphic $1$-form $\alpha_U$ on $U$ such that 
$\alpha_U(\partial_U)=1$. Let $\alpha_U^0$ be the $0$-th jet of $\alpha_U$ at $0$. Then $\alpha_U^0$ is $G$-invariant as well, and $\alpha_U^0(\partial_U)(0)=1$. In particular, we have $\alpha_U^0\neq 0$. On the other hand, $\alpha_U^0=df$ for some holomorphic function $f$ at $0$ such that $f(0)=0$. Observe that $f$ must be $G$-invariant, so that there exists a holomorphic function $t$ on some open analytic neighborhood of $x$ in $X$
such that $\partial(t)(x)\neq 0$. By a result of Zariski (see \cite[Lemma 4]{zariski}), this implies that 
$R=R_1[[t]]$, where $R$ is the formal completion of the local ring $(\sO_{X,x},\mathfrak{m}_x)$ and $R_1\subsetneq R$ is a 
$R_1$ is a Noetherian normal ring with $\dim R_1 = \dim R -1$. Moreover, the extension of $\frac{1}{\partial(t)}\partial$ to $R$ coincides with $\partial_t$. The proposition then follows from \cite[Lemma 1.3.2]{bogomolov_mcquillan01} (see Remark \ref{rem:strongly_regular_infinitesimal}).
\end{proof}

The following is an immediate consequence of Proposition \ref{prop:regular_foliation_klt_varieties}.

\begin{cor}\label{cor:regular_foliation_klt_surface}
Let $X$ be a normal surface, and let $\sG \subsetneq T_X$ be a weakly regular foliation of rank one. Suppose that $X$ has klt singularities and that $K_\sG$ is Cartier. Then $X$ is smooth.
\end{cor}

We finally describe the structure of weakly regular foliations on klt surfaces.

\begin{prop}\label{prop:regular_foliation_klt_surface}
Let $(X,x)$ be a germ of normal surface with klt singularities, and let $\sG \subset T_X$ be a foliation of rank one. Then $\sG$ is weakly regular if and only if
there exists a positive integer $m$, as well as non-negative integers $a$ and $b$ with $(a,m)=1$ and  
$(b,m)=1$ such that $(X,x) \cong (\mathbb{C}^2/G,0)$ and $\pi^{-1}\sG=\sO_{\mathbb{C}^2}\partial_{y_1}$ where $(y_1,y_2)$ are coordinates on $\mathbb{C}^2$, $G=<\zeta>$ is a cyclic group of order $m$ acting on $\mathbb{C}^2$ by $\zeta\cdot(y_1,y_2)=(\zeta^a y_1,\zeta^b y_2)$, and $\pi \colon \mathbb{C}^2 \to \mathbb{C}^2/G$ denotes the quotient map. 
\end{prop}

\begin{proof}
Suppose first that $\sG$ is weakly regular.
Recall that $X$ is $\mathbb{Q}$-factorial by \cite[Proposition 4.11]{kollar_mori}. Let $m$ be the smallest positive integer such that 
$mK_\sG$ is Cartier at $x$, and let $\pi\colon Y \to X$ be the associated local cyclic cover, which is a
quasi-\'etale cover of degree $m$ (see \cite[Definition 5.19]{kollar_mori}). By Proposition \ref{prop:regular_quasi_etale}, $\pi^{-1}\sG$ is a weakly regular foliation on $Y$ with $K_{\pi^{-1}\sG}$ Cartier. It follows from 
Corollary \ref{cor:regular_foliation_klt_surface} that $Y$ is smooth. 
The same argument used in the proof of \cite[Corollary I.2.2]{mcquillan08} then shows our claim.

Conversely, let $m$ be a positive integer, let $a$ and $b$ be non-negative integers such that $(a,m)=1$ and $(b,m)=1$, 
and let $G=<\zeta>$ be a cyclic group of order $m$ acting on $\mathbb{C}^2$ with coordinates $(y_1,y_2)$ by $\zeta\cdot(y_1,y_2)=(\zeta^a y_1,\zeta^b y_2)$. Set $X:=\mathbb{C}^2/G$, and denote by $\pi \colon \mathbb{C}^2 \to \mathbb{C}^2/G$ the natural morphism.
Let $\sL$ be the line bundle $\sO_{\mathbb{A}^2}$ equipped with the $G$-linearization given by the character 
$\zeta\mapsto\zeta^a$ of $G$. Note that there exists a reflexive rank $1$ sheaf $\sM$ on $X$ such that $\sL\cong \pi^{[*]}\sM$. 
Then $\partial_{y_1}$ (resp. $dy_1$)
is a global $G$-invariant section of $T_{\mathbb{C}^2}\otimes\sL$ (resp. $\Omega_{\mathbb{C}^2}^1\otimes\sL^*$), and thus yields 
a global section $\tau$ (resp. $\alpha$) of $T_X \boxtimes\sM$ (resp. $\Omega_X^{[1]}\boxtimes\sM^*$). Since $dy_1(\partial_{y_1})=1$, we must have $\alpha(\tau)=1$. This shows that the foliation induced by $\tau$ on $X$ is weakly regular, completing the proof of the proposition. 
\end{proof}

\section{Weakly regular foliations with algebraic leaves}\label{section:regular_foliations_algebraic_leaves}

It is well-known that an algebraically integrable regular foliation on a complex projective manifold is induced by a morphism onto a normal projective variety (see \cite[Proposition 2.5]{hwang_viehweg}). In the present section, we extend this result to weakly regular foliations with canonical singularities on mildly singular varieties.

\begin{thm}\label{thm:regular_foliation_morphism}
Let $X$ be a normal complex projective variety with $\mathbb{Q}$-factorial klt singularities, and let $\sG$ be a weakly regular algebraically integrable foliation on $X$. Suppose in addition that $\sG$ has canonical singularities. Then $\sG$ is induced by a surjective equidimensional morphism $\psi\colon X \to Y$ onto a normal projective variety $Y$. Moreover,
there exists an open subset $Y^\circ$ with complement of codimension at least two in $Y$ such that $\psi^{-1}(y)$ is irreducible for any $y\in Y^\circ$
\end{thm}

Before we give the proof of Theorem \ref{thm:regular_foliation_morphism}, we need the following auxiliary lemma.

\begin{lemma}\label{lemma:irreducible_fibers}
Let $X$ be a normal complex projective variety with klt singularities, and let $\sG$ be a weakly regular algebraically integrable foliation on $X$.
Let $\psi\colon Z \to Y$ be the family of leaves, and let $\beta\colon Z \to X$ be the natural morphism (see \ref{family_leaves}). Suppose in addition that $\sG$ has canonical singularities. Then there exists an open subset $Y^\circ$ with complement of codimension at least two in $Y$ such that $\psi^{-1}(y)$ is irreducible for any $y\in Y^\circ$. 
\end{lemma}

\begin{proof}
By \cite[Lemma 4.2]{druel15}, there exists a finite surjective morphism
$g\colon Y_1 \to Y$ with $Y_1$ normal and connected such that the following holds. If $Z_1$ denotes the normalization of $Y_1 \times_Y Z$, then the induced morphism $\psi_1\colon Z_1 \to Y_1$ has reduced fibers over codimension one points in $Y_1$. 
We obtain a commutative diagram

\begin{center}
\begin{tikzcd}[row sep=large]
Z_1  \ar[d, "{\psi_1}"']\ar[r, "f"]  & Z\ar[d, "{\psi}"]\ar[r, "{\beta}"]  & X\\
Y_1 \ar[r, "g"']& Y. &
\end{tikzcd}
\end{center}

\noindent Let $Y^\circ \subseteq Y_{\textup{reg}}$ be an open subset with complement of codimension at least two in $Y$ such that 
$Y_1^\circ:=g^{-1}(Y^\circ)$ is smooth and $R\big({\psi_1}_{|Z_1^\circ}\big)=0$, where 
$Z_1^\circ:=\psi_1^{-1}(Y_1^\circ)$ and 
$R\big({\psi_1}_{|Z_1^\circ}\big)$ denotes the ramification divisor of the restriction ${\psi_1}_{|Z_1^\circ}$ of $\psi_1$ to $Z_1^\circ$.
Set $Z^\circ:=\psi^{-1}(Y^\circ)$, and denote by
$R\big({\psi}_{|Z^\circ}\big)$ the ramification divisor of the restriction ${\psi}_{|Z^\circ}$ of $\psi$ to $Z^\circ$. By \cite[Lemma 5.4]{druel15}, the pair $\big(Z^\circ,-R(\psi_{|Z^\circ})\big)$
has canonical singularities. On the other hand, we have 
$$K_{Z_1^\circ/Y_1^\circ}=K_{Z_1^\circ/Y_1^\circ}-R\big({\psi_1}_{|Z_1^\circ}\big)\sim_\mathbb{Z} 
(f_{|Z_1^\circ})^*\big(K_{Z^\circ/Y^\circ} - R\big({\psi}_{|Z^\circ}\big)\big)$$
by Lemma \ref{lemma:pull_back_fol_and_finite_cover}.
Applying \cite[Proposition 3.16]{kollar97}, we see that
$Z_1^\circ$ has canonical singularities. 

To prove the statement, it suffices to show that $\psi^{-1}(y)$ is irreducible for any $y\in Y_1$ away from a codimension two closed subset.
We argue by contradiction and assume that there exists a prime divisor $D \subset Y_1$ such that
$\psi_1^{-1}(y)$ is reducible for a general point $y\in D$. 

Let $S \subseteq \psi_1^{-1}(D)$ be a subvariety of codimension two in $Z_1$ such that for a general point $z \in S$ there is at least two irreducible components of $\psi_1^{-1}\big(\psi_1(z)\big)$ passing through 
$z$. 
Let $z \in S$ be a general point.
Recall from \cite[Proposition 9.3]{greb_kebekus_kovacs_peternell10} that $z$
has an analytic neighborhood $U$ that is biholomorphic to an analytic neighborhood of the origin in a variety of the form 
$\mathbb{C}^{\dim X}/G$, where G is a finite subgroup of $\textup{GL}(\dim X,\mathbb{C})$
that does not contain any quasi-reflections. In particular, if $W$ denotes the inverse image of $U$ in the affine space 
$\mathbb{C}^{\dim X}$,
then the quotient map $\pi\colon W\to W/G\cong U$ is \'etale outside of the singular set. 

From Lemma \ref{lemma:regular_family_leaves} together with Proposition \ref{prop:regular_quasi_etale}, we see that
$\sG$ induces a regular foliation on $W$.
Let $F_1$ and $F_2$ be irreducible components of 
$\psi_1^{-1}\big(\psi_1(z)\big)$ passing through $z$ with $F_1\neq F_2$. 
Then $\pi^{-1}(F_1 \cap U)$ and $\pi^{-1}(F_2\cap U)$ have no common irreducible components but  
$\pi^{-1}(F_1\cap U)\cap \pi^{-1}(F_2\cap U)\neq\emptyset$. On the other hand, 
by general choice of $z$, $F_1$ and $F_2$ are not contained in the singular locus of $f^{-1}\beta^{-1}\sG$,
and hence both $\pi^{-1}(F_1\cap U)$ and $\pi^{-1}(F_2\cap U)$ are disjoint union of leaves, yielding a contradiction. This finishes the proof of the lemma.
\end{proof}

\begin{proof}[Proof of Theorem \ref{thm:regular_foliation_morphism}]
Set $r:=\textup{rank}\,\sG$. Let $\psi\colon Z \to Y$ be the family of leaves, and let $\beta\colon Z \to X$ be the natural morphism (see \ref{family_leaves}). Set also $\sG_Z:=\beta^{-1}\sG$.
To prove Theorem \ref{thm:regular_foliation_morphism}, we have to show that the $\beta$-exceptional set 
$\textup{Exc}\,\beta$ is empty. We argue by contradiction and assume that $\textup{Exc}\,\beta\neq \emptyset$.
Let $E$ be an irreducible component of $\textup{Exc}\,\beta$. Note that
$E$ has codimension one since $X$ is $\mathbb{Q}$-factorial by assumption. 
By Lemma \ref{lemma:canonical_foliation_versus_lc_pairs}, we have 
$K_{\sG_Z} \sim_\mathbb{Q}\beta^*K_\sG$. 
Applying Proposition \ref{proposition:dicritical_versus_canonical}, we see that $\psi(E)\subsetneq Y$.
On the other hand, there exists an open subset $Y^\circ$ with complement of codimension at least two in $Y$ such that $\psi^{-1}(y)$ is irreducible for any $y\in Y^\circ$ by Lemma \ref{lemma:irreducible_fibers}. 
This implies that $E=\psi^{-1}\big(\psi(E)\big)$. In particular, $E$ is invariant under $\sG_Z$.

Let $m$ be a positive integer and let $X ^\circ \subseteq X$ be a dense open subset such that 
$\sO_{X^\circ}\big(m{K_\sG}_{|X^\circ}\big)\cong \sO_{X^\circ}$. Suppose in addition that 
$\beta(E)\cap X^\circ\neq \emptyset$.
Let $f^\circ \colon X_1^\circ \to X^\circ$ be the associated cyclic cover, which is quasi-\'etale (see \cite[Definition 2.52]{kollar_mori}), and let $Z_1^\circ$ be the normalization of the product $Z^\circ \times_{X^\circ} X_1^\circ$, where 
$Z^\circ:=\beta^{-1}(X^\circ)$. Let also $\beta_1^\circ\colon Z_1^\circ \to X_1^\circ$ 
and $g^\circ\colon Z_1^\circ \to Z^\circ$ denote the natural morphisms. Recall from the proof of Lemma \ref{lemma:regular_family_leaves} 
that the foliations $\sG_{X_1^\circ}$ and $\sG_{Z_1^\circ}$ induced by $\sG$ on $X_1^\circ$ and $Z_1^\circ$ respectively are weakly regular foliations, and that 
$K_{\sG_{Z_1^\circ}} \sim_\mathbb{Z} (\beta_1^\circ)^*K_{\sG_{X_1^\circ}}$.
Let $E_1^\circ$ be a prime divisor on $Z_1^\circ$ such that $g^\circ(E_1^\circ)=E \cap Z^\circ=:E^\circ$.
Notice that $E_1^\circ$ is invariant under $\sG_{Z_1^\circ}$ since $E$ is $\sG_Z$-invariant. Moreover, $E_1^\circ$ is obviously $\beta_1^\circ$-exceptional.
Let $F_1^\circ$ denotes the normalization of $E_1^\circ$, and let $\alpha_1^\circ\colon F_1^\circ \to B_1^\circ$ be 
the Stein factorization of the map $F_1^\circ \to X_1^\circ$.
Shrinking $X^\circ$, if necessary, we may assume wihtout loss of generality that $B_1^\circ$ is smooth.
We obtain a commutative diagram 

\begin{center}
\begin{tikzcd}[row sep=small, column sep=normal]
F_1^\circ \ar[rr, "\alpha_1^\circ"]\ar[dd, "n_1^\circ"'] && B_1^\circ \ar[ddr, bend left=25, "j_1^\circ"] & \\
&&&\\
E_1^\circ \ar[r, hookrightarrow, "i_1^\circ"]\ar[dd] & Z_1^\circ \ar[rr, "\beta_1^\circ"]\ar[dd, "g^\circ"'] && X_1^\circ \ar[dd, "f^\circ"]\\
&&&\\
E^\circ \ar[r, hookrightarrow] \ar[d, hookrightarrow] & Z^\circ \ar[rr, "\beta^\circ"] \ar[d, hookrightarrow] && X^\circ \ar[d, hookrightarrow]\\
E \ar[r, hookrightarrow] & Z \ar[rr, "\beta"]\ar[dd, "\psi"'] && X\\
&&&\\
& Y. &&
\end{tikzcd}
\end{center}

\begin{claim}\label{claim:projectable}
The foliation on $F_1^\circ$ induced by $\sG_{Z_1^\circ}$ is projectable under the map $\alpha_1^\circ\colon F_1^\circ \to B_1^\circ$. 
\end{claim}

\begin{proof}[Proof of Claim \ref{claim:projectable}]
Let $\eta_{Z_1^\circ}\colon \Omega_{Z_1^\circ}^{[r]} \twoheadrightarrow \sO_{Z_1^\circ}\big(K_{\sG_{Z_1^\circ}}\big)$ the Pfaff field associated to 
$\sG_{Z_1^\circ}$. Since $E_1^\circ$ is invariant under $\sG_{Z_1^\circ}$, the composed map of sheaves
${\Omega_{Z_1^\circ}^{r}}_{|E_1^\circ} \to {\Omega_{Z_1^\circ}^{[r]}}_{|E_1^\circ} \to \sO_{Z_1^\circ}\big(K_{\sG_{Z_1^\circ}}\big)_{|E_1^\circ}$
factors through the natural morphism
${\Omega_{Z_1^\circ}^{r}}_{|E_1^\circ} \twoheadrightarrow \Omega_{E_1^\circ}^{r}$
and gives a map
$\Omega_{E_1^\circ}^{r} \to \sO_{Z_1^\circ}\big(K_{\sG_{Z_1^\circ}}\big)_{|E_1^\circ}$
(see \cite[Lemma 2.7]{fano_fols}).
By \cite[Proposition 4.5]{adk08},
the latter extends to a morphism
$$\eta_{F_1^\circ}\colon\Omega_{F_1^\circ}^{r} \to (n_1^\circ)^*\big(\sO_{Z_1^\circ}\big(K_{\sG_{Z_1^\circ}}\big)_{|E_1^\circ}\big).$$
By construction, $\eta_{F_1^\circ}$ is the Pfaff field associated to the foliation induced by $\sG_{Z_1^\circ}$ on $F_1^\circ$ on a dense open set.

Notice that 
$$(n_1^\circ)^*\big(\sO_{Z_1^\circ}\big(K_{\sG_{Z_1^\circ}}\big)_{|E_1^\circ}\big)
\cong (\alpha_1^\circ)^* \big((j_1^\circ)^*\sO_{X_1^\circ}\big(K_{\sG_{X_1^\circ}}\big)\big)$$
since $\sO_{Z_1^\circ}\big(K_{\sG_{Z_1^\circ}}\big) \cong (\beta_1^\circ)^*\sO_{X_1^\circ}\big(K_{\sG_{X_1^\circ}}\big)$ par construction.
Therefore, there exists a morphism of sheaves
$$\eta_{B_1^\circ}\colon \Omega_{B_1^\circ}^r \to (j_1^\circ)^*\sO_{X_1^\circ}\big(K_{\sG_{X_1^\circ}}\big)$$
whose pull-back under $\alpha_1^\circ$ is the composition

\begin{center}
\begin{tikzcd}
(\alpha_1^\circ)^*\Omega_{B_1^\circ}^r \ar[rr, "d\alpha_1^\circ"] && 
\Omega_{F_1^\circ}^{r} \ar[r, "\eta_{F_1^\circ}"] & (n_1^\circ)^*\big(\sO_{Z_1^\circ}\big(K_{\sG_{Z_1^\circ}}\big)_{|E_1^\circ}\big)
\cong (\alpha_1^\circ)^* \big((j_1^\circ)^*\sO_{X_1^\circ}\big(K_{\sG_{X_1^\circ}}\big)\big).
\end{tikzcd}
\end{center}

\medskip

Now, to show that the foliation on $F_1^\circ$ induced by $\sG_{Z_1^\circ}$ is projectable under the map $\alpha_1^\circ\colon F_1^\circ \to B_1^\circ$, it suffices to prove that $\eta_{B_1^\circ}$ is non-zero. 

\medskip

Let $\eta_{X_1^\circ}\colon \Omega_{X_1^\circ}^{[r]} \twoheadrightarrow \sO_{X_1^\circ}\big(K_{\sG_{X_1^\circ}}\big)$ the Pfaff field associated to $\sG_{X_1^\circ}$. Let $E_2^\circ \subseteq E_1^\circ$ denotes a smooth dense open set contained
in the smooth locus of $Z_1^\circ$, and let $F_2^\circ \subseteq F_1^\circ$ be its inverse image in $F_1^\circ$.
Notice that $E_2^\circ \cong F_2^\circ$. Set $i_2^\circ:= {i_1^\circ}_{|E_2^\circ} \colon E_2^\circ \into Z_1^\circ$ and 
$\alpha_2^\circ:= {\alpha_1^\circ}_{|F_2^\circ} \colon F_2^\circ \to B_1^\circ$.
By \cite[Proposition 6.1]{kebekus_pull_back} (see also subsection \ref{subsection:pull-back_morphims}),
the composition

\begin{center}
\begin{tikzcd}
\big((\beta_1^\circ)^*\Omega_{X_1^\circ}^{[r]}\big)_{|E_2^\circ} \ar[rr, "{d_{\textup{refl}}\beta_1^\circ}_{|E_2^\circ}"] && {\Omega_{Z_1^\circ}^{[r]}}_{|E_2^\circ} \cong {\Omega_{Z_1^\circ}^{r}}_{|E_2^\circ}
\ar[rr, twoheadrightarrow, "d i_2^\circ"] && \Omega_{E_2^\circ}^{r}
\end{tikzcd}
\end{center}

\noindent agrees with 

\begin{center}
\begin{tikzcd}
\big((\beta_1^\circ)^*\Omega_{X_1^\circ}^{[r]}\big)_{|E_2^\circ} \cong \big((\alpha_1^\circ)^*\big((j_1^\circ)^*\Omega_{X_1^\circ}^{[r]}\big)\big)_{|F_2^\circ} \ar[rrr, "\big((\alpha_1^\circ)^* d_{\textup{refl}}j_1^\circ\big)_{|F_2^\circ}"] &&& \big((\alpha_1^\circ)^*\Omega_{B_1^\circ}^{r}\big)_{|F_2^\circ}
\ar[rr, "d\alpha_2^\circ"] && \Omega_{F_2^\circ}^{r}.
\end{tikzcd}
\end{center}
On the other hand, recall from the proof of Lemma \ref{lemma:regular_bir_crepant_map}, that there is a commutative
diagram

\begin{center}
\begin{tikzcd}[row sep=large]
(\beta_1^\circ)^*\Omega_{X_1^\circ}^{[r]} \ar[r, "{d_{\textup{refl}}\beta_1^\circ}"]\ar[d, twoheadrightarrow, "{(\beta_1^\circ)^*\eta_{X_1^\circ}}"'] & \Omega_{Z_1^\circ}^{[r]}\ar[d, twoheadrightarrow, "{\eta_{Z_1^\circ}}"] \\
(\beta_1^\circ)^*\sO_{X_1^\circ}\big(K_{\sG_{X_1^\circ}}\big) \ar[r, "{\sim}"] & 
\sO_{Z_1^\circ}\big(K_{\sG_{Z_1^\circ}}\big).
\end{tikzcd}
\end{center}
One then readily checks that the diagram 

\begin{center}
\begin{tikzcd}[row sep=large]
(j_1^\circ)^*\Omega_{X_1^\circ}^{[r]} \ar[rr, "d_{\textup{refl}}j_1^\circ"] \ar[d, twoheadrightarrow, "(j_1^\circ)^*\eta_{X_1^\circ}"'] && \Omega_{B_1^\circ}^{r} \ar[dll, bend left=20, "\eta_{B_1^\circ}"] \\
(j_1^\circ)^*\sO_{X_1^\circ}\big(K_{\sG_{X_1^\circ}}\big) &&
\end{tikzcd}
\end{center}
is commutative as well.
This immediately implies that $\eta_{B_1^\circ}$ is non-zero, completing the proof of the claim. 
\end{proof}

Since the foliation induced by $\sG_{Z_1^\circ}$ on $F_1^\circ$ is projectable under $\alpha_1^\circ$,
there are infinitely many leaves of $\sG_{Z_1^\circ}$ contained in $E_1^\circ$ that 
maps to the same leaf of the foliation induced by $\eta_{B_1^\circ}$ on $B_1^\circ$ since $\dim E_1^\circ > \dim B_1^\circ$. 
On the other hand, recall that $\psi^{-1}(y)$ is irreducible for any point $y\in \psi(E) \cap Y^\circ$. Hence, there exists a positive integer $m$ such that the cycle theoretic fiber $\psi^{[-1]}(y)$ is $m[\psi^{-1}(y)]$ for a general point $y$ in 
$\psi(E)$. It follows that the restriction of map $Y \to \textup{Chow}(X)$ (see \ref{family_leaves}) to 
$\psi(E)$ has positive dimensional fibers, yielding a contradiction.
This finishes the proof of the theorem.
\end{proof}

\begin{rem}
In the setup of Theorem \ref{thm:regular_foliation_morphism}, let $\psi\colon Z \to Y$ be the family of leaves, and let
$\beta\colon Z \to X$ be the natural morphism (see \ref{family_leaves}). If $X$ is only assumed to have klt singularities, then the same argument used in the proof of the theorem shows that $\beta$ is a small birational map.
\end{rem}

\begin{cor}
Let $X$ be a normal complex projective surface with klt singularities, and let $\sG$ be an algebraically integrable foliation by curves on $X$ with canonical singularities. If $\sG$ is weakly regular, then it is induced by a surjective equidimensional morphism $\psi\colon X \to Y$ onto a smooth complete curve $Y$.
\end{cor}

\begin{proof}
This is an immediate consequence of Theorem \ref{thm:regular_foliation_morphism} since $X$ is automatically $\mathbb{Q}$-factorial by \cite[Proposition 4.11]{kollar_mori}. 
\end{proof}

The following is an easy consequence of Theorem \ref{thm:regular_foliation_morphism} above together with Lemma \ref{lemma:regular_versus_canonical}.

\begin{cor}\label{cor:regular_foliation_morphism}
Let $X$ be a normal complex projective variety with $\mathbb{Q}$-factorial klt singularities, and let $\sG$ be a weakly regular algebraically integrable foliation on $X$. Suppose in addition that $K_\sG$ is Cartier. Then $\sG$ is induced by a surjective equidimensional morphism $\psi\colon X \to Y$ onto a normal projective variety $Y$. Moreover, there exists an open subset $Y^\circ$ with complement of codimension at least two in $Y$ such that $\psi^{-1}(y)$ is irreducible for any $y\in Y^\circ$
\end{cor}

\section{Quasi-\'etale trivializable reflexive sheaves}\label{section:holonomy}

In this section we provide another technical tool for the proof of the main results.

\begin{prop}\label{prop:finite_versus_etale}
Let $X$ be a normal complex projective variety, and let $\sG$ be a coherent reflexive sheaf of rank $r$. Suppose that there exists a finite cover $f \colon Y \to X$ such that $f^{[*]}\sG\cong \sO_{Y}^{\oplus r}$. Then, there exists a quasi-\'etale cover $g \colon Z \to X$ such that $g^{[*]}\sG\cong \sO_{Z}^{\oplus r}$. 
\end{prop}

\subsection{The holonomy group of a stable reflexive sheaf} We briefly recall the definition of algebraic holonomy groups following Balaji and Koll\'ar (\cite{balaji_kollar}).

Let $X$ be a normal complex projective variety, and let $\sG$ be a coherent reflexive sheaf on $X$. 
Suppose that $\sG$ is stable with respect to an ample Cartier divisor $H$ and that $\mu_H(\sG)=0$.
For a sufficiently large positive integer $m$,
let $C \subset X$ be a general complete intersection curve of elements in $|mH|$, and let $x \in C$.
By the restriction theorem of Mehta and Ramanathan,
the locally free sheaf $\sG_{|C}$ is stable with $\deg\sG_{|C}=0$, and hence
it corresponds to a unique unitary representation $\rho\colon\pi_1(C,x)\to\mathbb{U}(\sG_x)$ 
by a result of Narasimhan and Seshadri (\cite{narasimhan_seshadri65}). 

\begin{defn}
The holonomy group $\Hol_x(\sG)$
of $\sG$ is the Zariski closure of $\rho\big(\pi_1(C,x)\big)$ in $\textup{GL}(\sG_x)$.
\end{defn}

\begin{rem}
The holonomy group $\Hol_x(\sG)$ does not depend on $C \ni x$ provided that $m$ is large enough (see \cite{balaji_kollar}).
\end{rem}

\subsection{Strong stability} The following notion is used in the formulation of Lemmas \ref{lemma:holonomy_group_versus_strong_stability} and
\ref{lemma:stability_versus_strong_stability}.

\begin{defn}
Let $X$ be a normal projective variety, let $H$ be a nef and big Cartier divisor on $X$, and let $\sG$ be a coherent reflexive sheaf. We say that $\sG$ is \textit{strongly stable with respect to $H$} if, for any 
normal projective variety $Y$ and any
generically finite surjective morphism $f \colon Y \to X$, the reflexive pull-back sheaf $f^{[*]}\sG$ is $f^{*}H$-stable.
\end{defn}

\begin{lemma}[{\cite[Lemma 6.3]{druel_bbcd2}}]\label{lemma:holonomy_group_versus_strong_stability}
Let $X$ be a normal complex projective variety, let $x\in X$ be a general point, and let $\sG$ be a coherent reflexive sheaf.
Suppose that $\sG$ is stable with respect to an ample divisor $H$ and that $\mu_H(\sG)=0$.
Suppose furthermore that the holonomy group $\Hol_x(\sG)$ is connected. 
Then $\sG$ is strongly stable with respect to $H$.
\end{lemma}

The following observation is needed for the main result of this section.

\begin{lemma}\label{lemma:stability_versus_strong_stability}
Let $X$ be a normal complex projective variety, and let $\sG$ be a coherent reflexive sheaf of $\sO_X$-modules.
Suppose that $\sG$ is polystable with respect to an ample divisor $H$ and that $\mu_H(\sG)=0$.
Then there exists a quasi-\'etale cover $g\colon Z \to X$ as well as coherent reflexive sheaves $(\sG_i)_{i\in I}$ on $Z$
such that the following holds.
\begin{enumerate}
\item There is a decomposition $g^{[*]}\sG\cong \oplus_{i\in I}\sG_i$.
\item The sheaves $\sG_i$ are strongly stable with respect to $g^*H$ with $\mu_{g^*H}(\sG_i)=0$.
\end{enumerate}
\end{lemma}

\begin{proof}
Suppose that there exists a quasi-\'etale cover $g_1\colon Z_1 \to X$ such that 
the reflexive pull-back $g_1^{[*]}\sG$ is not stable with respect to $g_1^*H$.
Applying \cite[Lemma 3.2.3]{HuyLehn}, we see that the $g_1^{[*]}\sG$ is polystable, and hence, there exist non-zero reflexive sheaves $(\sG_i)_{i \in I}$, $g_1^*H$-stable with slopes $\mu_{g_1^*H}(\sG_i)=\mu_{g_1^*H}\big(g_1^{[*]}\sG\big)=0$ such that 
$$g_1^{[*]}\sG\cong\bigoplus_{i \in I}\sG_i.$$

Suppose in addition that the number of direct summands is maximal.
Then, for any quasi-\'etale cover $g_2\colon Z_2 \to Z_1$, the reflexive pull-back $g_2^{[*]}\sG_i$ is obviously stable with respect to $(g_1\circ g_2)^*H$.

By \cite[Lemma 40]{balaji_kollar} (see also \cite[Lemma 6.19]{bobo}), there exists a quasi-\'etale cover
$g_i\colon Z_i \to Z_1$ such that $\Hol_{z_i}\big((g_1\circ g_i)^{[*]}\sG_i\big)$ is connected, where $z_i$ is a general point on $Z_i$. Let 
$Z$ be the normalization of $Z_1$ in compositum of the function fields $\mathbb{C}(Z_i)$. Observe that 
the natural morphism $g \colon Z \to Z_1$ is a quasi-\'etale cover, and factors through each $Z_i \to Z_1$. 
The proposition then follows from Lemma \ref{lemma:holonomy_group_versus_strong_stability} above. 
\end{proof}

\begin{rem}\label{remark:one_versus_all_polarizations}
In the setup of Lemma \ref{lemma:stability_versus_strong_stability}, suppose furthermore that $\sG$ is polystable with respect to any polarization $H$ on $X$ and that $\mu_H(\sG) = 0$. Let $H_1$ be any ample divisor on $X$.
Then the sheaves $\sG_i$ are strongly stable with respect to $g^*H_1$ and $\mu_{g^*H_1}(\sG_i)=0$.
\end{rem}

\begin{proof}[Proof of Proposition \ref{prop:finite_versus_etale}]
Let $H$ be an ample Cartier divisor on $X$. Applying \cite[Lemma 3.2.3]{HuyLehn}, we see that $\sG$ is polystable with respect to $H$ with $\mu_H(\sG)=0$. By Lemma \ref{lemma:stability_versus_strong_stability}, there exists a quasi-\'etale cover 
$g_1\colon Z_1 \to X$ as well as coherent reflexive sheaves $(\sG_i)_{i\in I}$ on $Z_1$
such that the following holds.
\begin{enumerate}
\item There is a decomposition $g_1^{[*]}\sG\cong \oplus_{i\in I}\sG_i$.
\item The sheaves $\sG_i$ are strongly stable with respect to $g_1^*H$ with $\mu_{g_1^*H}(\sG_i)=0$.
\end{enumerate}
Let $Y_1$ be the normalization of the product $Y \times_X Z_1$, with natural morphism $f_1 \colon Y_1 \to Z_1$. By construction, we have $(g_1\circ f_1)^{[*]}\sG\cong \sO_{Y_1}^{\oplus r}$. 
It follows that $(g_1\circ f_1)^{[*]}\sG_i\cong \sO_{Y_1}$ for all $i \in I$, and hence
$\sG_i^{[N]} \cong \sO_{Z_1}$, where $N:=\deg\,g_1 \circ f_1$. Replacing $Z_1$ by a further quasi-\'etale cover, we may assume that 
$\sG_i \cong \sO_{Z_1}$, proving the proposition.
\end{proof}

\begin{rem}
In the setup of Proposition \ref{prop:finite_versus_etale}, suppose that $X$ is a (smooth complete) curve. Then there is an alternative argument. Indeed, by \cite[Lemma 3.2.3]{HuyLehn}, $\sG$ is a polystable vector bundle. We may assume without loss of generality that $\sG$ is stable. By \cite{narasimhan_seshadri65}, $\sG$ corresponds to a unique unitary representation $$\rho\colon\pi_1(X,x)\to\mathbb{U}(\sG_x).$$ It follows that $f^*\sG$ corresponds to the induced representation
$$\rho\circ\pi_1(f)\colon\pi_1(Y,y)\to\mathbb{U}(\sG_x)$$
where $y$ is a point on $Y$ such that $f(y)=x$.
Applying \cite[Proposition 4.2]{narasimhan_seshadri64} to $\rho\circ\pi_1(f)$ and the trivial representation of $\pi_1(Y,y)$ in $\mathbb{U}(\sG_x)$, we see that $\rho\circ\pi_1(f)$ must be the trivial representation. The statement then follows since 
the image of $\pi_1(f)$ has finite index in $\pi_1(X,x)$.
\end{rem}

\section{A global Reeb stability theorem}\label{section:reeb}

\subsection{Global Reeb stability theorem}
Let $X$ be a complex manifold, and let $\sG$ be a regular foliation on $X$. Let $L$ be a compact leaf with finite holonomy group $G$, and let $x \in L$. The holomorphic version of the local Reeb stability theorem (see \cite[Theorem 2.4]{hwang_viehweg}) asserts that there exist an invariant open analytic neighborhood $U$ of $L$, a (local) transversal section $S$ at $x$ with a $G$-action, an unramified Galois cover $U_1 \to U$ with group $G$, a smooth proper $G$-equivariant morphism $U_1 \to S$, and a commutative diagram

\begin{center}
\begin{tikzcd}[row sep=large, column sep=huge]
U_1 \ar[r, "{\text{unramified}}"]\ar[d, "{\text{proper submersion}}"'] & U \ar[d] \\
 S \ar[r] & S/G 
\end{tikzcd}
\end{center}
such that the pull-back of $\sG_{|U}$ to $U_1$ is induced by the map $U_1 \to S$.
In this section, we prove a global version of Reeb stability theorem for weakly regular algebraically integrable foliations with trivial canonical class on mildly singular spaces. The following is the main result of this section (see \cite[Proposition 4.2]{druel_guenancia} for a somewhat related result).

\begin{thm}\label{thm:global_reeb_stability}
Let $X$ be a normal complex projective variety with klt singularities, and let $\sG$ be a weakly regular algebraically integrable foliation on $X$. Suppose that $K_\sG\sim_{\mathbb{Q}}0$.
Then there exist complex projective varieties $Y$ and $Z$ with klt singularities and a quasi-\'etale cover 
$f \colon Y \times Z \to X$ such that $f^{-1}\sG$ is induced by the projection $Y \times Z \to Y$.
\end{thm}

\medskip

The proof of Theorem \ref{thm:global_reeb_stability} makes use of the following result, which might be of independent interest.

\begin{prop}\label{proposition:trivial_tangent_bundle}
Let $X$ be a normal complex projective variety, and let $\sG$ be an algebraically integrable foliation on $X$. Suppose that $\sG$ is canonical and that $\sG\cong\sO_X^{\oplus \textup{rank} \sG}$. Then there exists an abelian variety $A$ as well as a normal projective variety $X_1$, and an \'etale cover $f\colon X_1 \to X$ such that $f^{-1}\sG=T_{A \times X_1/X_1}$. 
\end{prop}

\begin{proof}
Let $\psi\colon Z \to Y$ be the family of leaves, and let $\beta\colon Z \to X$ be the natural morphism (see \ref{family_leaves}). Let also $F$ be a general fiber of $\psi$.
Then $K_{\beta^{-1}\sG}\sim_\mathbb{Z}0$ by Lemma \ref{lemma:canonical_foliation_versus_lc_pairs} and Remark \ref{rem:family_leaves_cartier}. Moreover, $F$ has canonical singularities. By Example \ref{example:canonical_class_foliation} and the adjunction formula, we have $K_F\sim_{\mathbb{Z}}{K_{\beta^{-1}\sG}}_{|F}\sim_\mathbb{Z}0$.

The same argument used in the proof of \cite[Lemma 3.2]{fano_fols} (see also \cite[Remark 3.8]{fano_fols}) shows that the dual
map $\Omega_X^{[1]} \to \sG^*$ gives a generically surjective morphism $\Omega_{Z/Y}^{[1]} \to \beta^*\sG^*$ and 
a commutative diagram

\begin{center}
\begin{tikzcd}[row sep=large]
\beta^*\Omega_X^{1} \ar[d]\ar[rr, "{d\beta}"] & & \Omega_Z^{1}\ar[r, twoheadrightarrow] & \Omega_{Z/Y}^{1}\ar[d]\\
\beta^*\sG^*  &  & & \Omega_{Z/Y}^{[1]}\ar[lll]
\end{tikzcd}
\end{center}

\noindent 
This immediately implies that $T_F\cong \sO_F^{\oplus \dim F}$ since $K_F \sim_\mathbb{Z}0$.
By Lemma \ref{lemma:automorphism_group_zero_can_class}, we see that
the neutral component $\textup{Aut}^\circ(F)$ of the automorphism group $\textup{Aut}(F)$ of $F$ is an abelian variety.

Let $\textup{Aut}^\circ(X)$ denote the neutral component of $\textup{Aut}(X)$, and let 
$H \subseteq \textup{Aut}^\circ(X)$ be the connected complex Lie subgroup
with Lie algebra $H^0(X,\sG)\subseteq H^0(X,T_X)$.
Note that $\beta(F)$ is invariant under the action of $H$, and thus $H$ also acts on the normalization $F$ of $\beta(F)$.
The induced morphism of complex Lie groups $H \subseteq \textup{Aut}^\circ(F)$ is then an isomorphism
since the tangent map
$$\textup{Lie}\,H=H^0(X,\sG) \to H^0(F,T_F)=\textup{Lie}\,\textup{Aut}^\circ(F)$$
is surjective.  In particular, $H$ is a closed (projective) algebraic subgroup of $\textup{Aut}^\circ(X)$.

By \cite[Proof of Theorem 1.2, page 10]{brion_action}, there exist 
a normal projective variety $X_1$, and a $H$-equivariant finite \'etale cover 
$f\colon H \times X_1 \to X$, where $H$ acts 
trivially on $X_1$ and
diagonally on $H \times X_1$. 
In particular, we have $f^{-1}\sG=T_{H\times X_1/X_1}$, completing the proof of the proposition.
\end{proof}

Next, we consider the special case where the foliation is induced by a morphism equipped with a flat
connection.

\begin{lemma}\label{lemma:alg_int_zer_can_class}
Let $X$ be a normal complex quasi-projective variety, and let $\phi\colon X \to Y$ be a projective equidimensional morphism 
with connected fibers onto a smooth quasi-projective variety. 
Suppose that $X$ has klt singularities over the generic point of $Y$, and that $\phi$ has reduced fibers over codimension one points in $Y$.
Suppose furthermore that $K_{X/Y}$ is relatively numerically trivial and that there exists a foliation $\sE$ on $X$ such that $T_X = T_{X/Y}\oplus\sE$.
Then there exist complex varieties $B$ and $F$, as well as a 
quasi-\'etale cover $f\colon B \times F \to X$ such that 
$T_{B\times F/B}=f^{-1}T_{X/Y}$. Moreover, $B$ is smooth and quasi-projective, $F$ is projective with canonical singularities, and $K_F\sim_\mathbb{Z}0$.
\end{lemma}

\begin{proof}
Let $X^\circ \subseteq X$ be the open set where 
${\phi}_{|X^\circ}$ is smooth. Notice that $X^\circ$ has complement of codimension at least two since $\phi$ 
is equidimensional with reduced fibers over codimension one points in $Y$ by assumption. 
The restriction of the tangent map 
$$T{\phi}_{|X^\circ}\colon T_{X^\circ}\to ({\phi}_{|X^\circ})^*T_{Y}$$
to $\sE_{|X^\circ} \subseteq T_{X^\circ}$ 
then induces an isomorphism $\sE_{|X^\circ}\cong ({\phi}_{|X^\circ})^*T_{Y}$. Since 
$\sE_{|X^\circ}$ and $({\phi}_{|X^\circ})^*T_{Y}$
are both reflexive sheaves, we must have
$\sE\cong \phi^*T_{Y}.$
Thus $\sE$ yields a flat connection on $\phi$. 
A classical result of complex analysis then says that $\phi$ is a locally trivial analytic fibration for the analytic topology
(see \cite{kaup}).

Let $F$ be any fiber of $\phi$. Note that $F$ has klt singularities by \cite[Corollary 4.9]{kollar97}.
By the adjunction formula, $K_F\equiv 0$. 
Let $y\in Y$, and denote by $X_y\cong F$ the fiber $\phi^{-1}(y)$.
Let $\textup{Aut}^\circ(X_y)$ denotes the neutral component of 
the automorphism group $\textup{Aut}(X_y)$ of $X_y$. By Lemma \ref{lemma:automorphism_group_zero_can_class}, 
$\textup{Aut}^\circ(X_y)$ is an abelian variety. 
Recall that $\dim \textup{Aut}^\circ(X_y) = h^0(F,T_F)$.
By
\cite[Expos\'e VI$_{\textup{B}}$, Proposition 1.6]{sga3}, the group scheme $\textup{Aut}(X/Y)$ is
smooth over $Y$. 
Note that the existence of $\textup{Aut}(X/Y)$ is guarantee by \cite{fga221}.
Now, recall from \cite[Expos\'e VI$_{\textup{B}}$, Th\'eor\`eme 3.10]{sga3} 
that the algebraic groups
$\textup{Aut}^0(X_y)$ fit together to form a group scheme $\sA$
over $Y$. Note that $\sA$ is quasi-projective over $Y$ (see \cite{fga221}). Applying \cite[Corollaire 15.7.11]{ega28}, we see that $\sA$ is proper, and hence projective over $Y$. In particular, $\sA$ is an abelian scheme over $Y$.
Since $\sA$ is locally trivial, there exist 
an abelian variety $A$, and a finite \'etale
cover $Y_1 \to Y$ such that $\sA \times_Y Y_1 \cong A \times Y_1$
as group schemes over $Y_1$. 
This follows from the fact that there is a fine moduli scheme for polarized abelian varieties of dimension $g$, with level $N$ structure and polarization of degree $d$ provided that $N$ is large enough.
In particular, $A$ acts faithfully on $X_1:=X\times_Y Y_1$. 
By \cite[Proof of Theorem 1.2, page 10]{brion_action}, there exist a finite \'etale cover 
$X_2$
of $X_1$ 
equipped with a faithful action of $A$, and an $A$-isomorphism
$X_2 \cong A \times Z_2$ for some quasi-projective variety $Z_2$, where $A$ acts 
trivially on $Z_2$ and
diagonally on $A \times Z_2$. 
Using the rigidity lemma, we see that there exist a projective morphism with connected fibers $\phi_2\colon Z_2 \to Y_2$
onto a normal variety $Y_2$, and a commutative diagram

\begin{center}
\begin{tikzcd}[column sep=large]
X_2 \cong A \times Z_2 \ar[r, "f_2"]\ar[d] & X_1 \ar[dd, "\phi_1"]\ar[r, "f_1"] & X \ar[dd, "\phi"]\\
Z_2\ar[d, "\phi_2"'] & &\\
Y_2 \ar[r] & Y_1 \ar[r] & Y.
\end{tikzcd}
\end{center}

\noindent Since $\phi_1$ has reduced fibers over codimension one points in $Y_1$, the finite morphism $Y_2 \to Y_1$ is \'etale in codimension one, hence \'etale by the Nagata-Zariski purity theorem.
One also readily checks that $\phi_2$ has reduced fibers over codimension one points in $Y_2$, and that $Z_2$ has klt singularities over the generic point of $Y_2$. Moreover, we have $K_{A \times Z_2/Y_2}\equiv_{Y_2} 0$, and thus we also have 
$K_{Z_2/Y_2}\equiv_{Y_2} 0$. 
Finally, $\big((f_1\circ f_2)^{-1}\sE\big)_{|\{0_A \times Z_2\}}$ induces a flat connection on $\phi_2$.

Repeating the process finitely many times, if necessary, we may therefore assume without loss of generality that 
any fiber $F_2$ of $\phi_2$ satisfies $h^0(F_2,T_{F_2})=0$.
Let $H_2$ be an ample divisor on $Z_2$. Since $\phi_2$ admits a flat connection, it is given by 
a representation
$$\pi_1(Y_2,y) \to \textup{Aut}(F_2,{H_2}_{|F_2})$$
where $y:=\phi_2 (F_2)$, and where $\textup{Aut}(F_2,{H_2}_{|F_2})$ denotes the group
$\{u \in \textup{Aut}(F_2)\,|\,u^*{H_2}_{|F_2}\equiv {H_2}_{|F_2}\}$.
Now, recall that $\textup{Aut}(F_2,{H_2}_{|F_2})$ is an open and closed algebraic subgroup of finite type
of $\textup{Aut}(F_2)$, and hence finite since $\dim \textup{Aut}(F_2) = h^0(F_2,T_{F_2})=0$.
Therefore, replacing $Y_2$ with a finite \'etale cover, if necessary, we may assume that 
$Z_2 \cong Y_2 \times F_2$ as varieties over $Y_2$. Note that $K_{F_2} \equiv 0$.
By \cite[Corollary V 4.9]{nakayama04}, $K_{F_2}$ is torsion. Replacing $F_2$ with its canonical index one cover (see \cite[Definition 2.52]{kollar_mori}) we may therefore assume that $K_{F_2} \sim_{\mathbb{Z}} 0$.
This finishes the proof of the lemma.
\end{proof}

The same argument used in the proof of Lemma \ref{lemma:alg_int_zer_can_class} shows that the following holds.

\begin{lemma}\label{lemma:alg_int_zer_can_class2}
Let $X$ be a normal complex quasi-projective variety, and let $\phi\colon X \to Y$ be a projective equidimensional morphism with connected fibers onto a smooth quasi-projective variety. 
Suppose that $X$ has klt singularities over the generic point of $Y$, and that $\phi$ has reduced fibers over codimension one points in $Y$.
Suppose furthermore that $K_{X/Y}$ is relatively numerically trivial and that there exists a foliation $\sE$ on $X$ such that $T_X = T_{X/Y}\oplus\sE$.
Let $F$ be any fiber of $\phi$. Suppose finally that $h^0(F,T_F)=0$.
Then there exists a smooth quasi-projective variety $Y_1$ as well as a finite \'etale cover $Y_1 \to Y$ such that $Y_1 \times_Y X \cong Y_1 \times F$ as varieties over $Y_1$.
\end{lemma}

\begin{rem}
In the setup of Lemma \ref{lemma:alg_int_zer_can_class} or Lemma \ref{lemma:alg_int_zer_can_class2} above, suppose in addition that $X$ is projective with klt singularities, and that $K_\sG$ is Cartier. Then the existence of $\sE$ follows from the assumption $K_{X/Y}\equiv 0$ by Corollary \ref{corollary:compact_leaf_holomorphic_form} together with 
Proposition \ref{prop:generic_smoothness}.
\end{rem}

The proof of Theorem \ref{thm:global_reeb_stability} relies in part on the following descent result for foliations.

\begin{lemma}\label{lemma:infinitesimal_unicity_BB}
Let $X$ be a normal complex variety, and let $\sG$ be a foliation on $X$. 
Let also $Y$ be a normal variety, let $A$ be an abelian variety and let $B$ be a projective variety $B$ with canonical singularities, $K_B\sim_\mathbb{Z}0$, and $\wt q(B)=0$. Suppose that there is a finite cover 
$f \colon Y\times A \times B \to X$
such that $f^{-1}\sG = T_{Y\times A \times B/Y}$.
Suppose in addition that any codimension one irreducible component of the branch locus of $f$ is $\sG$-invariant.
Then there exist foliations $\sG_1 \subseteq \sG$ and $\sG_2 \subseteq \sG$ such that 
$f^{-1}\sG_1=T_{Y\times A \times B/Y\times B}$ and $f^{-1}\sG_2=T_{Y\times A \times B/Y\times A}$.
\end{lemma}

\begin{proof}
Set $Z:= Y\times A \times B$. Let $g \colon Z_1 \to Z$ be a finite cover such that the induced cover $f_1 \colon Z_1 \to X$ is Galois. We may assume without loss of generality that $g$ is quasi-\'etale away from the branch locus of $f$, so that any codimension one irreducible component of the branch locus of $f_1$ is $\sG$-invariant as well.
By Lemma \ref{lemma:pull_back_fol_and_finite_cover}, we must 
have $f_1^{-1}\sG \cong f_1^{[*]}\sG$. 

Set $\sE_1=g^{-1}T_{Y\times A \times B/Y\times B}$ and $\sE_2=g^{-1}T_{Y\times A \times B/Y\times A}$, and
let $\gamma$ be any automorphism of the covering $f_1 \colon Z_1 \to X$. In order to prove the statement, it suffices to show
that $\sE_i \subseteq T_{Z_1}$ is $\gamma$-invariant (see \cite[Lemme 2.13]{claudon_bourbaki}). 

Notice that $\sE_1\oplus\sE_2 = f_1^{-1}\sG \cong f_1^{[*]}\sG$ is $\gamma$-invariant by construction, and that $\sE_1\cong g^{[*]}T_{Y\times A \times B/Y\times B}\cong \sO_{Z_1}^{\oplus\dim A}$ and $\sE_2\cong g^{[*]}T_{Y\times A \times B/Y\times A}$ by Lemma \ref{lemma:pull_back_fol_and_finite_cover} again.
Let $F\cong A\times B$ be a general fiber of the projection $Y\times A \times B \to Y$. By general choice of $F$, any irreducible component $G$ of $g^{-1}(F)$ is normal (see \cite[Th\'eor\`eme 12.2.4]{ega28}) with
$T_G\cong {\sE_1}_{|G}\oplus{\sE_2}_{|G}$. In particular, we must have $K_G\sim_\mathbb{Z}0$.
Applying Theorem \ref{thm:kawamata_abelian_factor}, one readily checks that there exist an abelian variety $A_1$, a projective variety $B_1$ with canonical singularities, $K_{B_1}\sim_\mathbb{Z}0$, and $\wt q(B_1)=0$, and quasi-\'etale covers  
$A_1 \to A$ and $B_1 \to B$ such that $G \cong A_1\times B_1$ and such that 
the restriction of $g$ to $G$ identifies with $G \cong A_1\times B_1 \to A\times B\cong F$. Moreover,
${\sE_1}_{|G}\cong T_{A_1\times B_1/B_1}$ and ${\sE_2}_{|G}\cong T_{A_1\times B_1/A_1}$.
By Remark \ref{remark:augmented_irregularity_vector_fields}, we have $h^0(B_1,T_{B_1})=0$ and 
$h^1\big(B_1,\Omega_{B_1}^{[1]}\big)=0$. This immediately implies that any map 
$\gamma^*\sE_1 \to \sE_2$ or $\gamma^*\sE_2 \to \sE_1$ vanishes identically, completing the proof of the lemma.
\end{proof}

Before we give the proof of Theorem \ref{thm:global_reeb_stability}, we need the following auxiliary statements.

\begin{lemma}\label{lemma:produit}
Let $X$ be a normal complex projective variety, and let $\psi\colon X \to Y$ be a surjective equidimensional morphism with connected fibers onto a normal projective variety. Suppose that $X$ is $\mathbb{Q}$-factorial.
Let $F$ denotes a general fiber of $\psi$, and assume that $q(F)=0$. If 
$\psi^{-1}(Y^\circ) \cong Y^\circ \times F$ for some open set $Y^\circ \subseteq Y_{\textup{reg}}$ with complement of codimension at least two, then $X \cong Y \times F$.
\end{lemma}

\begin{proof}
Set $X^\circ:=\psi^{-1}(Y^\circ)$.
Let $H$ be an ample divisor on $X$. We have $\Pic\big(Y^\circ \times F\big)\cong \Pic\big(Y^\circ\big) \times \Pic(F)$ since $q(F)=0$ by assumption. Thus there exist Cartier divisors $H_{Y^\circ}$ and $H_F$ on $Y^\circ$ and $F$ respectively
such that $\sO_{X^\circ}\big(H_{|X^\circ}\big)\cong \sO_{Y^\circ}\big(H_{Y^\circ}\big) \boxtimes \sO_F\big(H_F\big)$. 
Let $H_Y$ be the Weil divisor on $Y$ that restricts to $H_{Y^\circ}$ on $Y^\circ$.
By Lemma \ref{lemma:pull-back_Q_Cartier}, $H_Y$ is $\mathbb{Q}$-Cartier. Notice that $m_0\big(H-\psi^*H_Y\big)$ is relatively ample over $Y$ for any positive integer $m_0$ such that $m_0\big(H-\psi^*H_Y\big)$ is Cartier.
It follows that 
$$X \cong \Proj_Y \bigoplus_{m \ge 0}\psi_*\sO_X\big(m m_0\big(H-\psi^*H_Y\big)\big).$$
On the other hand, by \cite[Corollary 1.7]{hartshorne80}, the coherent sheaves 
$\psi_*\sO_X\big(m m_0\big(H-\psi^*H_Y\big)\big)$ are reflexive, and restrict to 
$H^0\big(F,\sO_F(mm_0H_F)\big)\otimes\sO_{Y^\circ}$ by assumption. Therefore, there is an isomorphism of sheaves of graded $\sO_Y$-algebras
$$\bigoplus_{m \ge 0}\psi_*\sO_X\big(m m_0\big(H-\psi^*H_Y\big)\big) \cong 
\Big(\bigoplus_{m \ge 0}H^0\big(F,\sO_F(mm_0H_F)\big) \Big)\otimes\sO_Y.$$
This implies that  
$X \cong Y \times F$, finishing the proof of the lemma.
\end{proof}

\begin{lemma}[{\cite[Lemma 4.6]{bobo}}]\label{lemma:product_versus_contraction}
Let $X_1$, $X_2$ and $Y$ be complex normal projective varieties. Suppose that there exists a surjective morphism with connected fibers $\beta\colon X_1 \times X_2 \to Y$. Suppose furthermore that 
$q(X_1)=0$. Then $Y$ decomposes as a product 
$Y\cong Y_1\times Y_2$ of normal projective varieties,
and there exist surjective morphisms with connected fibers $\beta_1\colon X_1\to Y_1$ and 
$\beta_2\colon X_2\to Y_2$
such that $\beta = \beta_1 \times \beta_2$.
\end{lemma}

We are now in position to prove Theorem \ref{thm:global_reeb_stability}.

\begin{proof}[Proof of Theorem \ref{thm:global_reeb_stability}] We maintain notation and assumptions of Theorem \ref{thm:global_reeb_stability}. For the reader's convenience, the proof is subdivided into a number of steps.

\medskip

\noindent\textit{Step 1. Reduction to the case where $K_\sG\sim_\mathbb{Z}0$.} By \cite[Lemma 2.53]{kollar_mori} and Fact \ref{fact:quasi_etale_cover_and_singularities}, there exists a quasi-\'etale cover $f \colon X_1 \to X$ 
with $X_1$ klt 
such that $f^*K_{\sG}\sim_\mathbb{Z} 0$. 
Moreover, $K_{f^{-1}\sG}\sim_\mathbb{Z}f^*K_\sG$ and $f^{-1}\sG$ is weakly regular by Proposition \ref{prop:regular_quasi_etale}.
To prove Theorem \ref{thm:global_reeb_stability}, we can therefore assume without loss of generality that the following holds.

\begin{assumption}
The canonical class $K_\sG$ is trivial, $K_\sG\sim_\mathbb{Z}0$.
\end{assumption}

This implies that $\sG$ is canonical by Lemma \ref{lemma:regular_versus_canonical}.

\medskip

\noindent\textit{Step 2. Reduction to the case where $\sG$ is given by a morphism.} 
Suppose that the conclusion of Theorem \ref{thm:global_reeb_stability} holds under the additional assumption that the 
foliation is given by an equidimensional morphism. Then we show that it holds in general.

\medskip

Let
$\beta\colon X_1\to X$ be a $\mathbb{Q}$-factorialization (see \ref{say:q_factorialization}), and set $\sG_1:=\beta^{-1}\sG$.
By Lemma \ref{lemma:properties:regular} and Remark \ref{rem:bir_small_cartier},
$\sG_1$ is a weakly regular foliation with $K_{\sG_1}\sim_\mathbb{Z} 0$. 
By Lemma \ref{lemma:regular_versus_canonical} again, we see that $\sG_1$ is canonical.
It then follows from Theorem \ref{thm:regular_foliation_morphism} that $\sG_1$ is induced by a surjective equidimensional morphism onto a normal projective variety.
By assumption, there exist complex projective varieties $Y_1$ and $Z_1$ with klt singularities and a quasi-\'etale cover 
$f_1 \colon Y_1 \times Z_1 \to X_1$ such that $f_1^{-1}\sG_1$ is induced by the projection $Y_1 \times Z_1 \to Y_1$. 
In particular, we have $K_{Z_1} \sim_\mathbb{Z}0$. By Theorem \ref{thm:kawamata_abelian_factor}, replacing $Z_1$ by a further quasi-\'etale cover, if necessary, we may assume without loss of generality that $Z_1\cong A_1\times B_1$, where $A_1$ is an abelian variety and $B_1$ is normal projective variety with $\wt q(B_1)=0$.
Let $X_2$ be the normalization of $X$ in the function field of $Y_1\times Z_1$. 
We have a commutative diagram

\begin{center}
\begin{tikzcd}[row sep=large, column sep=huge]
Y_1 \times Z_1  \ar[d, "{f_1,\textup{ quasi-\'etale}}"']\ar[rr, "{\beta_2,\textup{ small and birational}}"] & & X_2 \ar[d, "{f,\textup{ quasi-\'etale}}"] \\
X_1 \ar[rr, "{\beta,\textup{ small and birational}}"']&& X.
\end{tikzcd}
\end{center}

\noindent Note that $T_{Y_1 \times A_1 \times B_1}\cong T_{Y_1}\boxplus T_{A_1} \boxplus T_{B_1}$. The direct summand $T_{A_1}$ of $T_{Y_1 \times A_1 \times B_1}$ induces an algebraically integrable foliation 
$\sE_2 \subseteq f^{-1}\sG $ with $\sE_2 \cong \sO_{X_2}^{\oplus \dim A_1}$. By Lemma \ref{lemma:direct_summand_regular}, $\sE_2$ is a weakly regular foliation, and it has canonical singularities by 
Lemma \ref{lemma:regular_versus_canonical}. Applying Proposition \ref{proposition:trivial_tangent_bundle} to $\sE_2$, we see that there exists an abelian variety $A_3$ as well as a normal projective variety $Z_3$, and a finite \'etale cover
$f_2 \colon X_3:=A_3 \times Z_3 \to X_2$ such that the foliation $\sE_3:=f_2^{-1}\sE_2$ is induced by the projection $A_3 \times Z_3 \to Z_3$. Observe that the direct summand $T_{B_1}$ of $T_{Y_1 \times A_1 \times B_1}$ induces an algebraically integrable foliation $\sE_3^\perp$ on $X_3$ such that $\sE_3\oplus \sE_3^\perp = f_2^{-1}f^{-1}\sG$.
One then readily checks that there exists an algebraically integrable foliation $\sG_3$ on $Z_3$ such that $\sE_3^\perp$ is the pull-back of $\sG_3$ under the projection $A_3 \times Z_3 \to Z_3$. Note that $\sG_3$ is weakly regular with $K_{\sG_3}\sim_\mathbb{Z} 0$
by Lemma \ref{lemma:properties:regular}.
By replacing $X$ by $Z_3$ and $\sG$ by $\sG_3$, and repeating the process finitely many times,
we may therefore assume that $\dim A_1 =0$. But then the conclusion of Theorem \ref{thm:global_reeb_stability} follows from Lemma \ref{lemma:product_versus_contraction} since 
$q(B_1)=0$ by construction.
To prove Theorem \ref{thm:global_reeb_stability}, we can therefore assume without loss of generality that the following holds.

\begin{assumption}
The foliation $\sG$ is induced by a surjective equidimensional morphism with connected fibers $\psi\colon X \to Y$ onto a normal projective variety $Y$.
\end{assumption}

\noindent\textit{Step 3.} Let $F$ be a general fiber of $\psi$. Note that $F$ has klt singularities by \cite[Corollary 4.9]{kollar97}. Applying Corollary \ref{corollary:compact_leaf_holomorphic_form}, we see that there is a decomposition $T_X\cong\sG\oplus\sE$ of $T_X$ into involutive subsheaves.

\begin{claim}\label{claim:base_change_isotriviality}
There exist an open subset $Y^\circ \subseteq Y_{\textup{reg}}$ with complement of codimension at least two in $Y$, and a finite galois cover $g^\circ\colon Y_1^\circ \to Y^\circ$ such that the following holds. Let $X_1^\circ$ be the normalization of $Y_1^\circ \times_{Y^\circ} X$, and denote by $\psi_1^\circ\colon X_1^\circ \to Y_1^\circ$ the natural morphism.
Then $\psi_1^\circ$ is a locally trivial analytic fibration for the analytic topology. In particular, $\psi_1^\circ$ has integral fibers. Moreover, there exists a decomposition $T_{X_1^\circ}=T_{X_1^\circ/Y_1^\circ}\oplus \sE_1^\circ$
of $T_{X_1^\circ}$ into involutive subsheaves.
\end{claim}

\begin{proof}
By \cite[Lemma 4.2]{druel15}, there exists a finite surjective morphism
$g\colon Y_1 \to Y$ with $Y_1$ normal and connected such that the following holds. If $X_1$ denotes the normalization of $Y_1 \times_Y X$, then the induced morphism $\psi_1\colon X_1 \to Y_1$ has reduced fibers over codimension one points in $Y_1$. 
By replacing $Y_1$ with a further finite cover, if necessary, we may assume without loss of generality that
the finite cover $Y_1 \to Y$ is Galois. We have a commutative diagram

\begin{center}
\begin{tikzcd}[row sep=large]
X_1  \ar[d, "\psi_1"']\ar[rr, "{f,\textup{ finite}}"] & & X\ar[d, "\psi"] \\
Y_1 \ar[rr, "{g,\textup{ finite}}"]&& Y.
\end{tikzcd}
\end{center}

Next, we show that the tangent sheaf $T_{X_1}$ decomposes as a direct sum 
$T_{X_1}=f^{-1}\sG\oplus f^{-1}\sE$ and that $K_{f^{-1}\sG}\sim_\mathbb{Z} 0$.
Set $q:=\textup{rank}\,\sG$, and let $\omega \in H^0\big(X,\Omega_X^{[q]}\big)$ a reflexive $q$-form defining $\sE$. 
The reflexive pull-back
$\omega_1\in H^0\big(X_1,\Omega_{X_1}^{[q]}\big)$ of $\omega$
then defines $f^{-1}\sE$ on a dense open set. In particular, $\omega_1$ induces an $\sO_{X_1}$-linear map 
$\big(\wedge^{q}T_{X_1}\big)^{**} \to \sO_{X_1}$
such that the composed morphism of reflexive sheaves of rank one
$$\tau\colon \det f^{-1}\sG \to  (\wedge^{q}T_{X_1})^{**} \to \sO_{X_1}$$
is generically non-zero. 
On the other hand, by Lemma \ref{lemma:pull_back_fol_and_finite_cover}, we have
$\det f^{-1}\sG\cong f^*\sO_X(-K_\sG)\cong \sO_{X_1}$
and hence $\tau$ must be an isomorphism. This easily implies that 
$T_{X_1}=f^{-1}\sG\oplus f^{-1}\sE$.

Let $Y_1^\circ$ be the smooth locus of $Y_1$, and set $X_1^\circ:=\psi_1^{-1}(Y_1^\circ)$.
Let $Z_1^\circ \subseteq X_1^\circ$ be the open set where 
${\psi_1}_{|X_1^\circ}$ is smooth. Notice that $Z_1^\circ$ has complement of codimension at least two 
in $X_1^\circ$ since $\psi_1$ has reduced fibers over codimension one points in $Y_1^\circ$. 
The restriction of the tangent map 
$$T{\psi_1}_{|X_1^\circ}\colon T_{X_1^\circ}\to \big({\psi_1}_{|X_1^\circ}\big)^*T_{Y_1^\circ}$$
to $f^{-1}\sE_{|Z_1^\circ} \subseteq T_{Z_1^\circ}$ 
then induces an isomorphism $f^{-1}\sE_{|Z_1^\circ}\cong \big({\psi_1}_{|Z_1^\circ}\big)^*T_{Y_1^\circ}$. Since 
$f^{-1}\sE_{|X_1^\circ}$ and $\big({\psi_1}_{|X_1^\circ}\big)^*T_{Y_1^\circ}$
are both reflexive sheaves, we finally obtain an isomorphism
$$f^{-1}\sE_{|X_1^\circ}\cong \big({\psi_1}_{|X_1^\circ}\big)^*T_{Y_1^\circ}.$$
A classical result of complex analysis then says that ${\psi_1}_{|X_1^\circ}$ is a locally trivial analytic fibration for the analytic topology. 
Set 
$Y^\circ:={Y}_{\textup{reg}}\setminus g\big(Y_1 \setminus {Y_1^\circ}\big)$. 
Then $g_{|g^{-1}(Y^\circ)}\colon g^{-1}(Y^\circ) \to Y^\circ$  
satisfies the conclusions of Claim \ref{claim:base_change_isotriviality}.
\end{proof}

\medskip

\noindent\textit{Step 4. Reduction to the case where $\wt q(F)=0$.} We maintain notation of Claim \ref{claim:base_change_isotriviality}. Set $f^\circ:=f_{|X_1^\circ}\colon X_1^\circ \to X^\circ:=\psi^{-1}(Y^\circ)$.

By Lemma \ref{lemma:alg_int_zer_can_class} and Theorem \ref{thm:kawamata_abelian_factor},  
there exist a quasi-projective variety $Y_2^\circ$, 
an abelian variety $A$ and a projective variety $B$ with canonical singularities, $K_B\sim_\mathbb{Z} 0$ and $\wt q(B)=0$, and 
a quasi-\'etale cover $f_1^\circ\colon X_2^\circ := Y_2^\circ \times A \times B \to X_1^\circ$ such that 
$(f_1^\circ)^{-1}\big(\sG_{|X_1^\circ}\big)$ is the foliation given by the projection $Y_2^\circ \times A \times B \to Y_2^\circ$. Notice that any codimension one irreducible component of the branch locus of $f^\circ \circ f_1^\circ$ is invariant under $\sG$ so that Lemma \ref{lemma:infinitesimal_unicity_BB} applies.
There exist algebraically integrable foliations $\sG_1 \subseteq \sG$ and $\sG_2 \subseteq \sG$ with torsion canonical class such that 
$(f^\circ \circ f_1^\circ)^{-1}\big({\sG_1}_{|X^\circ}\big)=T_{Y_2^\circ\times A \times B/Y_2^\circ \times B}$ and $(f^\circ \circ f_1^\circ)^{-1}\big({\sG_2}_{|X^\circ}\big)=T_{Y_2^\circ\times A \times B/Y_2^\circ \times A}$.
Note that $\sG_1$ and $\sG_2$ are weakly regular foliations by Lemma \ref{lemma:direct_summand_regular} since $\sG_1\oplus\sG_2\oplus\sE=T_X$ by construction.

Let $X_2$ be the normalization of $X$ in the function field of $X_2^\circ$. Applying Proposition \ref{prop:finite_versus_etale} to the cover $X_2\to X$ and to $\sG_1$, we see that, replacing $X$ by a further quasi-\'etale cover, if necessary, we may assume without loss of generality that $\sG_1 \cong \sO_{X}^{\dim A}$.
By Proposition \ref{proposition:trivial_tangent_bundle} together with Lemma \ref{lemma:regular_versus_canonical}, 
there exists an abelian variety $A_1$ as well as a normal projective variety $X_1$, and an \'etale cover 
$f\colon A_1\times X_1 \to X$ such that $f^{-1}\sG_1=T_{A_1 \times X_1/X_1}$. 
Then $f^{-1}\sG_2$ is the pull-back of a weakly regular algebraically integrable foliation $\sH_1$ on $X_1$ with $K_{\sH_1}\sim_\mathbb{Z} 0$ (see 
Lemma \ref{lemma:properties:regular}).
Moreover, by the rigidity lemma, $\sH_1$ is given by a surjective equidimensional morphism $\psi_1\colon X_1 \to Y_1$ onto a normal variety $Y_1$ whose genral fiber $F_1$ satisfies $\wt q(F_1)=0$ since irregularity increases in covers by \cite[Lemma 4.1.14]{lazarsfeld1}. To prove Theorem \ref{thm:global_reeb_stability}, we can therefore assume without loss of generality that the following holds.

\begin{assumption}\label{irregularity}
The augmented irregularity of a general fiber $F$ of $\psi$ is zero, $\wt q(F)=0$.
\end{assumption}

\noindent\textit{Step 5.} Next, we show the following.
\begin{claim}\label{claim:base_change_isotriviality2}
There exist a finite cover $g\colon Y_1 \to Y$ and an open subset $Y^\circ \subset Y_{\textup{reg}}$ with complement of codimension at least two in $Y$ such that the following holds. Let $X_1$ be the normalization of $Y_1 \times_{Y} X$, and denote by $f \colon X_1 \to X$ and $\psi_1\colon X_1 \to Y_1$ the natural morphisms. Then $f_1$ is a quasi-\'etale cover, and 
$\psi_1^{-1}(Y_1^\circ)\cong Y_1^\circ \times F$ as varieties over $Y_1^\circ$, where 
$Y_1^\circ:=g^{-1}(Y^\circ)$.
\end{claim}

\begin{proof}
We maintain notation of Claim \ref{claim:base_change_isotriviality}. Notice $h^0(F,T_{F})=0$ since $\wt q(F)=0$ (see Remark \ref{remark:augmented_irregularity_vector_fields}), so that Lemma \ref{lemma:alg_int_zer_can_class2} applies. Replacing $Y_1$ by a quasi-\'etale cover, if necessary, we may assume that $X_1^\circ \cong Y_1^\circ \times F$ as varieties over $Y_1^\circ$. 
We obtain a commutative diagram

\begin{center}
\begin{tikzcd}[row sep=large]
X_1^\circ\cong Y_1^\circ\times F \ar[r, hookrightarrow]\ar[d] & X_1  \ar[d, "{\psi_1}"']\ar[rr, "{f,\textup{ finite}}"] & & X\ar[d, "{\psi}"]\\
Y_1^\circ \ar[r, hookrightarrow] & Y_1 \ar[rr, "{g,\textup{ finite}}"']&& Y.
\end{tikzcd}
\end{center}

By replacing $Y_1$ with a further finite cover, if necessary, we may assume without loss of generality that
the finite cover $Y_1 \to Y$ is Galois. In particular,
there is a finite group $G$ acting on $Y_1$ with quotient $Y$. 
The group $G$ also acts on $X_1$ since $X_1$ identifies with the normalization of
$Y_1\times_{Y} X$. Moreover, $X_1^\circ\cong Y_1^\circ \times F$ is $G$-invariant.
Since $h^0(F,T_{F})=0$, $G$ acts on $F$ and its action on $Y_1^\circ\times F$ is the diagonal action. 
Let $G_1$ denote the kernel of the induced morphism of groups $G \to \textup{Aut}(F)$.
Note that $X_1^\circ/G_1 \cong (Y_1^\circ/G_1) \times F$.
By replacing $Y_1$ by $Y_1/G_1$, $X_1$ by $X_1/G_1$, and $G$ by $G/G_1$, if necessary, we may assume 
that $G \subseteq \textup{Aut}(F)$. Set $X^\circ:=\psi^{-1}(Y^\circ)$. 
We may obviously assume that $\dim Y \ge 1$ and $\dim F \ge 1$.
Then
the quotient map $Y_1^\circ \times F\to (Y_1^\circ \times F)/G\cong X_1^\circ$ is automatically
\'etale in codimension one, and hence quasi-\'etale by the Nagata-Zariski purity theorem.
This immediately implies that $f\colon X_1 \to X$ is a quasi-\'etale cover as well.
\end{proof}

\noindent\textit{Step 6. End of proof.} We maintain notation of Claim \ref{claim:base_change_isotriviality2}.
Let $\beta_1\colon X_2\to X_1$ be a $\mathbb{Q}$-factorialization (see \ref{say:q_factorialization}), and set $\sG_2:=\beta_1^{-1}f^{-1}\sG$. By Lemma \ref{lemma:properties:regular} and Remark \ref{rem:bir_small_cartier},
$\sG_2$ is a weakly regular foliation and $K_{\sG_2}\sim_\mathbb{Z} 0$. It follows that $\sG_2$ is induced by a surjective equidimensional morphism $\psi_2\colon X_2 \to Y_2$ onto a normal projective variety by Theorem \ref{thm:regular_foliation_morphism} again. 

Since $\beta_1$ is a small birational morphism, we may assume without loss of generality that $\psi_1\circ\beta_2$ has integral fibers over $Y_1^\circ$. In particular, the restriction of $\psi_1\circ\beta_2$ to 
$X_2^\circ:=(\psi_1\circ\beta_2)^{-1}(Y_1^\circ)$ is equidimensional. It follows that $X_2^\circ \to Y_1^\circ$ is the family of leaves of ${\sG_2}_{|X_2^\circ}$, and hence there exists an open set $Y_2^\circ \subseteq Y_2$ and an isomorphism
$Y_2^\circ\cong Y_1^\circ$ such that $X_2^\circ = \psi_2^{-1}(Y_2^\circ)$.
Notice that we have a decomposition $T_{X_2}\cong\sG_2\oplus \beta_1^{-1}f^{-1}\sE$ of $T_{X_2}$ into involutive subsheaves. Set $F_2:=\beta_1^{-1}(F)$, and observe that 
$h^0(F_2,T_{F_2})=h^0(F,T_{F})=0$ since ${\beta_1}_{|F_2}$ is a small birational morphism. By Lemma \ref{lemma:alg_int_zer_can_class2}, replacing $Y_1$ and $Y_2$ by further quasi-\'etale covers, if necessary, we may assume that $X_2^\circ \cong Y_2^\circ \times F_2$ as varieties over $Y_2^\circ$.
By Remark \ref{remark:augmented_irregularity_vector_fields}, we also have $q(F_2)=0$. On the other hand, $Y_2^\circ$ has complement of codimension at least $2$ in $Y_2$ since $X_2^\circ$ has complement of codimension at least $2$ in $X_2$ and 
$\psi_2$ is equidimensional. So Lemma \ref{lemma:produit} applies, and show that $X_2 \cong F_2 \times Y_2$.
From Lemma \ref{lemma:product_versus_contraction} again, we conclude that $X_1 \cong F_1 \times Y_1$.
This finishes the proof of Theorem \ref{thm:global_reeb_stability}.
\end{proof}

\begin{proof}[Proof of Theorem \ref{thm_intro:global_reeb_stability}] We maintain notation and assumptions of Theorem \ref{thm_intro:global_reeb_stability}. By Proposition \ref{prop:abundance_alg_int}, $K_\sG$ is torsion.
By \cite[Lemma 2.53]{kollar_mori} and Fact \ref{fact:quasi_etale_cover_and_singularities}, there exists a quasi-\'etale cover $f \colon X_1 \to X$ with $X_1$ klt such that $f^*K_{\sG}\sim_\mathbb{Z} 0$. Moreover, $f^{-1}\sG$ is canonical by Lemma \ref{lemma:canonical_quasi_etale_cover} and $K_{f^{-1}\sG}\sim_\mathbb{Z} 0$. 
Applying Corollary \ref{cor:canonical_versus_regular}, we see that $f^{-1}\sG$ is weakly regular.
Theorem \ref{thm_intro:global_reeb_stability} now follows from Theorem \ref{thm:global_reeb_stability}.
\end{proof}

\subsection{Application} The purpose of this subsection is to prove the following decomposition result (see \cite[Proposition 6.1]{druel15} for a somewhat related result).

\begin{prop}\label{prop:splitting_algebraic_transcendental}
Let $X$ be a normal complex projective variety with klt singularities, and let $\sG$ be a foliation on $X$ with canonical singularities and $K_\sG \equiv 0$. There exist normal projective varieties $Y$ and $Z$ with klt singularities, 
as well as a quasi-\'etale cover $f \colon Y \times Z \to X$ such that the following holds.
The foliation $f^{-1}\sG$ is the pull-back via the projection $Y \times Z \to Y$
of
a foliation $\sH$ on 
$Y$ with no positive-dimensional algebraic subvariety tangent to $\sH$ passing through a general point of
$Y$. In addition, $K_Z\sim_\mathbb{Z}0$, $\sH$ has canonical singularities, and $K_\sH\equiv 0$.
\end{prop}

We will need the following minor generalization of \cite[Corollary 3.9]{lpt}.

\begin{lemma}\label{lemma:stability_versus_uniruled}
Let $X$ be a normal complex projective variety, and let 
$\sG$ be a foliation on $X$ with canonical singularities and $K_\sG \equiv 0$. Then 
$\sG$ is semistable with respect to any polarization $H$ on $X$.
\end{lemma}

\begin{proof}
Let $\beta\colon Z \to X$ be a resolution of singularities.
Suppose that $\sG$ is not semistable with respect to $H$. Then $\beta^{-1}\sG$ is not semistable with respect to $\beta^*H$.
Applying \cite[Theorem 4.7]{campana_paun15} to the maximally destabilizing subsheaf of $\beta^{-1}\sG$ with respect to 
$\beta^*(H^{\dim X -1})$, we see that $\sG$ is uniruled. But this contradicts Proposition \ref{proposition:canonical_versus_uniruled}, proving the lemma.
\end{proof}

\begin{proof}[Proof of Proposition \ref{prop:splitting_algebraic_transcendental}]
Recall from \cite[Paragraph 2.3]{lpt}
that
there exist a normal projective variety $Y$, a dominant rational map $\phi\colon X \dashrightarrow Y$, and a foliation $\sH$ on $Y$ such that the following holds. 
\begin{enumerate}
\item[(a)] There is no positive-dimensional algebraic subvariety tangent to $\sH$
passing through a general point of $Y$; and
\item[(b)] $\sG$ is the pull-back of $\sH$ via $\phi$. 
\end{enumerate}
Let $\sE$ be the foliation on $X$ induced by $\phi$. We may assume without loss of generality that $Y$ is the space of leaves of $\sE$. Let $\psi\colon Z \to Y$ be the family of leaves, and let $\beta\colon Z \to X$ be the natural morphism (see \ref{family_leaves}). 

\begin{claim}
The canonical class $K_\sE$ of $\sE$ is torsion. In particular, $K_\sE$ is $\mathbb{Q}$-Cartier.
\end{claim}

\begin{proof}Replacing $X$ by a $\mathbb{Q}$-factorialization (see Paragraph \ref{say:q_factorialization}), we may assume in addition that $X$ is $\mathbb{Q}$-factorial by Lemma \ref{lemma:singularities_birational_morphism}. 

Since $\beta^{-1}\sG=\psi^{-1}\sH$, there is an exact sequence
$$0 \to \beta^{-1}\sE = T_{Z/Y} \to \beta^{-1}\sG \to \psi^{[*]}\sH.$$
Therefore, there exists an effective Weil divisor $B$ on $Z$
$$ K_{\beta^{-1}\sG}\sim_\mathbb{Z} K_{\beta^{-1}}\sE + \psi^* K_{\sH} + B.$$
It follows that
$$K_{\sG}-K_{\sE} \sim_\mathbb{Z} \phi^*K_\sH+C,$$
where $C:=\beta_* B$ is effective.
Notice that the pull-back $\phi^*K_\sH$ is well-defined (see Paragraph \ref{definition:pull-back}).
Applying \cite[Corollary 4.8]{campana_paun15} to the foliation induced by $\sH$ on a resolution of $Y$, we see that $K_{\sH}$ is pseudo-effective (see Remark \ref{rem:pseff_versus_numerical dimension} for this notion).
This immediately implies that $\phi^*K_\sH+C$ is pseudo-effective as well.
Let $H$ be an ample Cartier divisor on $X$. Then we have $\mu_H(K_\sG) \ge \mu_H(K_\sE)$.
On the other hand, 
by Lemma \ref{lemma:stability_versus_uniruled}, we also have $\mu_H(K_\sG) \le \mu_H(K_\sE)$, and hence $\mu_H(K_\sE) = \mu_H(K_\sG)=0$. This implies that $K_\sE \equiv 0$.
By Proposition \ref{proposition:canonical_versus_uniruled}, $\sE$ has canonical singularities since $\sG$ does.
Proposition \ref{prop:abundance_alg_int} then implies that $K_\sE$ is torsion, proving the claim. 
\end{proof}

By Theorem \ref{thm_intro:global_reeb_stability} applied to $\sE$ and \cite[Lemma 6.7]{fano_fols}, there exist normal projective varieties $Y_1$ and $Z_1$, as well as a foliation $\sH_1$ on 
$Y_1$ and a quasi-\'etale cover $f \colon Y_1 \times Z_1 \to X$ such that $f^{-1}\sE=T_{Y_1 \times Z_1/Y_1}$ and such that
$f^{-1}\sG$ is the pull-back of $\sH_1$ via the projection $Y_1 \times Z_1 \to Y_1$. Moreover, there is no positive-dimensional algebraic subvariety tangent to $\sH_1$
passing through a general point of $Y_1$.  
From Lemma \ref{lemma:pull-back_Q_Cartier}, we see that $K_{\sH_1}$ is $\mathbb{Q}$-Cartier. Since 
$K_{f^{-1}\sG} \equiv 0$ and $K_{f^{-1}\sE}\equiv 0$, we must have $K_{\sH_1} = K_{f^{-1}\sG} - K_{f^{-1}\sE}\equiv 0$.
By Lemma \ref{lemma:pull_back_foliation_singularities}, $\sH_1$ is canonical.
This finishes the proof of Proposition
\ref{prop:splitting_algebraic_transcendental}.
\end{proof}

\section{Algebraic integrability, I}\label{section:algebraic_integrability_1}

In this section we prove algebraicity criteria for leaves of algebraic foliations on uniruled varieties (see Proposition \ref{prop:grothendieck_katz_projective_connection} and Theorem \ref{thm:grothendieck_katz}).

\begin{exmp}\label{example:suspension} Let $n \ge 2$ be an integer.
Let $A=\mathbb{C}^{n-1}/\Lambda$ be a complex abelian variety, and let 
$\rho\colon \pi_1(A) \to \textup{PGL}(2,\mathbb{C})$ be a representation of the fundamental group $\pi_1(A)\cong\Lambda$ of $A$.
Then the group $\pi_1(A)$ acts diagonally on $\mathbb{C}^{n-1}\times \mathbb{P}^1$ by 
$\gamma\cdot(z,p)=\big(\gamma(z),\rho(\gamma)(p)\big)$. Set $X:=(\mathbb{C}^{n-1} \times \mathbb{P}^1)/\pi_1(A)$, 
and denote by $\psi\colon X \to A \cong \mathbb{C}^{n-1}/\pi_1(A)$ the projection morphism, which is $\mathbb{P}^1$-bundle.
The foliation on $\mathbb{C}^{n-1} \times \mathbb{P}^1$ induced by the projection $\mathbb{C}^{n-1} \times \mathbb{P}^1 \to \mathbb{P}^1$ is invariant under the action of $\pi_1(A)$ and yields a flat Ehresmann connection $\sG$ on $\psi$.
Note that $\rho$ is the representation induced by $\sG$. In particular, $\sG$-invariant sections of $\psi$ correspond to 
$\pi_1(A)$-invariant points on $\mathbb{P}^1$. 

Let $H \subseteq \textup{PGL}(2,\mathbb{C})$ be the Zariski closure of the image of $\rho$.
Observe that $H$ is abelian, and that
$\sG$ is algebraically integrable if and only if $H$ is finite.

Suppose that $\dim H >0$. Let $p\colon \textup{SL}(2,\mathbb{C}) \to \textup{PSL}(2,\mathbb{C})\cong \textup{PGL}(2,\mathbb{C})$ denotes the projection morphism, and set 
$$T:=\left\{
\begin{pmatrix}
a & 0 \\
0 & a^{-1}
\end{pmatrix},\, a\in\mathbb{C}^*
\right\}\subset \textup{SL}(2,\mathbb{C})
\quad \textup{and} \quad
U:=\left\{
\begin{pmatrix}
1 & b \\
0 & 1
\end{pmatrix},\, b\in\mathbb{C}
\right\}\subset \textup{SL}(2,\mathbb{C}).$$
Using \cite[Theorem 4.12]{kaplansky} (or \cite[Lemma 6.2]{corlette_simpson}), one readily checks that $H$ is conjugate to $p(T)$ or $p(U)$. In either case, there is a representation 
$$\wt\rho\colon \pi_1(A) \to \textup{SL}(2,\mathbb{C}) \quad \textup{such that} \quad \rho=\wt\rho \circ p.$$ 
Thus, there is a flat rank two vector bundle $(\sE,\nabla)$ on $A$ such that $X \cong \mathbb{P}_A(\sE)$ over $A$. Moreover, if 
$\sG_\nabla$ denotes the linear Ehresmann connection on the total space $E$ of $\sE^*$ induced by $\nabla$, then 
$\sG$ is the projection of ${\sG_\nabla}_{|E^\times}$ under the natural map $\pi\colon E^\times \to X$, where 
$E^\times:=E\setminus 0_X$.

\medskip

Suppose that $H=p(T)$. Then we may assume that $\sE=\sM\oplus \sM^{-1}$ as flat vector bundles for some $\sM\in \textup{Pic}^0(A)$. Since $\dim H >0$ by assumption, $\sM$ is not torsion. The $\sG$-invariant sections $D_1$ and $D_2$ of $\psi$ are the two sections given by the quotients $\sM\oplus \sM^{-1} \twoheadrightarrow \sM$ and $\sE=\sM\oplus \sM^{-1} \twoheadrightarrow \sM^{-1}$. 
The structure of algebraic group on $\mathbb{C}^n\times \mathbb{G}_m$ induces a structure of commutative algebraic group $G$ on $X^\circ:=X \setminus (D_1 \cup D_2)$ fitting into an exact sequence
$$1 \to \mathbb{G}_m \to G \to A \to 1.$$
Moreover, $X^\circ$ is a principal $\mathbb{G}_m$-bundle over $A$, and the action of $G$ on $X^\circ$ extends to an action on $G$. Finally, $\sG_{|X^\circ}$ yields a flat $\mathbb{G}_m$-equivariant connection on $X^\circ \to A$.

\medskip

Suppose that $H=p(U)$. Then there is a non-trivial exact sequence $0 \to \sO_A \to \sE \to \sO_A \to 0$. The 
only $\sG$-invariant section $D$ of $\psi$ corresponds to the quotient $\sE \twoheadrightarrow \sO_A$. 
In this case, there is a structure of commutative algebraic group $G$ on $X^\circ:=X \setminus D$ fitting into an exact sequence
$$1 \to \mathbb{G}_a \to G \to A \to 1,$$
and $X$ is an equivariant compactification of $G$. Moreover, $X^\circ$ is principal $\mathbb{G}_a$-bundle over $A$, and
$\sG_{|X^\circ}$ yields a flat $\mathbb{G}_a$-equivariant connection on $X^\circ \to A$.
\end{exmp}

The proof of Proposition \ref{prop:grothendieck_katz_projective_connection} below relies on an algebraicity criterion for leaves of algebraic foliations proved in \cite[Theorem 2.9]{bost}, 
which we recall now.

\medskip

Let $X$ be an algebraic variety over some field $k$ of positive 
characteristic $p$, and let $\sG \subseteq T_X$ be a subsheaf. We will denote by $\Frobabs \colon X \to X$ the absolute Frobenius morphism of $X$.

The sheaf of derivations $\textup{Der}_{k}(\sO_X)\cong T_X$ is endowed with the $p$-th power operation, which maps any local $k$-derivation $D$ of $\sO_X$ to its $p$-th iterate $D^{[p]}$. 
When $\sG$ is involutive, the map $\Frobabs^*\sG \to T_X/\sG$ which sends 
$D$ to the class in $T_X/\sG$ of $D^{[p]}$ is $\sO_X$-linear. The sheaf $\sG$ is said to be 
\emph{closed under $p$-th powers} if the map $\Frobabs^*\sG \to T_X/\sG$ vanishes.

Let $R \subset \mathbb{C}$ be a finitely generated $\mathbb{Z}$-algebra, with field of fractions $K$, and set $S := \textup{Spec}\, R$. For any closed point $s\in S$ with maximal ideal $\mathfrak{m}$, we let $k(s)$ be the finite field $R/\mathfrak{m}$.
We denote by $\bar{s}\colon \textup{Spec} \, k(\bar{s}) \to S$ a geometric point of $S$ lying over $s$ with
$k(\bar{s})$ an algebraic closure of $k(s)$.
Given a scheme $X$ over $S$, we let 
$X_\mathbb{C}:=X\otimes \mathbb{C}$,
and $X_{\bar{s}}:= X \otimes k(\bar{s})$.
Given a sheaf $\sG$ on $X$, we let $\sG_\mathbb{C}:=\sG\otimes \mathbb{C}$,
and $\sG_{\bar{s}}:= \sG \otimes k(\bar{s})$.

Let $X$ be a normal complex projective variety, and let $\sG$ be a foliation on 
$X$. Let $R \subset \mathbb{C}$ be a finitely generated $\mathbb{Z}$-algebra, with field of fractions $K$, and let $\bX$ be a projective model of $X$ over $\bS :=\textup{Spec} \,R$. Let also $\sbfG$
be a saturated subsheaf of the relative tangent sheaf $T_{\bX/\bS}$ such that $\sbfG_\mathbb{C}$ coincides
with $\sG$. We say that $\sG$ \textit{is closed under $p$-th powers for almost all primes $p$} if
there exists an open dense subset $\bU$ of $\bS$ such that for any closed point $s$ in $\bU$, 
the subsheaf $\sbfG_{\bar{s}}$ of $T_{\bX_{\bar{s}}}$
is closed under $p$-th powers, where $p$ denotes the characteristic of $k(\bar{s})$. 
This condition is independent of the choices of $\bS$ and $\bX$.

\begin{thm}[{\cite[Theorem 2.9]{bost}}]\label{thm:bost}
Let $G$ be an algebraic group over a number field $K$, whose neutral component is solvable. Let $P$ be a principal $G$-bundle over a smooth connected variety $B$ over $K$, and let $\sG$ be a flat $G$-equivariant connection on $P \to B$.
Suppose that $\sG$ is closed under $p$-th powers for almost all primes $p$.
Then $\sG$ is algebraically integrable.
\end{thm}

\begin{prop}\label{prop:grothendieck_katz_projective_connection}
Let $A$ be a complex abelian variety, and let $\psi\colon X \to A$ be a $\mathbb{P}^1$-bundle. Suppose that there exists a flat holomorphic connection $\sG \subset T_X$ on $\psi$, and that $\sG$ is closed under $p$-th powers for almost all primes $p$. Then $\sG$ is algebraically integrable.  
\end{prop}

\begin{proof}
Set $n:= \dim X$. Since $\psi$ admits a flat holomorphic connection, there is a representation
$$\rho\colon \pi_1(A) \to \textup{PGL}(2,\mathbb{C})$$ of $\pi_1(A)$ such that $X\cong(\mathbb{C}^{n-1} \times \mathbb{P}^1)/\pi_1(A)$, where the group $\pi_1(A)$ acts diagonally on $\mathbb{C}^{n-1}\times \mathbb{P}^1$. 
Moreover, $\sG$ is induced by the projection $\mathbb{C}^{n-1} \times \mathbb{P}^1 \to \mathbb{P}^1$. 

Let $H \subseteq \textup{PGL}(2,\mathbb{C})$ be the Zariski closure of the image of $\rho$.
Note that $\sG$ is algebraically integrable if and only if $H$ is finite.

We argue by contradiction and assume that $\dim H >0$. We use the notation introduced in Example \ref{example:suspension}.

There is a dense open set $X^\circ \subset X$ that is a principal bundle over $A$ with structure group 
$K=\mathbb{G}_m$ or $K=\mathbb{G}_a$, and 
$\sG_{|X^\circ}$ yields a flat invariant connection on $\psi^\circ:=\psi_{|X^\circ}\colon X^\circ \to A$, which is obviously closed under $p$-th powers for almost all primes $p$.

In addition, there is a flat rank two vector bundle $(\sE,\nabla)$ on $A$ such that $X \cong \mathbb{P}_A(\sE)$ over $A$. Moreover, if 
$\sG_\nabla$ denotes the linear Ehresmann connection on the total space $E$ of $\sE^*$ induced by $\nabla$, then 
$\sG$ is the projection of ${\sG_\nabla}_{|E^\times}$ under the natural map $\pi\colon E^\times \to X$, where 
$E^\times:=E\setminus 0_X$ (see Example \ref{example:suspension} below). 
Let $L \subset X$ be a leaf of $\sG$. Since ${\sG_\nabla}_{|E^\times}$ is invariant under the natural $\mathbb{C}^*$-action on $E^\times$, we have $\pi^{-1}(L)\cong L \times \mathbb{C}^*$ and the restriction of $\sG_\nabla$ to $\pi^{-1}(L)$ is given by the projection $L \times \mathbb{C}^* \to \mathbb{C}^*$ . It follows that $\sG$ is algebraically integrable if and only if so is $\sG_\nabla$.

To show Proposition \ref{prop:grothendieck_katz_projective_connection}, let $R$ be a subring of $\mathbb{C}$,
finitely generated over $\mathbb{Q}$, and let $\bA$ (resp. $\bK$, $\bX$ and $\bX^\circ$) be a smooth projective model of $A$ (resp. $K$, $X$ and $X^\circ$)
over $\bT:=\textup{Spec}\, R$. We may assume that there exist a rank two vector bundle $\sbfE$ on $\bX$ and 
a relative flat connection $\bm{\nabla}\colon \sbfE \to \Omega^1_{\bX/\bA}\otimes\sbfE$ on $\sbfE$ such that 
$(\sbfE_{\mathbb{C}},\bm{\nabla}_\mathbb{C})\cong (\sE,\nabla)$ and such that 
$\bX=\mathbb{P}_\bA(\sbfE)$.
Let $\sbfG$ be the subbundle
of $T_{\bX/\bT}$ induced by $\bm{\nabla}$, so that $\sbfG_\mathbb{C}$ coincides with $\sG$. Shrinking $\bT$, if necessary, we may assume without loss of generality, that for any closed point 
$t\in \bT$, $\sbfG_{\bar{t}}$ is closed under $p$-th powers for almost all primes $p$. 
We may finally also assume that $\bX^\circ_{\bar{t}}$ is a principal bundle over $\bA_{\bar{t}}$ with structure group
$\bK_{\bar{t}}$, and that the restriction of 
$\sbfG_{\bar{t}}$ to $\bX_{\bar{t}}^\circ$ is $\bK_{\bar{t}}$-equivariant.

Now,
when $t\in \bT$ is a closed point, its residue field is a number field, and hence $\sbfG_{\bar{t}}$
is algebraically integrable by Theorem \ref{thm:bost} above. Let 
$\sbfG_{\bm{\nabla}}$ denotes the linear Ehresmann connection on the total space $\bE$ of $\sbfE^*$ induced by $\bm{\nabla}$. By construction,
$\sbfG$ is the projection of ${\sbfG_{\bm\nabla}}_{|\bE^\times}$ under the map $\bE^\times \to \bX$, where 
$\bE^\times:=\bE\setminus 0_{\bX}$. Moreover, if $t\in \bT$ is a closed point, then $(\sbfG_{\bm{\nabla}})_{\bar{t}}$ is algebraically integrable since $\sbfG_{\bar{t}}$ does. In other words, the flat connection $(\sbfE_{\bar{t}},{\bm{\nabla}}_{\bar{t}})$ has finite monodromy representation. By \cite[Theorem 7.2.2]{andre}, we conclude that $(\sE,\nabla)$ has finite monodromy representation as well, so that $\dim H = 0$. This yields a contradiction, completing the proof of the proposition.
\end{proof}

The following is the main result of this section.

\begin{thm}\label{thm:grothendieck_katz}
Let $X$ be a normal complex projective variety with terminal singularities, and let $\sG$ be a codimension one foliation on 
$X$. Suppose that $\sG$ is canonical, and that it is closed under $p$-th powers for almost all primes $p$.
Suppose furthermore that $K_X$ is not pseudo-effective, and that $K_\sG\equiv 0$. 
Then $\sG$ is algebraically integrable. 
\end{thm}

Before proving Theorem \ref{thm:grothendieck_katz} below, we note the following corollary.

\begin{cor}\label{cor:grothendieck_katz}
Let $X$ be a normal complex projective variety with canonical singularities, and let $\sG$ be a codimension one foliation on 
$X$.
Suppose that $\sG$ is canonical with $K_\sG$ Cartier and $K_\sG\equiv 0$ and that $\sG$ is closed under $p$-th powers for almost all primes $p$. Suppose in addition that $K_X$ is not pseudo-effective. Then 
$\sG$ is algebraically integrable. 
\end{cor}

\begin{proof}
Let $\beta \colon Z \to X$ be a $\mathbb{Q}$-factorial terminalization of $X$. By Proposition \ref{prop:terminalization_canonical_singularities}, $\beta^{-1}\sG$ is canonical with $K_{\beta^{-1}\sG}\sim_\mathbb{Z}\beta^*K_\sG$. Moreover, $\beta^{-1}\sG$ is obviously closed under $p$-th powers for almost all primes $p$.
Finally, $K_Z$ is not pseudo-effective since $\beta_*K_Z\sim_\mathbb{Z} K_X$ and $K_X$ is not pseudo-effective by assumption.
The statement now follows from Theorem \ref{thm:grothendieck_katz} applied to $\beta^{-1}\sG$.
\end{proof}

\begin{proof}[Proof of Theorem \ref{thm:grothendieck_katz}] For the reader's convenience, the proof is subdivided into a number of relatively independent
steps.

\medskip
\noindent\textit{Step 1.} Let $\beta \colon Z \to X$ be a $\mathbb{Q}$-factorialization of $X$. By Lemma 
\ref{lemma:singularities_birational_morphism}, $\beta^{-1}\sG$ is canonical with $K_{\beta^{-1}\sG}\sim_\mathbb{Q}\beta^*K_\sG$. Moreover, $\beta^{-1}\sG$ is obviously closed under $p$-th powers for almost all primes $p$. Finally, $K_Z$ is not pseudo-effective since $\beta_*K_Z\sim_\mathbb{Z} K_X$ and $K_X$ is not pseudo-effective by assumption.
Thus, replacing $X$ by $Z$, if necessary, we may assume without loss of generality that the following holds.

\begin{assumption}
The variety $X$ has $\mathbb{Q}$-factorial terminal singularities.
\end{assumption}

\noindent\textit{Step 2.} Since $K_X$ is not pseudo-effective by assumption, we may run a minimal model program for $X$ and end with a Mori fiber space (see \cite[Corollary 1.3.3]{bchm}). Therefore, there exists a sequence of maps

\begin{center}
\begin{tikzcd}[row sep=large, column sep=large]
X:=X_0 \ar[r, "{\phi_0}", dashrightarrow] & X_1 \ar[r, "{\phi_1}", dashrightarrow] & \cdots \ar[r, "{\phi_{i-1}}", dashrightarrow] & X_i \ar[r, "{\phi_i}", dashrightarrow] & X_{i+1} \ar[r, "{\phi_{i+1}}", dashrightarrow] & \cdots \ar[r, "{\phi_{m-1}}", dashrightarrow] & X_m \ar[d, "{\psi_m}"]\\
&&&&&& Y
\end{tikzcd}
\end{center}

\noindent where the $\phi_i$ are either divisorial contractions or flips, and $\psi_m$ is a Mori fiber space. The spaces $X_i$ are normal, $\mathbb{Q}$-factorial, and $X_i$ has terminal singularities for all $0\le i \le m$. 
Let $\sG_i$ be the foliation on $X_i$ induced by $\sG$. By \cite[Lemma 3.2.5]{kmm}, we have $K_{\sG_i}\equiv 0$. Moreover, by Lemma \ref{lemma:singularities_birational_morphism}, $\sG_i$ has canonical singularities. 
Finally, one readily checks that $\sG_m$ is closed under $p$-th powers for almost all primes $p$.
Hence, we may assume without loss of generality that the following holds.

\begin{assumption}
There exists a Mori fiber space $\psi \colon X \to Y$.
\end{assumption}

\noindent\textit{Step 3.} First, we show that $\dim X - \dim Y = 1$. We argue by contradiction and assume
that $\dim X - \dim Y \ge 2$. Let $F$ be a general fiber of $\psi$. Note that $F$ has terminal singularities, and that $K_F \sim_\mathbb{Z} {K_X}_{|F}$ by the adjunction formula. Moreover, $F$ is a Fano variety by construction. 
Note also that $F$ is smooth in codimension two since it has terminal singularities.
Let $\sH$ be the foliation on $F$ induced by $\sG$. Since $c_1(\sN_\sG)\equiv -K_X$ is relatively ample, we see that $\sH$ has codimension one.
By Proposition \ref{prop:bertini}, we have $K_\sH\sim_\mathbb{Z} {K_\sG}_{|F} - B$ for some effective Weil divisor $B$ on $F$. Suppose that $B\neq 0$.
Applying \cite[Theorem 4.7]{campana_paun15} to the pull-back of $\sH$ on a resolution of $F$, we see that $\sH$ is uniruled. This implies that $\sG$ is uniruled as well since $F$ is general. But this contradicts Proposition \ref{proposition:canonical_versus_uniruled}, and shows that $B=0$. By Proposition \ref{proposition:canonical_versus_uniruled} applied to $\sH$, we see that $\sH$ is canonical. Finally, one readily checks that $\sH$ is closed under $p$-th powers for almost all primes $p$.

Let $S \subseteq F$
be a two dimensional complete intersection of general elements of a very ample linear system $|H|$ on $F$. We may assume wihtout loss of generality that $S$ is smooth and contained in 
$F_{\textup{reg}}$. Let $\sL$ be the foliation by curves on $S$ induced by $\sH$. By Proposition \ref{prop:bertini}, we have $\det\sN_\sL \cong (\det\sN_\sH)_{|S}$. In particular, we have 
\begin{equation}\label{BB_formula0}
c_1(\sN_\sL)^2 = K_F ^2 \cdot H^{\dim F -2} > 0.
\end{equation}
On the other hand, by the Baum-Bott formula (see \cite[Chapter 3, Theorem 1]{brunella}), we have 
\begin{equation}\label{BB_formula}
c_1(\sN_\sL)^2 =\sum_x \textup{BB}(\sL,x),
\end{equation}
where $x$ runs through all singular points of $\sL$, and $\textup{BB}(\sL,x)$ denotes the Baum-Bott index of $\sL$ at $x$ (we refer to \cite[Chapter 3]{brunella} for this notion). 

Let $x$ be a singular point of $\sL$. If $\sH$ is regular at $x$, then there is a holomorphic function $f$ defined in a neighborhood of $x$ such that $df$ vanishes at finitely many points and such that $df$ defines $\sL$ (see Proposition \ref{prop:bertini}). It follows that 
$\textup{BB}(\sL,x)=0$. Suppose now that $x$ is a singular point of $\sH$. 
By \cite[Corollary 7.8]{lpt}, 
$\sH$ is defined at $x$ by the local $1$-form $\omega=pz_2dz_1 + qz_1dz_2$, where $p$ and $q$ are positive integers and $(z_1,\ldots,z_s)$ are analytic coordinates on $F$ centered at $x$.
This implies in particular that the singular locus of $\sH$ is smooth in a neighborhood of $x$. By general choice of $S$,
we may assume without loss of generality that $S$ intersects the singular locus of $\sH$ transversely.
It follows that $\sL$ is defined at $x$ by the local $1$-form $\omega=pvdu + qudv$, where $(u,v)$ are analytic coordinates on $S$ centered at $x$, and hence $\textup{BB}(\sL,x)=-\frac{(p-q)^2}{pq}\le 0$.
In either case, we have $\textup{BB}(\sL,x) \le 0$, and hence $c_1(\sN_\sL)^2 \le 0$ by equation \eqref{BB_formula}.
But this contradicts inequality \eqref{BB_formula0} above, and shows that $\dim X - \dim Y = 1$.

\medskip

Next, we show the following.

\begin{claim}\label{claim:bundle}
There exists an open subset $Y^\circ \subseteq Y_{\textup{reg}}$ 
with complement of codimension at least two such that $\psi^\circ:=\psi_{|X^\circ}$ is a $\mathbb{P}^1$-bundle, and such that 
$\sG_{|X^\circ}$ yields a flat Ehresmann connection on $\psi^\circ$, where $X^\circ:=\psi^{-1}(Y^\circ)$.
\end{claim}

\begin{proof}[Proof of Claim \ref{claim:bundle}]
Recall that $X$ is smooth in codimension two since it has terminal singularities.
Since $\dim X - \dim Y = 1$, there exists an open subset $Y^\circ \subseteq Y_{\textup{reg}}$ 
with complement of codimension at least two such that 
$X^\circ:=\psi^{-1}(Y^\circ) \subseteq X_{\textup{reg}}$. From \cite[Theorem 4.1]{andreatta_wisniewski_view}, we conclude that
$\psi^\circ:=\psi_{|X^\circ}$ is a conic bundle. 

Suppose that there is a codimension two irreducible component of the singular set of $\sG$ which is mapped onto a divisor $D$ by $\psi$.

Suppose first that $\psi^\circ$ is smooth over the generic point of $D$. Let $B \subset Y$ be a germ of analytic curve passing through a general point $y$ of $D$ and transverse to $D$ at $y$, set $S =: (\psi^\circ)^{-1}(B) \subset X$, and let $\sL$ be the foliation by curves on $S$ induced by $\sG$. Denote by $\pi \colon S \to B$ the restriction of $\psi^\circ$ to $S$, and
set $C:=\pi^{-1}(y) \cong \mathbb{P}^1$. We have $c_1(\sN_\sG)\cdot C =2$ and either $C$ is tangent to $\sG$, or $C$ is transverse to $\sG$ along $C$. In the latter case, $\sG$ must be regular along $C$, yielding a contradiction. This shows 
$C$ is tangent to $\sG$. We may assume without loss of generality that $\sL$ is everywhere transverse to $\pi$ away from $C$, and that $\sG$ intersects $S$ transversely at a general point of $C$ (see Proposition \ref{prop:bertini}).
This implies in particular that $C$ is $\sL$-invariant.
Let $x\in C$ be a singular point of $\sL$. If $\sG$ is regular at $x$, 
then there is a holomorphic function $f$ defined in a neighborhood of $x$ in $S$ such that $df$ vanishes at finitely many points and such that $df$ defines $\sL$ (see Proposition \ref{prop:bertini}). 
Suppose that $C$ is given at $x$ by equation $t=0$, and that $f(x)=0$. Then $f=tg$ for some local holomorphic function $g$ on $S$ at $x$, and
$-\textup{CS}(\sL,C,x)$ is equal to the vanishing order of $g_{|C}$ at $x$, where $\textup{CS}(\sL,C,x)$ denotes the Camacho-Sad index (we refer to \cite[Chapter 3]{brunella} for this notion). In particular, $\textup{CS}(\sL,C,x) \le 0$.
Suppose now that $\sG$ is singular at $x$. 
By \cite[Corollary 7.8]{lpt}, 
$\sG$ is defined at $x$ by the local $1$-form $\omega=pz_2dz_1 + qz_1dz_2$, where $p$ and $q$ are positive integers and $(z_1,\ldots,z_n)$ are analytic coordinates on $X$ centered at $x$.
This implies in particular that the singular locus of $\sG$ is smooth in a neighborhood of $x$. By general choice of $B$, we may also assume without loss of generality that $S$ intersects the singular locus of $\sG$ transversely. It follows that $\sL$ is defined at $x$ by the local $1$-form 
$\omega=pvdu + qudv$, where $(u,v)$ are local coordinates on $S$ centered at $x$. Since $C$ is $\sL$-invariant, we may also assume that $C$ is given at $x$ by equation $u=0$.
But then $\textup{CS}(\sL,C,x)=-\frac{p}{q}<0$. 
On the other hand, by the Camacho-Sad formula, we have
$$C^2 = \sum_x \textup{CS}(\sL,C,x),$$
where $x$ runs through all singular points of $\sL$. This yields a contradiction since $C^2=0$.

Suppose now that $D$ is an irreducible component of the critical set of $\psi^\circ$, and let $y\in D$ be a general point.
Let also $C=C_1\cup C_2$ be the corresponding fiber of $\psi^\circ$. Note that $C_1 \neq C_2$, and that $C_i$ is tangent to $\sG$ since $c_1(\sN_\sG)\cdot C_i =1$. We argue as in the previous case. Consider a general germ of analytic curve passing through $y$, set $S =: (\psi^\circ)^{-1}(B) \subset X$, and let $\sL$ be the foliation by curves on $S$ induced by $\sG$. Denote also by $\pi \colon S \to B$ the restriction of $\psi^\circ$ to $S$. Note that $S$ is smooth by general choice of $B$.
As above, we may assume without loss of generality that $\sL$ is everywhere transverse to $\pi$ away from $C$, and that $\sG$ intersects $S$ transversely at a general point of $C$.
This again implies that $C_i$ is $\sL$-invariant. Let $x\in C$ be a singular point of $\sL$. If $x$ is a regular point of $\sG$, then we have $\textup{CS}(\sL,C,x)\le 0$ as before. If $\sG$ is singular at $x$, then 
$\textup{CS}(\sL,C,x)=-\frac{p}{q}<0$ for some positive integers $p$ and $q$ if $x \neq C_1\cap C_2$, and 
$\textup{CS}(\sL,C,x)=-\frac{(p-q)^2}{pq}$ otherwise. Together with the Camacho-Sad formula, this yields a contradiction since $C^2=0$.

It follows that the singular set of $\sG$ is mapped in codimension at least two in $Y$.
Let $C\cong \mathbb{P}^1$ be a smooth fiber of $\psi^\circ$, and suppose that $\sG$ is regular along $C$. 
By \cite[Lemma 3.2]{druel_bbcd2}, we see that $C$ is not tangent to $\sG$ since $c_1(\sN_\sG)\cdot C = 2$.
This immediately implies that $\sG$ is transverse to $\psi^\circ$ at any point of $C$, 
finishing the proof of Claim \ref{claim:bundle}.
\end{proof}

\noindent\textit{Step 4.} 
By \cite[Corollary 4.5]{fujino99} and Lemma \cite[Lemma 5.1.5]{kmm}, $Y$ has $\mathbb{Q}$-factorial
klt singularities. Since $\textup{codim}\, Y \setminus Y^\circ \ge 2$ and $K_\sG\equiv 0$, we must have $K_Y\equiv 0$. 
Applying \cite[Corollary V 4.9]{nakayama04}, we conclude that $K_Y$ is torsion.
Let $Y_1 \to Y$ be the index one canonical cover, which is quasi-\'etale (\cite[Definition 2.52]{kollar_mori}).
By construction, $K_{Y_1}\sim_\mathbb{Z}0$. In particular, $Y_1$ has canonical singularities.
By Theorem \ref{thm:kawamata_abelian_factor} applied to $Y_1$, we see that there exists an abelian variety as well as
a projective variety $Z$ with $K_{Z}\sim_\mathbb{Z}0$ and 
$\wt q(Z) = 0$, and a quasi-\'etale cover 
$f\colon  A \times Z\to Y$. 

Recall that $f$ branches only on the singular set of $Y$, so that $f^{-1}(Y^\circ)$ is smooth. On the other hand, 
since $f^{-1}(Y^\circ)$ has complement of codimension at least two in $A\times Z_\textup{reg}$, we have 
$\pi_1\big(A\times Z_\textup{reg}\big) \cong \pi_1\big(f^{-1}(Y^\circ)\big)$.
Now, consider the representation
$$\rho\colon \pi_1\big(A\times Z_\textup{reg}\big) \cong \pi_1\big(f^{-1}(Y^\circ)\big) \to 
\pi_1\big(Y^\circ\big) \to \textup{PGL}(2,\mathbb{C})$$
induced by $\sG_{|X^\circ}$. By \cite[Theorem I]{GGK}, the induced representation 
$$ \pi_1\big(Z_\textup{reg}\big) \to \pi_1\big(A\big)\times\pi_1\big(Z_\textup{reg}\big)\cong
\pi_1\big(A\times Z_\textup{reg}\big) \to \textup{PGL}(2,\mathbb{C})$$
has finite image.
Thus, replacing $Z$ with a quasi-\'etale cover, if necessary, we may assume without loss of generality that 
$\rho$ factors through the projection $\pi_1\big(A\times Z_\textup{reg}\big) \to \pi_1(A)$. 
Let $P$ be the corresponding $\mathbb{P}^1$-bundle over $A$. The natural projection $P \to A$ comes with a flat connection
$\sG_P \subset T_P$. By the GAGA theorem, $P$ is a projective variety. By assumption, its pull-back to 
$A\times Z_\textup{reg}$ agrees with $f^{-1}(Y^\circ)\times_{Y^\circ} X^\circ$ over $f^{-1}(Y^\circ)$. Moreover, the pull-backs 
on $A\times Z_\textup{reg}$ of the foliations $\sG$ and $\sG_P$ agree as well, wherever this makes sense. In particular, $\sG$ is algebraically integrable if and only if so is $\sG_P$.
Now, one readily checks that $\sG_P$ is closed under $p$-th powers for almost all primes $p$.
Theorem \ref{thm:grothendieck_katz} then follows from Proposition \ref{prop:grothendieck_katz_projective_connection}.\end{proof}

\section{Algebraic integrability, II}\label{section:algebraic_integrability_2}

In this section, we address codimension one foliations with numerically trivial canonical class on mildly singular varieties $X$ with 
pseudo-effective canonical divisor. An analogue of the Bogomolov vanishing theorem then says that $X$ has numerical dimension 
$\nu(X) \le 1$ (see Lemma \ref{lemma:bogomolov}).  
We first describe codimension one foliations with numerically trivial canonical class on varieties with $\nu(X)=0$
(see Lemma \ref{lemma:K_torsion} and Proposition \ref{prop:nu_zero_versus_torsion}). We then give two algebraicity criteria for leaves of algebraic foliations on varieties with $\nu(X)=1$ (see Theorem
\ref{thm:algebraic_integrability_nu_un_regular} and Theorem \ref{thm:algebraic_integrability_nu_un}).

\begin{lemma}\label{lemma:K_torsion}
Let $X$ be a normal complex projective variety with klt singularities, and let $\sG$ be a codimension one foliation on $X$ with $K_\sG$ $\mathbb{Q}$-Cartier and $K_\sG\equiv 0$. Suppose in addition that $K_X \equiv 0$. There exist an abelian variety $A$, a normal projective variety $Z$ with $K_Z\sim_\mathbb{Z}0$ and $\wt q(Z)=0$, and a quasi-\'etale cover $f \colon A \times Z \to X$ such that $f^{-1}\sG$ is the pull-back of a codimension one linear foliation on $A$ via the projection $A \times Z \to A$.
\end{lemma}

\begin{proof}
By \cite[Corollary V 4.9]{nakayama04}, $K_X$ is torsion. Let $f \colon X_1 \to X$ be the associated cyclic cover, which is quasi-\'etale (see \cite[Definition 2.52]{kollar_mori}). Recall from Fact \ref{fact:quasi_etale_cover_and_singularities} that $X_1$ is klt.
Notice that $K_{f^{-1}\sG}\sim_\mathbb{Z}f^*K_\sG$. In particular, $K_{f^{-1}\sG}$ is $\mathbb{Q}$-Cartier, and $K_{f^{-1}\sG}\equiv 0$.
To prove the statement, we can therefore assume without loss of generality that 
$K_X \sim_\mathbb{Z} 0$. Then $X$ has canonical singularities, so that Theorem \ref{thm:kawamata_abelian_factor} applies. Replacing $X$ by a further quasi-\'etale cover, if necessary, we may therefore assume that there exist an abelian variety $A$, a normal projective variety $Z$ with $K_Z\sim_\mathbb{Z}0$ and $\wt q(Z)=0$ such that $X= A \times Z$.

Let $\beta\colon Z_1 \to Z$ be a resolution of singularities, and let $m$ be a positive integer such that  
$\sL:=\big(\det\sN_{\sG}\big)^{[\otimes m]}$ is a line bundle. Note that $c_1(\sL)\equiv 0$ by assumption.
Since $\wt q(Z)=0$, we have $h^0\big(Z,\Omega_Z^{[1]}\big)=0$ by Hodge symmetry for klt spaces (see \cite[Proposition 6.9]{gkp_bo_bo}), and hence $h^0\big(Z_1,\Omega_{Z_1}^{1}\big)=0$. It follows that 
$\beta^*(\sL_{|Z})$ is torsion, and hence, $\sL_{|Z_{\textup{reg}}}$ is torsion as well.
Replacing $Z$ by a further quasi-\'etale cover, if necessary, we may therefore assume that 
$\big(\det\sN_{\sG}\big)_{|Z} \cong \sO_Z$. In particular, we see that 
$\det\sN_{\sG}$ is a line bundle. One then readily checks that 
$\det\sN_{\sG} \cong \sO_{A \times Z}$. This implies that $\sG$ is defined by a global $1$-form on $X$.
Since $h^0\big(Z,\Omega_Z^{[1]}\big)=0$, $\sG$ is the pull-back of a codimension one linear foliation on $A$ via the projection $A\times Z \to A$. This finishes the proof the lemma.
\end{proof}

We will use Lemma \ref{lemma:K_torsion}
together with Proposition \ref{prop:nu_zero_versus_torsion} below.

\begin{prop}\label{prop:nu_zero_versus_torsion}
Let $X$ be a normal complex projective variety with klt singularities, and let $\sG$ be a codimension one foliation on $X$ with 
canonical singularities and $K_\sG\equiv 0$. If $\nu(X)=0$, then $K_X$ is torsion.
\end{prop}

\begin{proof}
Applying Proposition \ref{prop:splitting_algebraic_transcendental}, we may assume without loss of generality that 
there is no positive-dimensional algebraic subvariety tangent to $\sG$ passing through a general point of
$X$.

We argue by contradiction and assume that $K_X$ is not torsion. By \cite[Corollary V 4.9]{nakayama04}, we have $\kappa(X)=0$. Replacing $X$ by a $\mathbb{Q}$-factorialization (see Paragraph \ref{say:q_factorialization}), we may assume that $X$ is $\mathbb{Q}$-factorial by Lemma \ref{lemma:singularities_birational_morphism}. 
By \cite[Th\'eor\`eme 1.2]{deczar}, we may run a minimal model program for $X$ and end with a minimal model.
Therefore, there exists a sequence of maps

\begin{center}
\begin{tikzcd}[row sep=large, column sep=large]
X:=X_0 \ar[r, "{\phi_0}", dashrightarrow] & X_1 \ar[r, "{\phi_1}", dashrightarrow] & \cdots \ar[r, "{\phi_{i-1}}", dashrightarrow] & X_i \ar[r, "{\phi_i}", dashrightarrow] & X_{i+1} \ar[r, "{\phi_{i+1}}", dashrightarrow] & \cdots \ar[r, "{\phi_{m-1}}", dashrightarrow] & X_m
\end{tikzcd}
\end{center}

\noindent where the $\phi_i$ are either divisorial contractions or flips, and $K_{X_m}$ is torsion. The spaces $X_i$ are normal, $\mathbb{Q}$-factorial, and $X_i$ has klt singularities for all $0\le i \le m$. 
Let $\sG_i$ be the foliation on $X_i$ induced by $\sG$. By \cite[Lemma 3.2.5]{kmm}, we have $K_{\sG_i}\equiv 0$. 
Moreover, by Lemma \ref{lemma:singularities_birational_morphism}, $\sG_i$ has canonical singularities. 
Note that since $K_X$ is not torsion, there is some $i$ such that $\phi_i$ is a divisorial contraction.
Let $i_0$ be the largest integer $i$ such that $\psi_i$ is a divisorial contraction. 
By replacing $X$ by $X_{i_0}$, if necessary, we may assume that there exist a divisorial 
extremal contraction $ \phi\colon X \to Y$ and that 
$K_Y$ is torsion. 

Let $\sE$ denotes the foliation on $Y$ induced by $\sG$.
By Lemma \ref{lemma:K_torsion}, there exist an abelian variety $A$ and a quasi-\'etale cover $f \colon A \to Y$ such that $f^{-1}\sE$ is a codimension one linear foliation on $A$. 
Let $X_1$ be the normalization of the fiber product $X\times_Y A$, and denote by $\phi_1 \colon X_1 \to A$ and
$f_1\colon X_1 \to X$ the natural morphisms. Let $E$ be the exceptional locus of $\phi$. Observe that any irreducible component of $f_1^{-1}(E)$ is invariant under $f_1^{-1}\sG$ since $f^{-1}\sE$ is a regular foliation. This implies that
$E$ is invariant under $\sG$. Applying Lemma \ref{lemma:pull_back_fol_and_finite_cover}, we see that
$K_{f_1^{-1}\sG}\equiv 0$. Then the map
$f_1^{-1}\sG \to \phi_1^*\big(f^{-1}\sE\big)$ induced by the tangent map $T\phi_1\colon T_{X_1} \to \phi_1^*T_A$
is an isomorphism. In particular, $T\phi_1$ has rank at least $\dim A-1=\dim X_1 -1$ at any point in $f_1^{-1}(E)$, yielding a contradiction.
This shows that $K_X$ is torsion, completing the proof of the proposition.
\end{proof}

Next, we address weakly regular codimension one foliations with trivial canonical class on varieties $X$ with $\nu(X)=1$.

\begin{thm}\label{thm:algebraic_integrability_nu_un_regular}
Let $X$ be a normal complex projective variety with klt singularities, and let $\sG$ be a weakly regular codimension one foliation on $X$. Suppose that $\sG$ is canonical with $K_\sG\equiv 0$. Suppose in addition that $X$ is smooth in codimension two, and that $\nu(X)=1$. Then $\sG$ is algebraically integrable. 
\end{thm}

\begin{proof}
For the reader's convenience, the proof is subdivided into a number of relatively independent
steps.

\medskip
\noindent\textit{Step 1.} 
Applying Proposition \ref{prop:splitting_algebraic_transcendental} together with Corollary \ref{cor:regular_quasi_etale} and Lemma \ref{lemma:properties:regular}, we may assume without loss of generality that 
there is no positive-dimensional algebraic subvariety tangent to $\sG$ passing through a general point of
$X$. 

\medskip
\noindent\textit{Step 2.} 
Let $\beta \colon Z \to X$ be a resolution of singularities with exceptional set $E$, and suppose that $E$ is a divisor with simple normal crossings. Suppose in addition that the restriction of $\beta$ to $\beta^{-1}(X_\textup{reg})$ is an isomorphism. 
Let $E_1$ be the reduced divisor on $Z$ whose support is
the union of all irreducible components of $E$ that are invariant under $\beta^{-1}\sG$.
Note that $-c_1(\sN_\sG)\equiv K_X$ by assumption. By Proposition \ref{prop:numerical_dimension} and Remark \ref{rem:exceptional_set}, there exists a rational number $0 \le \varepsilon <1$ such that 
$$\nu\big(-c_1(\sN_{\beta^{-1}\sG})+\varepsilon E_1\big)=\nu\big(-c_1(\sN_\sG)\big)=1.$$ 
By \cite[Theorem 6]{touzet_conpsef} applied to $\beta^{-1}\sG$, we may assume that there exists an arithmetic irreducible lattice $\Gamma$ of $\textup{PSL}(2,\mathbb{R})^N$ for some integer $N \ge 2$, as well as a morphism $\phi\colon Z \to \mathfrak{H}:=\mathbb{D}^N/\Gamma$ of quasi-projective varieties such that $\sG=\phi^{-1}\sH$, where $\sH$ is a weakly regular codimension one foliation on  
$\mathfrak{H}$ induced by one of tautological foliations on the polydisc $\mathbb{D}^N$. Note that $\phi$ is generically finite by Step 1.

By a result of Selberg, there exists a torsion-free subgroup $\Gamma_1$ of $\Gamma$ of finite index.
Set $\mathfrak{H}_1:=\mathbb{D}^N/\Gamma_1$, and denote by $\pi \colon \mathfrak{H}_1 \to \mathfrak{H}$ the natural finite morphism. Recall that $\mathfrak{H}$ has isolated quotient singularities. It follows that $\pi$ is a quasi-\'etale cover since $N \ge 2$.

Let $F \subset Z$ be a prime divisor, and assume that $F$ is not $\beta$-exceptional.
Set $G:=\beta(F)$. We show that $\dim \phi(F) \ge 1$. We argue by contradiction, and assume that 
$\dim \phi(F) = 0$. One readily checks that there is only one separatrix
for $\sH$ at any (singular) point.
This implies that $F$ must be invariant under $\beta^{-1}\sG$ since otherwise, $\phi(Z)$ would be contained in a leaf of $\sH$. 
But this is impossible since $\beta^{-1}\sG$ has codimension one.
Set $X^\circ:=X_{\textup{reg}}$.
By assumption, $\sG_{|X^\circ}$ is regular, and hence $G^\circ:=G \cap X^\circ$ is a smooth hypersurface with  
normal bundle $\sN_{G^\circ/X^\circ}\cong {\sN_\sG}_{|G^\circ}$.
Let $S \subseteq X$ be a two dimensional complete intersection of general elements of a very ample linear system $|H|$ on $X$. We may assume without loss of generality that $S$ is smooth and contained in 
$X^\circ$, so that $\phi \circ \beta^{-1}$ is regular in a neighborhood of $S$. 
Set $C:= S \cap G$. Then $G\cdot C<0$ since $C$ is contracted by the generically finite morphism ${\phi \circ \beta^{-1}}_{|S}$. On the other hand, we must have $G \cdot C = 0$ by \cite[Lemma 3.2]{druel_bbcd2}, yielding a contradiction.
This proves that $\dim \phi(F) \ge 1$.

\medskip
\noindent\textit{Step 3.} 
By Step 2, the natural map $Z\times_\mathfrak{H} \mathfrak{H_1} \to Z$ is a quasi-\'etale cover away from the exceptional locus of $\beta$. In particular, it induces a quasi-\'etale cover $f_1\colon X_1 \to X$. Let $Z_1\to Z\times_\mathfrak{H} \mathfrak{H_1}$ be a resolution of singularities. 
We obtain a diagram

\begin{center}
\begin{tikzcd}[row sep=large, column sep=large]
X_1\ar[d, "{f_1, \textup{ quasi-\'etale}}"'] && Z_1 \ar[ll, "{\beta_1,\textup{ birational}}"'] \ar[d, "{g_1}"]\ar[rrr, "{\phi_1,\textup{ generically finite}}"] & & & \mathfrak{H}_1\ar[d, "{\pi,\textup{ quasi-\'etale}}"]\\
X && Z \ar[ll, "{\beta,\textup{ birational}}"]\ar[rrr, "{\phi,\textup{ generically finite}}"']&&& \mathfrak{H}.
\end{tikzcd}
\end{center}
Let $\sE_i$ the codimension one regular foliations on $\mathfrak{H}_1$ induced by the tautological foliations 
on $\mathbb{D}^N$ so that $\Omega^1_{\mathfrak{H}_1}\cong \bigoplus_{1\le i \le N} \sE^*_i$. 
Set $n:=\dim X$. We may assume without loss of generality that the natural map 
$\bigoplus_{1\le i \le n} \phi_1^*\sE^*_i \to \Omega_Z^1$ is generically injective.
Now observe that the line bundle $\sN_{\sE_i}^*$ is hermitian semipositive, so that 
$\phi_1^*\sN^*_{\sE_i}$ is nef. On the other hand,
we have $c_1(\phi_1^*\sN^*_{\sE_1})\cdot\cdots\cdot c_1(\phi_1^*\sN^*_{\sE_n})>0$.
This immediately implies that $\kappa(Z_1)=\nu(Z_1)=\dim Z_1$. It follows that $\kappa(X_1)=\nu(X_1)=\dim X_1$ since $\beta_1$ is a birational morphism. Applying \cite[Proposition 2.7]{nakayama04}, we see that 
$\nu(X)=\nu(X_1)=\dim X_1=\dim X$ since $K_{X_1}\sim_\mathbb{Q}f_1^* K_X$. It follows that $\dim X =1$ since $\nu(X)=1$ by assumption. This finishes the proof of the theorem.
\end{proof}

The following is the main result of this section.

\begin{thm}\label{thm:algebraic_integrability_nu_un}
Let $X$ be a normal complex projective variety with terminal singularities, and let $\sG$ be a codimension one foliation on $X$. Suppose that $\sG$ is canonical with $K_\sG\equiv 0$ and that $\nu(X)=1$. Suppose in addition that $\sG$ is closed under $p$-th powers for almost all primes $p$. Then $\sG$ is algebraically integrable. 
\end{thm}

Before proving Theorem \ref{thm:algebraic_integrability_nu_un} below, we note the following corollary.

\begin{cor}\label{cor:algebraic_integrability_nu_un}
Let $X$ be a normal complex projective variety with canonical singularities, and let $\sG$ be a codimension one foliation on 
$X$.
Suppose that $\sG$ is canonical with $K_\sG$ is Cartier and $K_\sG\equiv 0$ and that $\sG$ is closed under $p$-th powers for almost all primes $p$. Suppose in addition that $\nu(X)=1$. Then 
$\sG$ is algebraically integrable. 
\end{cor}

\begin{proof}
Let $\beta \colon Z \to X$ be a $\mathbb{Q}$-factorial terminalization of $X$. By Proposition \ref{prop:terminalization_canonical_singularities}, $\beta^{-1}\sG$ is canonical with $K_{\beta^{-1}\sG}\sim_\mathbb{Z}\beta^*K_\sG$. Moreover, $\beta^{-1}\sG$ is obviously closed under $p$-th powers for almost all primes $p$.
Finally, $\nu(Z)=1$ since $K_Z\sim_\mathbb{Q} \beta^*K_X$ and $\nu(X)=1$ by assumption (see \cite[Proposition V 2.7]{nakayama04}). The statement now follows from Theorem \ref{thm:algebraic_integrability_nu_un} applied to $\beta^{-1}\sG$.
\end{proof}

\begin{proof}[Proof of Theorem \ref{thm:algebraic_integrability_nu_un}]
For the reader's convenience, the proof is subdivided into a number of relatively independent
steps. Set $n=\dim X$.

\medskip
\noindent\textit{Step 1.} Let $\beta \colon Z \to X$ be a $\mathbb{Q}$-factorialization of $X$. By Lemma 
\ref{lemma:singularities_birational_morphism}, $\beta^{-1}\sG$ is canonical with $K_{\beta^{-1}\sG}\sim_\mathbb{Q}\beta^*K_\sG$. Moreover, $\beta^{-1}\sG$ is obviously closed under $p$-th powers for almost all primes $p$. 
Finally, $\nu(Z)=1$ since $K_Z\sim_\mathbb{Q} \beta^*K_X$ and $\nu(X)=1$ by assumption (see \cite[Proposition V 2.7]{nakayama04}). Thus, replacing $X$ by $Z$, if necessary, we may assume that
$X$ is $\mathbb{Q}$-factorial.
 
\medskip
\noindent\textit{Step 2.} By \cite[Th\'eor\`eme 3.3]{deczar}, we may run a minimal model program for $X$ and end with a minimal model in codimension one.
Therefore, there exists a sequence of maps 

\begin{center}
\begin{tikzcd}[row sep=large, column sep=large]
X:=X_0 \ar[r, "{\phi_0}", dashrightarrow] & X_1 \ar[r, "{\phi_1}", dashrightarrow] & \cdots \ar[r, "{\phi_{i-1}}", dashrightarrow] & X_i \ar[r, "{\phi_i}", dashrightarrow] & X_{i+1} \ar[r, "{\phi_{i+1}}", dashrightarrow] & \cdots \ar[r, "{\phi_{m-1}}", dashrightarrow] & X_m
\end{tikzcd}
\end{center}

\noindent where the $\phi_i$ are either divisorial contractions or flips, and $K_{X_m}$ is movable. The spaces $X_i$ are normal, $\mathbb{Q}$-factorial, $X_i$ has terminal singularities for all $0\le i \le m$, and $\nu(X_i)=1$. 
Let $\sG_i$ be the foliation on $X_i$ induced by $\sG$. By \cite[Lemma 3.2.5]{kmm}, we have $K_{\sG_i}\equiv 0$. 
Moreover, by Lemma \ref{lemma:singularities_birational_morphism}, $\sG_i$ has canonical singularities. 
Replacing $X$ by $X_m$, we may therefore assume without loss of generality that 
$X$ is $\mathbb{Q}$-factorial and $K_X$ is movable.

\medskip
\noindent\textit{Step 3.} Let $S \subseteq X$
be a two dimensional complete intersection of general elements of a very ample linear system $|H|$ on $X$. We may assume wihtout loss of generality that $S$ is smooth and contained in 
$X_{\textup{reg}}$. Let $\sL$ be the foliation by curves on $S$ induced by $\sG$. By Proposition \ref{prop:bertini}, we have 
$\det\sN_\sL \cong (\det\sN_\sG)_{|S}$. In particular, we have 
\begin{equation}\label{BB_formula1}
c_1(\sN_\sL)^2 = K_X ^2 \cdot H^{n -2} \ge 0 
\end{equation}
since $K_X$ is movable by Step 2.

On the other hand, by the Baum-Bott formula (see \cite[Chapter 3, Theorem 1]{brunella}), we have 
\begin{equation}\label{BB_formula2}
c_1(\sN_\sL)^2 =\sum_x \textup{BB}(\sL,x),
\end{equation}
where $x$ runs through all singular points of $\sL$, and $\textup{BB}(\sL,x)$ denotes the Baum-Bott index of $\sL$ at $x$. 
Let $x$ be a singular point of $\sL$. If $\sG$ is regular at $x$, then there is a holomorphic function $f$ defined in a neighborhood of $x$ such that $df$ vanishes at finitely many points and such that $df$ defines $\sL$ (see Proposition \ref{prop:bertini}). It follows that 
$\textup{BB}(\sL,x)=0$. Suppose now that $x$ is a singular point of $\sG$. 
By \cite[Corollary 7.8]{lpt}, 
$\sG$ is defined at $x$ by the local $1$-form $\omega=pz_2dz_1 + qz_1dz_2$, where $p$ and $q$ are positive integers and $(z_1,\ldots,z_n)$ are analytic coordinates on $X$ centered at $x$.
This implies in particular that the singular locus of $\sG$ is smooth in a neighborhood of $x$. By general choice of $S$,
we may assume without loss of generality that $S$ intersects the singular locus of $\sG$ transversely.
It follows that $\sL$ is defined at $x$ by the local $1$-form $\omega=pvdu + qudv$, where $(u,v)$ are analytic coordinates on $S$ centered at $x$, and hence $\textup{BB}(\sL,x)=-\frac{(p-q)^2}{pq}\le 0$.
In either case, we have $\textup{BB}(\sL,x) \le 0$. Equations \eqref{BB_formula1} and
\eqref{BB_formula2} above then show that $\textup{BB}(\sL,x)=0$ for any singular point $x$ of $\sL$. It follows that there exists an open set $X^\circ \subseteq X$ with complement of codimension at least three such that 
$\sG_{|X^\circ}$ is defined by closed holomorphic $1$-forms with zero set of codimension at least two locally for the analytic topology.

\medskip
\noindent\textit{Step 4.} Suppose first that $q(X)=0$. Then the Picard group of $X$ is discrete, and thus $K_\sG$ is torsion.
Replacing $X$ by the associated cyclic cover, which is quasi-\'etale (see \cite[Definition 2.52]{kollar_mori}), and using Lemma \ref{lemma:canonical_quasi_etale_cover} and Remark \ref{rem:quasi_etale_smooth_locus}, we may assume without loss of generality that $K_\sG\sim_\mathbb{Z} 0$. By Proposition \ref{prop:criterion_regularity_2}, we see that $\sG$ is weakly regular. Theorem \ref{thm:algebraic_integrability_nu_un_regular} above then says that $\sG$ is algebraically integrable.

\medskip
\noindent\textit{Step 5.} To prove the statement, we argue by induction on $\dim X$. 

If $\dim X=1$, then $\sG$ is obviously algebraically integrable.

Suppose from now on that $\dim X \ge 2$. By Step 4, we may assume wihtout loss of generality that $q(X)\neq 0$. Let 
$$a_X\colon X \to \textup{A}$$
be the Albanese morphism, that is, the universal morphism to an abelian variety (see \cite{serre}). Since $X$ is a projective variety with rational singularities, $\Pic^\circ(X)$ is an abelian variety and $\textup{A}\cong\big(\Pic^\circ(X)\big)^\vee$. 
Moreover, the Albanese morphism is induced by the universal line bundle (see \cite[Lemma 8.1]{kawamata85}).
In particular, $\dim \textup{A} = q(X)>0$ by assumption. 
Let $F$ be a general fiber of the Stein factorization of $X \to a_X(X)$, and let 
$\sH$ be the foliation on $F$ induced by $\sG$.

Suppose first that $\dim F>0$.
Then $F$ has terminal singularities, and that $K_F \sim_\mathbb{Z} {K_X}_{|F}$ by the adjunction formula. In particular, $K_F$ is nef and $\nu(F)\in\{0,1\}$. 
If $\sH=T_F$, then $\sH$ is obviously algebraically integrable.
Suppose that $\sH$ has codimension one. By Proposition \ref{prop:bertini}, we have $K_\sH\sim_\mathbb{Z} {K_\sG}_{|F} - B$ for some effective Weil divisor $B$ on $F$. Suppose that $B\neq 0$.
Applying \cite[Theorem 4.7]{campana_paun15} to the pull-back of $\sH$ on a resolution of $F$, we see that $\sH$ is uniruled. This implies that $\sG$ is uniruled as well since $F$ is general. But this contradicts Proposition \ref{proposition:canonical_versus_uniruled}, and shows that $B=0$. By Proposition \ref{proposition:canonical_versus_uniruled} applied to $\sH$, we see that $\sH$ is canonical. Notice that
$\sH$ is closed under $p$-th powers for almost all primes $p$. 
If $\nu(F)=0$, then $\sH$ is algebraically integrable by Lemma \ref{lemma:K_torsion} and Proposition \ref{prop:ESBT_av} below.
If $\nu(F)=1$, then $\sH$ is algebraically integrable by induction.
So, if $\dim F \ge 2$, then the statement follows from Proposition \ref{prop:splitting_algebraic_transcendental} using the induction hypothesis again.

Suppose from now on that $\dim F \le 1$ and that $\sH$ is the foliation by points. Then the tangent map to $a_X$ yields an 
inclusion map $\sG \subseteq a_X^*T_{\textup{A}}$.
On the other hand, by Lemma \ref{lemma:stability_versus_uniruled}, $\sG$ is semistable with respect to any polarization on $X$.
This immediately implies that $\sG\cong \sO_X^{n-1}$ since $\sG$ is reflexive and $K_\sG\equiv 0$.
Moreover, $\sG$ is a sheaf of abelian Lie algebras. Let $\textup{Aut}^\circ(X)$ denote the neutral component of the automorphism group $\textup{Aut}(X)$ of $X$, and let $H \subset \textup{Aut}^\circ(X)$ be the 
connected commutative Lie subgroup with Lie algebra $H^0(X,\sG) \subset H^0(X,T_X)$. Finally, let $G \subseteq \textup{Aut}^\circ(X)$ be its Zariski closure. Note that $\sG$ is induced by $H$ by construction. Moreover, $G$ is a commutative algebraic group. If $H=G$, then $\sG$ is algebraically integrable. Suppose from now on that $\dim G \ge \dim X$.
Arguing as in the proof of \cite[Lemma 9.3]{lpt}, one shows that $\dim G = \dim X$ and that 
$X$ is an equivariant compactification of $G$. 
Since $K_X$ is pseudo-effective and $X$ has terminal singularities, $X$ is not uniruled. It follows that $G=X$ 
is an abelian variety by a theorem of Chevalley. But this contradicts the assumption that $\nu(X)=1$, completing the proof of  the theorem.
\end{proof}

\begin{prop}\label{prop:ESBT_av} Let $A$ be a complex abelian variety, and let $\sG$ be a linear foliation on $A$. Suppose that $\sG$ is closed under $p$-th powers for almost all primes $p$. Then $\sG$ is algebraically integrable.
\end{prop}

\begin{proof}
By \cite[Proposition 3.6]{esbt}, we may assume without loss of generality that $X$ and $\sG$ are defined over a number field.
The statement then follows from \cite[Theorem 2.3]{bost}.
\end{proof}

\section{Foliations defined by closed rational 1-forms}\label{section:closed}

In this section, we address codimension one foliations with numerically trivial canonical class defined by closed rational $1$-forms with values in flat line bundles whose zero sets have codimension at least two. Recall that a \textit{flat vector bundle} on a normal complex variety $X$ is a vector bundle of rank $r$ induced by a representation $\pi_1(X) \to \textup{GL}(r,\mathbb{C})$. 
The following is the main result of the present section.

\begin{thm}\label{thm:closed_rational_1_form_2}
Let $X$ be a normal complex projective variety with canonical singularities, and let $\sG$ be a codimension one foliation on 
$X$.
Suppose that $\sG$ is canonical with $K_\sG$ Cartier and $K_\sG\equiv 0$.
Suppose in addition that $\sG$ is given by a closed rational $1$-form 
$\omega$ with values in a flat line bundle $\sL$, whose zero set has codimension at least two. 
Then there exists a normal projective equivariant compactification $Z$ of a commutative algebraic group $G$ of dimension at least $2$ as well as a codimension one foliation $\sH\cong \sO_Z^{\,\dim Z-1}$ on $Z$ induced by a codimension one Lie subgroup of $G$, a normal projective variety $Y$ with $K_Y\sim_\mathbb{Z}0$, and a quasi-\'etale cover $f \colon Y \times Z \to X$ such that $f^{-1}\sG$ is the pull-back of $\sH$ via the projection $Y \times Z \to Z$. Moreover, there is no positive-dimensional algebraic subvariety  tangent to $\sH$ passing through a general point of $Z$.
\end{thm}

\begin{exmp}\label{example:suspension_2} Let $n \ge 2$ be an integer.
Let $A=\mathbb{C}^{n-1}/\Lambda$ be a complex abelian variety, and let 
$\rho\colon \pi_1(A) \to \textup{PGL}(2,\mathbb{C})$ be a representation of the fundamental group $\pi_1(A)\cong\Lambda$ of $A$.
Then the group $\pi_1(A)$ acts diagonally on $\mathbb{C}^{n-1}\times \mathbb{P}^1$ by 
$\gamma\cdot(z,p)=\big(\gamma(z),\rho(\gamma)(p)\big)$. Set $X:=(\mathbb{C}^{n-1} \times \mathbb{P}^1)/\pi_1(A)$, 
and denote by $\psi\colon X \to A \cong \mathbb{C}^{n-1}/\pi_1(A)$ the projection morphism, which is $\mathbb{P}^1$-bundle.
The foliation on $\mathbb{C}^{n-1} \times \mathbb{P}^1$ induced by the projection $\mathbb{C}^{n-1} \times \mathbb{P}^1 \to \mathbb{P}^1$ is invariant under the action of $\pi_1(A)$ and yields a flat Ehresmann connection $\sG$ on $\psi$.
Let $H \subseteq \textup{PGL}(2,\mathbb{C})$ be the Zariski closure of the image of $\rho$.
Suppose that $\dim H >0$.

We use notation of Example \ref{example:suspension}.

\medskip

Suppose that $H=p(T)$. Let $z$ be a coordinate on $\mathbb{C} \subset \mathbb{P}^1$ such that the inverse images of
$D_1$ and $D_2$ on the universal covering space $\mathbb{C}^{n-1} \times \mathbb{P}^1$ of $X$ are given by equations 
$z=0$ and $z=\infty$ respectively. Then $\frac{dz}{z}$ induces a closed logarithmic $1$-form with poles along $D_1$ and $D_2$ and empty zero set defining $\sG$. The $1$-form $\frac{dz}{z^2}$ induces a closed rational $1$-form on $X$ with values in the flat line bundle $\sL=\psi^*\sM^{\otimes -2}$ whose divisor of zeroes and poles is $-2D_1$. 

\medskip

Suppose that $H=p(U)$. Let $z$ be a coordinate on $\mathbb{C} \subset \mathbb{P}^1$ such that the inverse image of
$D$ on the universal covering space $\mathbb{C}^{n-1} \times \mathbb{P}^1$ of $X$ is given by equation
$z=0$. Then the $1$-form $\frac{dz}{z^2}$ induces a closed rational $1$-form on $X$ defining $\sG$ 
with divisor of zeroes and poles poles $-2D$. 
\end{exmp}

The proof of Theorem \ref{thm:closed_rational_1_form_2} makes use of the following result, which might be of independent interest.

\begin{thm}\label{thm:closed_rational_1_form}
Let $X$ be a normal complex projective variety with klt singularities. Let $\omega$ be a closed rational $1$-form, 
and let $\sG$ be the foliation on $X$ defined by $\omega$.
Suppose that $\sG$ has canonical singularities and $K_\sG\equiv 0$. 
Then one of the following holds.
\begin{enumerate}
\item There exist a complete smooth curve $C$, a complex projective variety $Y$ with canonical singularities and 
$K_Y \sim_\mathbb{Y}0$, as well as a quasi-\'etale cover $f \colon Y \times C \to X$ such that $f^{-1}\sG$ is induced by the projection $Y \times C \to C$.
\item Then there exists a normal projective equivariant compactification $Z$ of a commutative algebraic group $G$ of dimension at least $2$ as well as a codimension one foliation $\sH\cong \sO_Z^{\,\dim Z-1}$ on $Z$ induced by a codimension one Lie subgroup of $G$, a normal projective variety $Y$ with $K_Y\sim_\mathbb{Z}0$, and a quasi-\'etale cover $f \colon Y \times Z \to X$ such that $f^{-1}\sG$ is the pull-back of $\sH$ via the projection $Y \times Z \to Z$. Moreover, there is no positive-dimensional algebraic subvariety  tangent to $\sH$ passing through a general point of $Z$.
\end{enumerate}
\end{thm}

To prove Theorem \ref{thm:closed_rational_1_form} above, we will need the following auxiliary result.

\begin{lemma}\label{lemma:structure_closed_rational}
Let $X$ be a normal complex projective variety.
Let $\omega$ be a closed rational $1$-form, 
and let $\sG$ be the foliation on $X$ defined by $\omega$. Then there is a finite dimensional complex abelian Lie algebra $V$ and a morphism of sheaves of Lie algebras $\sG \to V\otimes\sO_X$ whose kernel $\sE$
is an algebraically integrable foliation on $X$.
\end{lemma}

\begin{proof}
Let $\beta\colon Z \to X$ be a resolution of singularities, and let $D$ be the divisor on $Z$ of codimension one poles of 
$\omega$. We may assume without loss of generality that $D$ has simple normal crossings.
Set $Z^\circ:= Z\setminus \textup{Supp}\,D$, and let $z_0 \in Z^\circ$.
Consider the multivalued holomorphic function
$h(z):=\int_\gamma \omega$ on $Z^\circ$, where $\gamma$ is a path connecting $z_0$ and $z$ contained
in $Z^\circ$. Then $h$ yields local first integrals for $\beta^{-1}\sG$. 
Let 
$$\rho\colon \pi_1(Z^\circ) \to \mathbb{G}_a$$
be the representation of $\pi_1(Z^\circ)$ induced by analytic continuation of the multivalued holomorphic function
$h$ along loops with base point $z_0$. 
Now, let
$$\textup{qa}_{Z^\circ} \colon Z^\circ \to \textup{G}\big(Z^\circ\big):= H^0\big(Z,\Omega^{1}_{Z}(\textup{log }D)\big)^*/H_1(Z^\circ,\mathbb{Z})$$ be the quasi-Albanese morphism, that is the universal morphism to a semi-abelian variety (see \cite{iitaka}). 
Let also $F^\circ$ be a general fiber of $\textup{qa}_{Z^\circ}$. The kernel of the natural morphism 
$\pi_1\big(Z^\circ\big) \to \pi_1\big(\textup{G}(Z^\circ)\big)$ is generated by $\bigl\lbrack\pi_1\big(Z^\circ\big),\pi_1\big(Z^\circ\big)\bigr\rbrack$ together with finitely many torsion elements. It follows that 
the representation $\pi_1(F^\circ) \to \mathbb{G}_a$ given by the restriction of $\rho$ to $\pi_1(F^\circ)$ is trivial. This in turn implies that the foliation induced by $\sG$ on $F^\circ$ has algebraic leaves by
\cite[Th\'eor\`eme III. 2.1]{cerveau_mattei}.

Let $\sE_Z\subseteq \beta^{-1}\sG$ be the foliation on $Z$ such that ${\sE_Z}_{|Z^\circ}$ is the kernel of 
the morphism 
$${\beta^{-1}\sG}_{|Z^\circ} \to (\textup{qa}_{Z^\circ})^*T_{\textup{G}(Z^\circ)}\cong H^0\big(Z,\Omega^1_{Z}(\textup{log }D)\big)^*\otimes \sO_{Z^\circ},$$ and let $\sE$ be the induced foliation on $X$. 
We have shown that $\sE$ is algebraically integrable.
On the other hand, any irreducible component of $D$ is invariant under $\beta^{-1}\sG$
by \cite[Proposition III. 1.1]{cerveau_mattei}. This implies that the contraction map 
$$\beta^{-1}\sG \to H^0\big(Z,\Omega^{1}_{Z}(\textup{log }D)\big)^*\otimes\sO_Z$$
is well-defined. Therefore, we must have exact sequences
$$0 \to \sE_Z \to \beta^{-1}\sG \to H^0\big(Z,\Omega^{1}_{Z}(\textup{log }D)\big)^*\otimes\sO_Z$$
and
$$0 \to \sE \to \sG \to H^0\big(Z,\Omega^{1}_{Z}(\textup{log }D)\big)^*\otimes\sO_X.$$
This finishes the proof of the lemma.
\end{proof}

\begin{proof}[Proof of Theorem \ref{thm:closed_rational_1_form}]
We maintain notation and assumptions of Theorem \ref{thm:closed_rational_1_form}.
By Proposition \ref{prop:splitting_algebraic_transcendental}, there exist normal projective varieties $Y$ and $Z$, a foliation $\sH$ on $Z$ such that there is no positive-dimensional algebraic subvariety  tangent to $\sH$
passing through a general point of $Z$, and a quasi-\'etale cover $f \colon Y \times Z \to X$ such that $f^{-1}\sG$ is the pull-back of $\sH$ via the projection $Y \times Z \to Z$. Moreover, $\sH$ has canonical singularities and $K_\sH\equiv 0$.

If $\dim Z=1$, then we are in case (1) of Theorem \ref{thm:closed_rational_1_form}.

Suppose from now on that $\dim Z \ge 2$. Let $F \cong Z $ be a general fiber of the projection $Y \times Z \to Y$. Then $(f^{-1}\sG)_{|F} \cap T_F \cong \sH$, and hence the restriction of $df (\omega)$ to $F$ is a closed rational $1$-form defining $\sH$. 
Applying Lemma \ref{lemma:structure_closed_rational} above to $\sH$, we see that
there is a finite dimensional complex abelian Lie algebra $V$ and an injective morphism of sheaves of Lie algebras 
$\sH \to V\otimes\sO_Z$.
This immediately implies that $\sG\cong \sO_Z^{\, \dim Z -1}$ since $\sH$ is reflexive and $K_\sH \equiv 0$. Moreover, $\sH$ is a sheaf of abelian Lie algebras. Let $\textup{Aut}^\circ(Z)$ denote the neutral component of the automorphism group $\textup{Aut}(Z)$ of $Z$, and let $H \subset \textup{Aut}^\circ(Z)$ be the 
connected commutative Lie subgroup with Lie algebra $H^0(Z,\sH) \subset H^0(Z,T_Z)$. Finally, let $G \subseteq \textup{Aut}^\circ(Z)$ be its Zariski closure. Note that $\sH$ is induced by $H$ by construction, and that $\dim G \ge \dim Z$ since $\sH$ is not algebraically integrable. Moreover, $G$ is a commutative algebraic group.
Arguing as in the proof of \cite[Lemma 9.3]{lpt}, one shows that $\dim G = \dim Z$ and that 
$Z$ is an equivariant compactification of $G$. Thus, we are in case (2) of Theorem \ref{thm:closed_rational_1_form}.
\end{proof}

In the setting of Theorem \ref{thm:closed_rational_1_form_2}, the twisted rational $1$-form $\omega$ is not 
determined by $\sG$ (see Example \ref{example:suspension_2}). The following result addresses this issue.

\begin{lemma}\label{lemma:unicity}
Let $X$ be a normal complex projective variety, and let $\alpha$ and $\beta$ be closed rational $1$-forms with values in flat line bundles $\sL$ and $\sM$. Suppose that the divisors of zeroes and poles of $\alpha$ and $\beta$ coincide. Suppose in addition that $\alpha\wedge\beta=0$. Then either $\sL\cong\sM$ as flat line bundles and
$\alpha=c\beta$ for some $c \in \mathbb{C}^*$, or
there is a non-zero $\gamma \in H^0(X,\Omega_X^{[1]})$ such that $\alpha \wedge \gamma=0$.  
\end{lemma}

\begin{proof}Let $D$ be the divisor of zeroes and poles of $\alpha$, and set $X^\circ:=X_{\textup{reg}}$.
There exist a covering $(U_i)_{i\in I}$ of $X^\circ$ by analytically open sets  
and closed meromorphic $1$-forms $\alpha_i$ (resp. $\beta_i$) on $U_i$ with divisor of zeroes and poles 
$D_{|U_i}$ satisfying $\alpha_i=a_{ij}\alpha_j$ (resp. $\beta_{i}=b_{ij}\beta_{j}$) over $U_i\cap U_j$ with
$a_{ij}\in\mathbb{C}^*$ (resp. $b_{ij}\in\mathbb{C}^*$) and such that $(\alpha_i)_{i\in I}$ (resp. $(\beta_i)_{i\in I}$) represents $\alpha_{|X^\circ}$ (resp. $\beta_{|X^\circ}$).

Since $\alpha_i\wedge \beta_i=0$ by assumption, there exists a meromorphic function $f_i$ on $U_i$ such that 
$\alpha_i = f_i \beta_i$. Note that we must have $df_i\wedge \beta_i=0$ since $d\alpha_i=d\beta_i=0$.
Now, observe that $f_i$ is a nowhere vanishing holomorphic function since the 
the divisors of zeroes and poles of $\alpha_i$ and $\beta_i$ coincide by assumption. On the other hand, we obviously have
$f_i=\frac{a_{ij}}{b_{ij}}f_j$ over $U_i\cap U_j$. If the functions $f_i$ are constant, then $\sL\cong \sM$ as flat line bundles, and we can suppose without loss of generality that $a_{ij}=b_{ij}$ for all indices $i$ and $j$.
Then $\alpha=c\beta$ with $c=f_i=f_j$.
Suppose that some $f_i$ is not a constant function.
Since $\frac{df_i}{f_i}=\frac{df_j}{f_j}$, there exists a non-zero holomorphic $1$-form $\gamma^\circ$ on $X^\circ$ that restricts to $\frac{df_i}{f_i}$ over $U_i$.
By construction we have $\alpha_{|X^\circ} \wedge \gamma^\circ=0$. The claim now follows easily from the GAGA theorem. 
\end{proof}

To prove Theorem \ref{thm:closed_rational_1_form_2} in the case where $X$ is uniruled, we
will reduce to foliations defined by closed rational $1$-form using Proposition \ref{prop:flatness_torsion_2} below. We first consider the special case where $X$ has terminal singularities.

\begin{prop}\label{prop:flatness_torsion}
Let $X$ be a normal complex projective variety with terminal singularities, and let $\sG\subset T_X$ be a codimension one foliation with canonical singularities. Suppose that $\sG$ is given by a closed rational $1$-form 
$\omega$ with values in a flat line bundle $\sL$ whose zero set has codimension at least two. 
Suppose furthermore that $K_X$ is not pseudo-effective, and that $K_\sG\equiv 0$. 
Then we can choose $\omega$ with $\sL$ a torsion flat line bundle.
\end{prop}

\begin{proof} For the reader's convenience, the proof is subdivided into a number of steps.

\medskip
\noindent\textit{Step 1.} 
By Proposition \ref{prop:splitting_algebraic_transcendental}, there exist normal projective varieties $Y$ and $Z$, a foliation $\sH$ on $Y$ such that there is no positive-dimensional algebraic subvariety  tangent to $\sH$
passing through a general point of $Y$, and a quasi-\'etale cover $f \colon Y \times Z \to X$ such that $f^{-1}\sG$ is the pull-back of $\sH$ via the projection $Y \times Z \to Y$. Moreover, $\sH$ has canonical singularities and $K_\sH\equiv 0$.
Let $F \cong Y $ be a general fiber of the projection $Y \times Z \to Z$. Then $(f^{-1}\sG)_{|F} \cap T_F \cong \sH$, and hence the restriction of $df (\omega)$ to $F$ is a closed rational $1$-form with values in $\sL_{|F}$
defining $\sH$ whose zero set has codimension at least two. Moreover, its pull-back to $Y \times Z$ is a closed rational $1$-form with values in the pull-back of $\sL_{|F}$ whose zero set has codimension at least two defining $\sG$.
Finally, $K_Y$ is obviously not pseudo-effective, and one readily checks $Y$ has terminal singularities using Fact \ref{fact:quasi_etale_cover_and_singularities}.
Therefore, replacing $\sG$ by $\sH$, if necessary, we may assume that there is no positive-dimensional algebraic subvariety  tangent to $\sG$ passing through a general point of $X$.

\medskip
\noindent\textit{Step 2.} Let $\beta \colon Z \to X$ be a $\mathbb{Q}$-factorialization of $X$. By Lemma 
\ref{lemma:singularities_birational_morphism}, $\beta^{-1}\sG$ is canonical with $K_{\beta^{-1}\sG}\sim_\mathbb{Q}\beta^*K_\sG$. Moreover, $\beta^{-1}\sG$ is obviously given by a closed rational $1$-form
with values in the flat line bundle 
$\beta^*\sL$ whose zero set has codimension at least two.
Finally, $K_Z$ is not pseudo-effective since $\beta_*K_Z\sim_\mathbb{Z} K_X$ and $K_X$ is not pseudo-effective by assumption.
Thus, replacing $X$ by $Z$, if necessary, we may assume without loss of generality that $X$ has $\mathbb{Q}$-factorial terminal singularities.

\medskip
\noindent\textit{Step 3.} Since $K_X$ is not pseudo-effective by assumption, we may run a minimal model program for $X$ and end with a Mori fiber space (see \cite[Corollary 1.3.3]{bchm}). Therefore, there exists a sequence of maps

\begin{center}
\begin{tikzcd}[row sep=large, column sep=large]
X:=X_0 \ar[r, "{\phi_0}", dashrightarrow] & X_1 \ar[r, "{\phi_1}", dashrightarrow] & \cdots \ar[r, "{\phi_{i-1}}", dashrightarrow] & X_i \ar[r, "{\phi_i}", dashrightarrow] & X_{i+1} \ar[r, "{\phi_{i+1}}", dashrightarrow] & \cdots \ar[r, "{\phi_{m-1}}", dashrightarrow] & X_m \ar[d, "{\psi_m}"]\\
&&&&&& Y
\end{tikzcd}
\end{center}

\noindent where the $\phi_i$ are either divisorial contractions or flips, and $\psi_m$ is a Mori fiber space. The spaces $X_i$ are normal, $\mathbb{Q}$-factorial, and $X_i$ has terminal singularities for all $0\le i \le m$. 
Let $\sG_i$ be the foliation on $X_i$ induced by $\sG$. By \cite[Lemma 3.2.5]{kmm}, we have $K_{\sG_i}\equiv 0$. Moreover, by Lemma \ref{lemma:singularities_birational_morphism}, $\sG_i$ has canonical singularities. 
Finally, one readily checks that $\sG_m$ is given by a closed rational $1$-form $\omega_m$
with values in a flat line bundle 
$\sL_m$ whose zero set has codimension at least two, using \cite[Lemma 3.2.5]{kmm} again.
Thus, replacing $X$ by $X_m$, if necessary, we may assume without loss of generality that 
there is a Mori fiber space $\psi \colon X \to Y$.

\medskip
\noindent\textit{Step 4.} By assumption, there exist prime divisors $(D_i)_{1\le i \le r}$ on $X$ and positive integers $(m_i)_{1\le i \le r}$ such that $\sN_\sG \cong \sO_X\big(\sum_{1\le i \le r} m_i D_i\big)\otimes\sL$.
Let $I \subseteq \{1,\ldots,r\}$ be the set of indices $i\in \{1,\ldots,r\}$ such that $\psi(D_i)=Y$. Note that $I\neq \emptyset$ since 
$\sum_{1\le i \le r}m_i D_i \equiv -K_X$ is relatively ample.

Suppose that there exists $i \in I$ such that the residue of $\omega$ at a general point of $D_i$ is non-zero. Then arguing as in \cite[Section 8.2.1]{lpt}, one shows that 
$\sL$ is torsion.

\medskip

Suppose from now on that the residue of $\omega$ at a general point of $D_i$ is zero for
any $i\in I$.

\medskip
\noindent\textit{Step 5.} Let $F$ be a general fiber of $\psi$. Note that $F$ has terminal singularities, and that $K_F \sim_\mathbb{Z} {K_X}_{|F}$ by the adjunction formula. Moreover, $F$ is a Fano variety by construction. 
Let $\sH$ be the foliation on $F$ induced by $\sG$. 
Note that $\sH$ has codimension one by Step 1.
By Proposition \ref{prop:bertini}, we have $K_\sH\sim_\mathbb{Z} {K_\sG}_{|F} - B$ for some effective Weil divisor $B$ on $F$. Suppose that $B\neq 0$.
Applying \cite[Theorem 4.7]{campana_paun15} to the pull-back of $\sH$ on a resolution of $F$, we see that $\sH$ is uniruled. This implies that $\sG$ is uniruled as well since $F$ is general. But this contradicts Proposition \ref{proposition:canonical_versus_uniruled}, and shows that $B=0$. By Proposition \ref{proposition:canonical_versus_uniruled} applied to $\sH$, we see that $\sH$ is canonical.
Since $\pi_1(F)$ is finite (see \cite{hacon_mckernan}), there is an \'etale cover $f \colon F_1 \to F$ such that 
$f^{-1}\sH$ is given by a closed rational $1$-form (possibly with codimension one zeroes) with zero residues at general points of its codimension one poles. 
Lemma \ref{lemma:no_residue} then implies that $F\cong\mathbb{P}^1$.
Now, we have $\sum_{1\le i \le r}m_iD_i\cdot F=2$. Suppose that $r=2$ and $D_1\cdot F=D_2\cdot F=1$ or that $r=1$ and $D_1\cdot F=2$. Let $y\in Y_{\textup{reg}}$ be a general point, and let $U\subseteq Y_{\textup{reg}}$ be an analytically open neighborhood of $y$ such that $\psi^{-1}(U)\cong U \times \mathbb{P}^1$. We may assume that there exists a coordinate $z$ on $\mathbb{C} \subset \mathbb{P}^1$ such that the poles of $\omega_{|\psi^{-1}(U)}$ are given by equations $z=0$ and $z=\infty$.
Then $$\omega_{|U \times \mathbb{C}}=a\frac{dz}{z}+\frac{1}{z}(\alpha+z\beta+z^2\gamma)$$ 
where $a$ is a holomophic function on $U$, and $\alpha$, $\beta$ and $\gamma$ are holomorphic $1$-forms on $U$. Observe that $a(y)\neq 0$ since $\sG$ is generically transverse to $F$. This implies that
$\omega$ has non-zero residue along $D_i$ for any $i\in I$, yielding a contradiction. 
Therefore, we must have $r=1$, $m_1=2$, and $D_1\cdot F=1$.

\medskip
\noindent\textit{Step 6.} Recall that $X$ is smooth in codimension two since it has terminal singularities.
Since $\dim X - \dim Y = 1$, there exists an open subset $Y^\circ \subseteq Y_{\textup{reg}}$ 
with complement of codimension at least two such that 
$X^\circ:=\psi^{-1}(Y^\circ) \subseteq X_{\textup{reg}}$. From \cite[Theorem 4.1]{andreatta_wisniewski_view}, we conclude that
$\psi^\circ:=\psi_{|X^\circ}$ is a conic bundle. It follows that $\psi^\circ$ is smooth since $\psi$ is an elementary Mori contraction and $D_1\cdot F=1$.

Next, we show that $\sG_{|X^\circ}$ yields a flat Ehresmann connection on $\psi^\circ$.
Let $C\cong \mathbb{P}^1$ be any fiber of $\psi^\circ$. We have $c_1(\sN_\sG)\cdot C =2$ and either $C$ is tangent to $\sG$, or $C$ is transverse to $\sG$ along $C$. Thus, we have to show that $C$ is not tangent to $\sG$.

Set $y:=\psi(C) \in Y^\circ$. By assumption, there is an analytically open neighborhood $U$ of $y$ in $Y^\circ$ such that $\sG$ is defined over $\psi^{-1}(U)$ by a closed rational $1$-form $\omega_U$ with poles along $D_1\cap \psi^{-1}(U)$. Shrinking $U$, if necessary, we may assume that $\psi^{-1}(U)\cong U \times \mathbb{P}^1$ and that there exists a coordinate $z$ on $\mathbb{C} \subset \mathbb{P}^1$ such that the pole of $\omega_U$ is given by equation $z=\infty$. Then 
$${\omega_U}_{|U \times \mathbb{C}}=adz+\alpha+z\beta+z^2\gamma$$ 
where $a$ is a holomophic function on $U$, and $\alpha$, $\beta$ and $\gamma$ are holomorphic $1$-forms on $U$. Since 
$d\omega_U=0$ by assumption, we must have $d\alpha=0$, $\beta=da$, and $\gamma=0$. Shrinking $U$ further, we may assume that 
$\alpha = db$ for some holomorphic function $b$ on $U$, so that 
${\omega_U}_{|U \times \mathbb{A}^1}=d(az+b)$. Set $u:=\frac{1}{z}$. Then $\sG$ is given by $d(\frac{a}{u}+b)$ in a neighborhood of $z=\infty$. If $a(y)= 0$, then one readily checks that $\sG$ is not canonical in a neighborhood of
$C \cap \{u=0\}$ (see \cite[Observation I.2.6]{mcquillan08}), yielding a contradiction. This shows that 
$a(y) \neq 0$, and hence $C$ is not tangent to $\sG$. This proves that $\sG_{|X^\circ}$ defines a flat Ehresmann connection on $\psi^\circ$.

\medskip
\noindent\textit{Step 7.}
By \cite[Corollary 4.5]{fujino99} and Lemma \cite[Lemma 5.1.5]{kmm}, $Y$ has $\mathbb{Q}$-factorial
klt singularities. Since $\textup{codim}\, Y \setminus Y^\circ \ge 2$ and $K_\sG\equiv 0$, we must have $K_Y\equiv 0$. 
Applying \cite[Corollary V 4.9]{nakayama04}, we conclude that $K_Y$ is torsion. 
Let $Y_1 \to Y$ be the index one canonical cover, which is quasi-\'etale (\cite[Definition 2.52]{kollar_mori}).
By construction, $K_{Y_1}\sim_\mathbb{Z}0$. In particular, $Y_1$ has canonical singularities.
By Theorem \ref{thm:kawamata_abelian_factor} applied to $Y_1$, we see that there exists an abelian variety as well as
a projective variety $Z$ with $K_{Z}\sim_\mathbb{Z}0$ and 
$\wt q(Z) = 0$, and a quasi-\'etale cover 
$f\colon  A \times Z\to Y$. 
Recall that $f$ branches only on the singular set of $Y$, so that $f^{-1}(Y^\circ)$ is smooth. On the other hand, 
since $f^{-1}(Y^\circ)$ has complement of codimension at least two in $A\times Z_\textup{reg}$, we have 
$\pi_1\big(A\times Z_\textup{reg}\big) \cong \pi_1\big(f^{-1}(Y^\circ)\big)$.
Now, consider the representation
$$\rho\colon \pi_1\big(A\times Z_\textup{reg}\big) \cong \pi_1\big(f^{-1}(Y^\circ)\big) \to 
\pi_1\big(Y^\circ\big) \to \textup{PGL}(2,\mathbb{C})$$
induced by $\sG_{|X^\circ}$. By \cite[Theorem I]{GGK}, the induced representation 
$$ \pi_1\big(Z_\textup{reg}\big) \to \pi_1\big(A\big)\times\pi_1\big(Z_\textup{reg}\big)\cong
\pi_1\big(A\times Z_\textup{reg}\big) \to \textup{PGL}(2,\mathbb{C})$$
has finite image.
Thus, replacing $Z$ with a quasi-\'etale cover, if necessary, we may assume without loss of generality that 
$\rho$ factors through the projection $\pi_1\big(A\times Z_\textup{reg}\big) \to \pi_1(A)$. 
Let $P$ be the corresponding $\mathbb{P}^1$-bundle over $A$. The natural projection $\pi\colon P \to A$ comes with a flat Ehresmann connection
$\sG_P \subset T_P$. By the GAGA theorem, $P$ is a projective variety. By assumption, its pull-back to 
$A\times Z_\textup{reg}$ agrees with $f^{-1}(Y^\circ)\times_{Y^\circ} X^\circ$ over $f^{-1}(Y^\circ)$. Moreover, the pull-backs 
on $A\times Z_\textup{reg}$ of the foliations $\sG$ and $\sG_P$ agree as well, wherever this makes sense. 
Since there is no positive-dimensional algebraic subvariety tangent to $\sG$ passing through a general point of $X$ by assumption, we must have $\dim Z=0$.

Set $A^\circ:=f^{-1}(Y^\circ)$, $P^\circ:=\pi^{-1}(A^\circ)$, and $\sG_{P^\circ}:={\sG_P}_{|P^\circ}$. Let $g^\circ\colon P^\circ \to X^\circ$ denote the natural morphism, which is an \'etale cover. Set $D_P^\circ:=(g^\circ)^{-1}(D_1 \cap X^\circ)$. 
Then $\sG_{P^\circ}$ is given by the closed rational $1$-form $\omega_{P^\circ}:= dg^\circ(\omega_{|X^\circ})$ with values in the flat line bundle $\sL_{P^\circ}:=(g^\circ)^*(\sL_{|X^\circ})$. Moreover, the zero set of $\omega_{P^\circ}$ has codimension at least two. Since $P$ is smooth and $P^\circ$ has codimension at least two in $P^\circ$, $\sL_{P^\circ}$ is the restriction to $P^\circ$ of a flat line bundle $\sL_P$ on $P$, and $\omega_P^\circ$ extends to a closed rational $1$-form
$\omega_P$ with values in $\sL_P$ whose zero set has codimension at least two. Note that the divisor of zeroes and poles of $\omega_P$ is $-2D_P$, where $D_P$ denotes the Zariski closure of $D_P^\circ$. By \cite[Proposition III. 1.1]{cerveau_mattei},
$D_P$ is a leaf of $\sG_P$. It follows that $D_P$ is a section of $\pi$.

Let $H \subseteq \textup{PGL}(2,\mathbb{C})$ be the Zariski closure of the image of $\rho$.
We use the notation introduced in Examples \ref{example:suspension} and \ref{example:suspension_2}. Recall from Step 1 that there is no positive-dimensional algebraic subvariety  tangent to $\sG_P$ passing through a general point of $P$. Therefore, $H$ is conjugate to $p(T)$ or 
$p(U)$. Observe that $\sG_P$ is not given by a (closed) holomorphic form since it is transverse to $\pi$ by Step 6.

If $H=p(U)$, then $\sL_P\cong \sO_P$ as flat line bundles by Lemma \ref{lemma:unicity} together with Example \ref{example:suspension_2}. This in turn implies that $\sL$ is torsion since the image of $\pi_1(P^\circ)$ in $\pi_1(X^\circ)$ has finite index.

Suppose from now on that $H=p(T)$. Let 
$$a_Y\colon Y \to \textup{A}(Y)$$
be the Albanese morphism. Since $X$ is a projective variety with rational singularities, 
$\Pic^\circ(Y)$ is an abelian variety, and $\textup{A}(Y)\cong\big(\Pic^\circ(Y)\big)^\vee$. 
Moreover, the Albanese morphism is induced by the universal line bundle (see \cite[Lemma 8.1]{kawamata85}).
In particular, there is a flat line bundle $\sL_{\textup{A}(Y)}$ on $\textup{A}(Y)$  such that 
$\sL\cong \psi^*\big({a_Y}^*\sL_{\textup{A}(Y)}\big)$.
Set $P_{\textup{A}(Y)}:=\mathbb{P}_{\textup{A}(Y)}\big(\sO_{\textup{A}(Y)}\oplus \sL_{\textup{A}(Y)}\big)$.
Lemma \ref{lemma:unicity} together with Example \ref{example:suspension_2} then imply that $P\cong A \times_{\textup{A}(Y)} P_{\textup{A}(Y)}$
and that $\sG_P$ is the pull-back of a foliation on $P_{\textup{A}(Y)}$.
We conclude that $a_Y\colon Y \to \textup{A}(Y)$ is generically finite since there no positive-dimensional algebraic subvariety  tangent to $\sG_P$ passing through a general point of $P$.
Recall from \cite[Corollary V 4.9]{nakayama04} that $\kappa(Y)=0$. Applying \cite[Theorem 13]{kawamata_av} to the Stein factorization of $a_Y\colon Y \to \textup{A}(Y)$, we see that $a_Y$ is an isomorphism. In particular, we can choose $A=Y$ above. Then the restriction of the tangent map $T\psi \colon T_X \to \psi^*T_A$ to $\sG$ gives an isomorphism
$\sG \cong \psi^*T_A$, so that $\sG$ induces a flat connection on $\psi$.
A classical result of complex analysis then says that $\phi$ is a locally trivial analytic fibration for the analytic topology (see \cite{kaup}). This shows that $X \cong P$ and that $\sG$ corresponds to $\sG_P$.
In particular, $\sG$ is defined by a closed logarithmic $1$-form, completing the proof of the proposition.
\end{proof}

\begin{lemma}\label{lemma:no_residue}
Let $X$ be a normal complex projective variety with klt singularities. Let $\omega$ be a closed rational $1$-form, 
and let $\sG$ be the foliation on $X$ defined by $\omega$.
Suppose that $\sG$ has canonical singularities and $K_\sG\equiv 0$. 
Suppose in addition that the residues of $\omega$ at general points of its codimension one poles are zero.
If $X$ is Fano, then $X\cong \mathbb{P}^1$.
\end{lemma}

\begin{proof}
By Theorem \ref{thm:closed_rational_1_form}, we may assume that there exists a normal projective equivariant compactification $Z$ of a commutative linear algebraic group $G$ as well as a quasi-\'etale cover $f \colon Z \to X$ such that 
$f^{-1}\sG$ is induced by a codimension one Lie subgroup $H$ of $G$. Moreover,
there is no positive-dimensional algebraic subvariety tangent to $\sG$ passing through a general point of $X$. 

Recall that $G \cong (\mathbb{G}_m)^{r}\times (\mathbb{G}_a)^{s}$ for some non-negative integers $r$ and $s$. If $s \ge 2$, then $\dim \textup{Lie}\, H \cap \textup{Lie}\, (\mathbb{G}_a)^{s} \ge 1$, and hence $\sG$ is uniruled. But this contradicts Proposition \ref{proposition:canonical_versus_uniruled}, and shows that $s\le 1$. By \cite[Theorem 2]{arzhantsev_kotenkova}, $X$ is a toric variety. It follows that $X_\textup{reg}$ has finite fundamental group by \cite[Proposition 1.9]{oda}.

Since $\omega$ is closed and since the residues of $\omega$ at general points of its codimension one poles are zero, we conclude that there is a meromorphic function $f$ on $X_\textup{reg}$ such that $\omega=df$ using \cite[Th\'eor\`eme III. 2.1]{cerveau_mattei}. By the Levi extension theorem, $f$ extends to a meromorphic function on $X$, and thus, $f$ is a rational function. This in turn implies that $\sG$ is algebraically integrable. It follows that $\dim X =1$, since there is no positive-dimensional algebraic subvariety tangent to $\sG$ passing through a general point of $X$. This completes the proof of the lemma.
\end{proof}

\begin{prop}\label{prop:flatness_torsion_2}
Let $X$ be a normal complex projective variety with canonical singularities, and let $\sG$ be a codimension one foliation on 
$X$.
Suppose that $\sG$ is canonical with $K_\sG$ Cartier and $K_\sG\equiv 0$ and that $K_X$ is not pseudo-effective.
Suppose in addition that $\sG$ is given by a closed rational $1$-form 
$\omega$ with values in a flat line bundle $\sL$ whose zero set has codimension at least two. 
Then we can choose $\omega$ with $\sL$ a torsion flat line bundle.
\end{prop}

\begin{proof}By Proposition \ref{prop:splitting_algebraic_transcendental}, there exist normal projective varieties $Y$ and $Z$, a foliation $\sH$ on $Y$ such that there is no positive-dimensional algebraic subvariety  tangent to $\sH$
passing through a general point of $Y$, and a quasi-\'etale cover $f \colon Y \times Z \to X$ such that $f^{-1}\sG$ is the pull-back of $\sH$ via the projection $Y \times Z \to Y$. Moreover, $\sH$ has canonical singularities and $K_\sH\equiv 0$.
Let $F \cong Y $ be a general fiber of the projection $Y \times Z \to Z$. Then $(f^{-1}\sG)_{|F} \cap T_F \cong \sH$, and hence 
$K_\sH$ is Cartier and the restriction of $df (\omega)$ to $F$ is a closed rational $1$-form with values in $\sL_{|F}$
defining $\sH$ whose zero set has codimension at least two. Its pull-back to $Y \times Z$ is a closed rational $1$-form with values in the pull-back of $\sL_{|F}$  defining $\sG$ whose zero set has codimension at least two.
Finally, $K_Y$ is obviously not pseudo-effective, and one readily checks $Y$ has canonical singularities using Fact \ref{fact:quasi_etale_cover_and_singularities}.
Therefore, replacing $\sG$ by $\sH$, if necessary, we may assume that there is no positive-dimensional algebraic subvariety  tangent to $\sG$ passing through a general point of $X$.

Let $\beta \colon Z \to X$ be a $\mathbb{Q}$-factorial terminalization of $X$. By Proposition \ref{prop:terminalization_canonical_singularities}, $\beta^{-1}\sG$ is canonical with $K_{\beta^{-1}\sG}\sim_\mathbb{Z}\beta^*K_\sG$. Suppose that $\sG$ is closed under $p$-th powers for almost all primes $p$. Then $\sG$ is algebraically integrable by Corollary \ref{cor:grothendieck_katz}, yielding a contradiction. Thus, 
by Proposition \ref{prop:p_closed_or_not} and Lemma \ref{lemma:stability_versus_uniruled}, $\beta^{-1}\sG$ is given by a closed rational $1$-form $\omega_Z$ with values in a flat line bundle $\sL_Z$ whose zero set has codimension at least two.
Finally, $K_Z$ is not pseudo-effective since $\beta_*K_Z\sim_\mathbb{Z} K_X$ and $K_X$ is not pseudo-effective by assumption.
By Proposition \ref{prop:flatness_torsion}, we may therefore assume without loss of generality that $\sL_Z$ is torsion.
Applying \cite{takayama_fundamental_group}, we see that there exists a (torsion) flat line bundle $\sL_X$ on $X$ such that 
$\sL_Z\cong \beta^*\sL_X$. Then $\omega_Z$ induces a closed rational $1$-form $\omega_X$ on $X$ with values in $\sL_X$ whose zero set has codimension at least two. Moreover, $\omega_X$ defines $\sG$ by construction, completing the proof of the proposition. 
\end{proof}

Before proving Theorem \ref{thm:closed_rational_1_form_2} below, we address foliations defined by closed rational $1$-forms with values in flat line bundles on projective varieties with pseudo-effective canonical class.

\begin{prop}\label{prop:closed__rational_1_form_K_pseff}
Let $X$ be a normal complex projective variety with klt singularities, and let $\sG$ be a codimension one foliation on $X$ with 
canonical singularities and $K_\sG\equiv 0$. Suppose that $K_X$ is pseudo-effective, and that $\sG$ is given by a closed rational $1$-form $\omega$ with values in a flat line bundle $\sL$ whose zero set has codimension at least two. 
Then there exist an abelian variety $A$ and a quasi-\'etale cover $f \colon A \to X$ such that $f^{-1}\sG$ is a codimension one linear foliation on $A$. 
\end{prop}

\begin{proof}
There exists an effective divisor $D$ on $X$ such that $\sN_\sG\cong \sO_X(D)\otimes \sL$. On the other hand, 
$c_1(\sN_\sG)\equiv -K_X$ since $K_\sG\equiv 0$, and hence $-K_X \equiv D$. It follows that $K_X \equiv 0$ since $K_X$ is pseudo-effective by assumption. Proposition \ref{prop:closed__rational_1_form_K_pseff} then follows from 
Lemma \ref{lemma:K_torsion}. 
\end{proof}

\begin{proof}[Proof of Theorem \ref{thm:closed_rational_1_form_2}]
We maintain notation and assumptions of Theorem \ref{thm:closed_rational_1_form_2}. By Proposition \ref{prop:splitting_algebraic_transcendental}, we may assume without loss of generality that 
there is no positive-dimensional algebraic subvariety tangent to $\sG$ passing through a general point of
$X$ (see Step 1 of proof of Proposition \ref{prop:flatness_torsion}).

If $K_X$ is pseudo-effective, then the statement follows easily from Proposition \ref{prop:closed__rational_1_form_K_pseff}.

Suppose that $K_X$ is not pseudo-effective. By Proposition \ref{prop:flatness_torsion_2}, we may assume that $\sL$ is torsion.
Therefore, there exists a quasi-\'etale cover $f \colon X_1 \to X$ such that $f^*\sL\cong \sO_{X_1}$ as flat line bundles.
It follows that $f^{-1}\sG$ is given by a closed rational $1$-form. Note that $X_1$ canonical by
Fact \ref{fact:quasi_etale_cover_and_singularities}.
By Lemma \ref{lemma:canonical_quasi_etale_cover}, $f^{-1}\sG$ is canonical, and we obviously have $K_{f^{-1}\sG}\sim_\mathbb{Z}f^*K_\sG \equiv 0$. 
Theorem \ref{thm:closed_rational_1_form_2} then follows from Theorem \ref{thm:closed_rational_1_form}.
\end{proof}

Finally, we prove abundance in the setting of Proposition \ref{prop:flatness_torsion}.

\begin{prop}\label{prop:closed_form_abundance}
Let $X$ be a normal complex projective variety with terminal singularities, and let $\sG\subset T_X$ be a codimension one foliation with canonical singularities. Suppose that $\sG$ is given by a closed rational $1$-form 
$\omega$ with values in a flat line bundle $\sL$ whose zero set has codimension at least two. 
Suppose furthermore that $K_X$ is not pseudo-effective, and that $K_\sG\equiv 0$. 
Then $K_\sG$ is torsion. 
\end{prop}

\begin{proof}
By Proposition \ref{prop:flatness_torsion}, we may assume that $\sL$ is torsion.
Therefore, there exists a quasi-\'etale cover $f \colon X_1 \to X$ such that $f^*\sL\cong \sO_{X_1}$ as flat line bundles. It follows that $f^{-1}\sG$ is given by a closed rational $1$-form. Note that $X_1$ terminal by
Fact \ref{fact:quasi_etale_cover_and_singularities}.
By Lemma \ref{lemma:canonical_quasi_etale_cover}, $f^{-1}\sG$ is canonical, and  we obviously have $K_{f^{-1}\sG}\sim_\mathbb{Z}f^*K_\sG \equiv 0$. 
Proposition \ref{prop:closed_form_abundance} then follows easily from Theorem \ref{thm:closed_rational_1_form}.
\end{proof}

\section{Proofs of Theorems \ref{thm_intro:main} and \ref{thm_intro:abundance} and proof of Corollary \ref{cor_intro}}\label{section:proofs}

The present section is devoted to the proof of Theorems \ref{thm_intro:main} and \ref{thm_intro:abundance} and to the proof of Corollary \ref{cor_intro}.
Lemma \ref{lemma:not_p_closed} and Proposition \ref{prop:p_closed_or_not} below extend \cite[Theorem 7.5]{lpt} to the singular setting. 

\begin{lemma}\label{lemma:not_p_closed}
Let $X$ be a normal projective variety over some algebraically closed field $k$ of positive 
characteristic $p$ with $\dim X \ge 2$, and let $\sG \subset T_X$ be a codimension one foliation on $X$. 
Let $\omega$ be a rational $1$-form $\omega$ defining $\sG$, and let $B$ denotes the divisor whose support is the union of
codimension one zeroes and poles of $\omega$. If $\sG$ is not closed under $p$-th powers, then the following holds.
\begin{enumerate}
\item There exist a reduced effective Weil divisor $D$ on $X$ that does not contain any irreducible component of $B$ is its support, and $\alpha \in H^0\big(X,\Omega_X^{[1]}(\textup{log}\,(B+D))\big)$ with $d\alpha = 0$ such that $d\omega=\alpha\wedge\omega$. If $C$ is any irreducible component of $\textup{Supp}(B+D)$, then the residue of $\alpha$ at a general point of $C$ is a constant function with values in the prime field
$\mathbb{F}_p\subset k$.
\item Suppose that $X$ is smooth, and let $A$ be a nef divisor such that $T_X(A)$ is generated by global sections. 
Suppose furthermore that $\sG$ is semistable with respect to a nef and big divisor $H$ on $X$ and that $\mu_H(\sG) \ge 0$. Then we have
$$D \cdot H^{\dim X -1} \le c_1(\sN_\sG)\cdot H^{\dim X -1} + (\dim X-2) A \cdot H^{\dim X -1}:=M.$$
Let $C$ be any irreducible component of $\textup{Supp}(B+D)$ and let $m\in \mathbb{Z}$ be the vanishing order of $\omega$ along $C$. Then we have $\textup{res}_C\,\alpha \in \{m,\ldots,m+M\}\subseteq \mathbb{F}_p$.
\end{enumerate}
\end{lemma}

\begin{proof}
Let $X^\circ$ denotes an open set of $X$ with complement of codimension at least two, contained in the regular loci of $X$ and $\sG$.
Let $(U_i)_{i\in I}$ be a finite covering of $X^\circ$ by open affine subsets and let $\omega_i$ 
be a regular $1$-form on $U_i$ with zero set of codimension at least two such that $\omega_i \wedge \omega =0$.
Write $\omega_i= g_i\omega$ for some rational function $g_i$ on $X$, and set $g_{ij}:=\frac{g_i}{g_j}$.
We have $\omega_i = g_{ij} \omega_j$, and hence $g_{ij}$ is a nowhere vanishing regular function on $U_i\cap U_j$.

Set $\sG^\circ:=\sG_{|X^\circ}$, and consider the non-zero map 
$\Frobabs^*\sG^\circ \to T_{X^\circ}/\sG^\circ={\sN_\sG}_{|X^\circ} =:\sN_\sG^\circ$ induced by the $p$-th power operation.
Note that $\sN_{\sG}^\circ$ is a line bundle. Shrinking $X^\circ$, if necessary, we may assume that there is an effective divisor $D^\circ$ on $X^\circ$ such that the above map induces a surjective morphism 
$$\Frobabs^*\sG^\circ \twoheadrightarrow \sN_\sG^\circ(-D^\circ).$$
We may also assume without loss of generality, that there exits a regular vector field $\partial_i$ on $U_i$ such that
$f_i:=\omega_i(\partial_i^p)$ is a defining equation of $D^\circ$ on $U_i$. By \cite[Proposition 7.3]{lpt}, $\frac{1}{f_i}\omega_i$ is a closed rational $1$-form, and hence $d\omega_i=\frac{df_i}{f_i}\wedge \omega_i$. On the other hand, by 
\cite[Corollary 7.4]{lpt}, we must have $$\frac{df_i}{f_i}-\frac{df_j}{f_j}=\frac{dg_{ij}}{dg_{ij}},$$
and hence
$$\frac{df_i}{f_i}-\frac{dg_i}{g_i}=\frac{df_j}{f_j}-\frac{dg_j}{g_j}.$$
This immediately implies that there exists $\alpha \in H^0\big(X,\Omega_X^{[1]}(\textup{log}\,(B+D_{\textup{red}})\big)$ with $d\alpha = 0$ such that $\alpha$ restricts to $-\frac{dg_i}{g_i}+\frac{df_i}{f_i}$ on $U_i$, where $D$ denotes the Weil divisor on $X$ such that $D_{|X^\circ}=D^\circ$. 
A straightforward computation then shows that
$$d\omega=\alpha\wedge \omega.$$
This proves (1).

To prove (2), observe that we must have 
$$\mu_{\textup{min}}\big(\Frobabs^*\sG\big) \le 
\mu_H\big(\sN_\sG(-D)\big)=\mu_H(\sN_\sG)-\mu_H\big(\sO_X(D)\big)\le \mu_H(\sN_\sG)-\mu_H\big(\sO_X(D_\textup{red})\big).$$ On the other hand, by the proof of 
\cite[Corollary 2.5]{langer_ss_sheaves}), we have 
$$\mu_{\textup{max}}\big(\Frobabs^*\sG\big)-\mu_{\textup{min}}\big(\Frobabs^*\sG\big)\le (\textup{rank}\,\sG-1) A \cdot H^{\dim X-1}.$$
Now, we must have $\mu_{\textup{max}}\big(\Frobabs^*\sG\big)\ge 0$ since $\mu_H(\sG) \ge 0$ by assumption. The claim then follows easily.
\end{proof}

\begin{rem}Notation as in the proof of Lemma \ref{lemma:not_p_closed}. By \cite[Corollary 7.4]{lpt}, there is a rational function $h_{ij}$ on $X$ such that $f_i=g_{ij}f_jh_{ij}^p$.
Note that $h_{ij}$ is regular on $U_i\cap U_j$ since both $f_i$ and $f_j$ are local equations of $D^\circ$ on $U_i\cap U_j$, and that the $h_{ij}$ automatically satisfy the cocycle condition.
Therefore, there exists a rank one reflexive sheaf $\sL$ on $X$ as well as an effective Weil divisor such that $\sN \cong \sO_X(D)\boxtimes \sL^{\otimes p}$.
\end{rem}

\begin{prop}\label{prop:p_closed_or_not}
Let $X$ be a normal complex projective variety, and let $\sG$ be a codimension one foliation on $X$. 
Let $\beta\colon Z \to X$ be a resolution of singularities.
Suppose that $\sG$ is semistable with respect to some ample divisor $H$ on $X$ with $\mu_H(\sG)\ge 0$.
Then either $\sG$ is closed under $p$-th powers for almost all primes $p$, or $\beta^{-1}\sG$ is given by
a closed rational $1$-form with values in a flat line bundle, whose zero set has codimension at least two.
\end{prop}

\begin{rem}\label{rem:p_closed_or_not}
In the setup of Proposition \ref{prop:p_closed_or_not}, suppose in addition that $X$ has klt singularities.
Suppose that $\beta^{-1}\sG$ is given by a closed rational $1$-form $\omega_Z$ with values in a flat line bundle $\sL_Z$ whose zero set has codimension at least two.
Applying \cite{takayama_fundamental_group}, we see that there exists a flat line bundle $\sL_X$ on $X$ such that 
$\sL_Z\cong \beta^*\sL_X$. Then $\omega_Z$ induces a closed rational $1$-form $\omega_X$ on $X$ with values in $\sL_X$ defining $\sG$. Moreover, its zero set has codimension at least two.
\end{rem}

\begin{proof}[Proof of Proposition \ref{prop:p_closed_or_not}]
Proposition \ref{prop:p_closed_or_not} follows from Lemma \ref{lemma:not_p_closed} using the spreading out technique, which we recall now. Assume that $\dim X \ge 2$.

Let $\omega$ be a rational $1$-form defining $\beta^{-1}\sG$.
Let $E$ be an effective $\beta$-exceptional divisor such that 
$A:=m_0\beta^*H - E$ is ample for some positive integer $m_0$, and let $B$ be a reduced effective divisor that contains codimension one zeroes and poles of $\omega$ and all $\beta$-exceptional divisors in its support.
Replacing $H$ by $m_0H$, we may assume that $m_0=1$.
Let $m$ be a positive integer such that $T_Z(mA)$ is generated by global sections, and 
set $$M:=c_1\big(\sN_{\beta^{-1}\sG}\big)\cdot (\beta^* H)^{\dim X -1} + m(\dim X-2) A \cdot (\beta^*  H)^{\dim X -1}.$$

Let $R \subset \mathbb{C}$ be a finitely generated $\mathbb{Z}$-algebra, and let $\bX$ (resp. $\bZ$) be a projective 
(resp. smooth projective) model of $X$ (resp. $\bZ$) over $\bS :=\textup{Spec} \,R$.
Let $\boldsymbol{\beta}\colon \bZ \to \bX$ be a projective birational morphism such that $\boldsymbol{\beta}_\mathbb{C}$ coincides with $\beta$. Let $\sbfG$ (resp. $\boldsymbol{\beta}^{-1}\sbfG$) 
be a saturated subsheaf of the relative tangent sheaf $T_{\bX/\bS}$ (resp. $T_{\bZ/\bS}$), flat over $\bS$, such that $\sbfG_\mathbb{C}$ (resp. $\boldsymbol{\beta}^{-1}\sbfG_\mathbb{C}$) coincides
with $\sG$ (resp. $\beta^{-1}\sG$). We may assume that for any closed point $s \in \bS$, $\bX_{\bar{s}}$ is normal, and that $\sbfG_{\bar{s}}$ (resp. $\boldsymbol{\beta}^{-1}\sbfG_{\bar{s}}$) is a foliation on $\bX_{\bar{s}}$ (resp. $\bZ_{\bar{s}}$). Let $\bH$ (resp. $\bA$) be an ample Cartier divisor on $\bX$ (resp. $\bZ$) such that $\bH_\mathbb{C} \sim_\mathbb{Z} H$ (resp. $\bA_\mathbb{C} \sim_\mathbb{Z} A$). Suppose that 
$\bA=\boldsymbol{\beta}^*H-\bE$ for some effective $\boldsymbol{\beta}$-exceptional relative divisor $\bE$ over $\bS$, and that 
$E_{\bar{s}}$ is $\boldsymbol{\beta}_{\bar{s}}$-exceptional for any geometric point $s \in \bS$.
Finally, let $\bB$ be a reduced effective relative divisor on $\bZ$ over $\bS$ and let
$\boldsymbol{\omega}$ be rational section of $\Omega^1_{\bZ/\bT}$ such that $\bB_{\mathbb{C}}=B$ and 
$\boldsymbol{\omega}_{\mathbb{C}}=\omega$. Shrinking $\bS$, if necessary, we may assume that $\boldsymbol{\omega}_{\bar{s}}$ defines 
$\boldsymbol{\beta}^{-1}\sbfG_{\bar{s}}$ for any geometric point $s \in \bS$, and that $\bB_{\bar{s}}$ contains the codimension one zeroes and poles of $\boldsymbol{\omega}_{\bar{s}}$. We can also choose $\bB$ so that 
$\bB$ contains all $\boldsymbol{\beta}$-exceptional prime divisors in its support.

Since semistability with respect to an ample divisor is an open condition in flat families of sheaves (see proof of 
\cite[Proposition 2.3.1]{HuyLehn}), we may assume that the sheaves $\sG_{\bar{s}}$ are semistable with respect to $\bH_{\bar{s}}$, with slopes $\mu_{\bH_{\bar{s}}}\big(\sG_{\bar{s}}\big)=\mu_H(\sG) \ge 0$.
It follows that $\boldsymbol{\beta}_{\bar{s}}^{-1}\sbfG_{\bar{s}}$ is semistable with respect to $\boldsymbol{\beta}_{\bar{s}}^*\bH_{\bar{s}}$ with slope $\mu_{\boldsymbol{\beta}_{\bar{s}}^*\bH_{\bar{s}}}\big(\boldsymbol{\beta}^{-1}\sbfG_{\bar{s}}\big)=0$.

By \cite[Lemme 2.4]{fga221} and \cite[Proposition 4.1]{fga232}, there exists a quasi-projective $\bS$-scheme 
$\textup{Div}_{\bZ/\bS}^{\le M}$ parameterizing effective relative Cartier divisor on $\bZ$ over $\bS$ with $\bA$-degree at most $M$.
Using \cite[Th\'eor\`eme 12.2.1]{ega28}, we see that there is an open set $\bT \subset \textup{Div}_{\bZ/\bS}^{\le M}$ that parametrizes geometrically reduced divisors. Replacing $\bT$ by $\bT_{\textup{red}}$, if necessary, we may assume that 
$\bT$ is reduced. By generic flatness, we can suppose that $\bT$ is flat over $\bS$.
Let $\bD \subset \bT \times_\bS \bZ$ be the universal effective relative Cartier divisor,
and denote by $\boldsymbol{\pi}\colon \bT\times_{\bS} \bZ\to \bZ$ and
$\boldsymbol{\upsilon}\colon \bT\times_{\bS} \bZ \to \bT$ the projections.
Set $\bC:=\boldsymbol{\pi}^*\bB$. Observe that $\bC$ is a reduced relative effective Cartier divisor on 
$\bT\times_{\bS} \bZ$ over $\bT$. Write $\Omega^{[1]}_{\bT\times_{\bS} \bZ/\bT}\big(\textup{log}\,(\bC+\bD)\big)$
for the reflexive sheaf on $\bT\times_{\bS} \bZ$ whose restriction to the open set $\bZ_\bT^\circ$ where $\bC+\bD$ has relative simple normal crossings over $\bT$ is $\Omega^{1}_{\bZ_\bT^\circ/\bT}\big(\textup{log}\,(\bC_{|\bZ_\bT^\circ}+\bD_{|\bZ_\bT^\circ})\big)$, and set
$$\bU:=\boldsymbol{\upsilon}_* \Omega^{[1]}_{\bT\times_{\bS} \bZ/\bT}\big(\textup{log}\,(\bC+\bD)\big).$$ 
By generic flatness and the base change theorem, we see that, replacing $\bT$ with a finite disjoint union of locally closed subsets, we may assume without loss of generality that 
$\bU$ is flat over $\bT$, and that the formation of
$\boldsymbol{\upsilon}_* \Omega^{[1]}_{\bT\times_{\bS} \bZ/\bT}\big(\textup{log}\,(\bC+\bD)\big)$
commutes with arbitrary base change.
We will also assume that, for any geometric point $\bar{t} \in \bT$, 
the restriction of $\Omega^{[1]}_{\bT\times_{\bS} \bZ/\bT}\big(\textup{log}\,(\bC+\bD)\big)$ to the fiber of the projection
$\bT\times_{\bS} \bZ \to \bT$ over $\bar{t}$ is reflexive, so that 
$${\Omega^{[1]}_{\bT\times_{\bS} \bZ/\bT}\big(\textup{log}\,(\bC+\bD)\big)}_{|\bZ_{\bar{s}}}\cong 
\Omega^{[1]}_{\bZ_{\bar{s}}}\big(\textup{log}\,(\bB_{\bar{s}}+\bD_{\bar{t}})\big),$$
where $\bar{s} \in \bS$ is the image of $\bar{t}$ in $\bS$.
By \cite[Corollaire 9.7.9]{ega28}, there exist a finite set $I \subset \mathbb{N}$ and a decomposition 
$$\bT=\bigsqcup_{i \in I} \bT_i$$
of $T$ into locally closed subsets such that
any geometric fiber of $\bD_i:=\bD\times_{\bT}\bT_i \to \bT_i$ has $i$ irreducible components, and such that 
any geometric fiber of $\bC_i:=\bC\times_{\bT}\bT_i \to \bT_i$ has $n(i)$ irreducible components.
Let $\bD_i^\circ$ (resp. $\bC_i^\circ$) denotes the open set where $\bD_i \to \bT_i$ 
(resp. $\bC_i \to \bT_i$) is smooth. 
We can choose $\bT_i$ so that there exist sections $a_{i,1},\ldots,a_{i,i}$ 
(resp. $b_{i,1},\ldots,b_{i,n(i)}$)
of $\bD_i^\circ \to \bT_i$ (resp. $\bC_i^\circ\to \bT_i$)
with $a_{i,j}(\bar{t})\in \bD^\circ_{\bar{t},i,j}$ (resp $b_{i,j}(\bar{t})\in \bC^\circ_{\bar{t},i,j}$), where the
$\bD^\circ_{\bar{t},i,1},\ldots, \bD^\circ_{\bar{t},i,i}$ (resp. 
$\bC^\circ_{\bar{t},i,1},\ldots,\bC^\circ_{\bar{t},i,n(i)}$) 
are the irreducible components of the corresponding fiber of $\bD_i^\circ \to \bT_i$ (resp. $\bC_i^\circ \to \bT_i$).

Let $\bar{t}\in T$ be a geometric point, and denote by $\bar{s}$ the image of $\bar{t}$ in $\bS$.
Let $l_{\bar{t},i,j}$ (resp. $m_{\bar{t},i,j}$) be the vanishing order of $\boldsymbol{\omega}_{\bar{s}}$ along
$\bD^\circ_{\bar{t},i,j}$ (resp. $\bC^\circ_{\bar{t},i,j}$).

Let $\bP\subseteq \bU$ be the closed subset defined by the conditions 
\begin{enumerate}
\item $d\alpha(\bar{s})=0$,
\item $d\boldsymbol{\omega}_{\bar{s}}= \alpha(\bar{s})\wedge\boldsymbol{\omega}_{\bar{s}}$,
\item $\textup{res}_{\bD_{\bar{t},i,j}^\circ}\alpha(\bar{s})\big(a_{i,j}(\bar{t})\big) \in \{l_{\bar{t},i,j},\ldots,l_{\bar{t},i,j}+M\}\subset k(\bar{t})$ for all indices $1 \le j\le i$, and
\item $\textup{res}_{\bC_{\bar{t},i,j}^\circ}\alpha(\bar{s})\big(b_{i,j}(\bar{t})\big) \in \{m_{\bar{t},i,j},\ldots,m_{\bar{t},i,j}+M\}\subset k(\bar{t})$ for all indices $1 \le j\le n(i)$,
\end{enumerate}
where $\bar{t}\in\bT_i$ 
and $\alpha(\bar{s}) \in H^0\Big(\bZ_{\bar{s}},\Omega^{[1]}_{\bZ_{\bar{s}}}\big(\textup{log}\,(\bB_{\bar{s}}+\bD_{\bar{t}})\big)\Big)$.

Fix a closed point $s$ in $\bS$, and denote by $p>0$ the characteristic of $k(\bar{s})$. Suppose that $\sbfG_{\bar{s}}$ is not closed under $p$-th powers. This immediately implies that $\boldsymbol{\beta}^{-1}\sbfG_{\bar{s}}$ is not closed under $p$-th powers as well. Applying Lemma \ref{lemma:not_p_closed} to $\boldsymbol{\beta}^{-1}\sbfG_{\bar{s}}$ and 
$\boldsymbol{\beta}_{\bar{s}}^*\bH_{\bar{s}}$, we conclude that 
there exist a reduced effective Cartier divisor $\bD(\bar{s})$ on $\bZ_{\bar{s}}$ 
that does not contain any irreducible component of $\bB_{\bar{s}}$ in its support,
and $$\alpha(\bar{s}) \in H^0\big(\bZ_{\bar{s}},\Omega_{\bZ_{\bar{s}}}^{[1]}(\textup{log}\,(\bB_{\bar{s}}+\bD(\bar{s})))\big)$$ with $d\alpha(\bar{s}) = 0$ such that $d\boldsymbol{\omega}_{\bar{s}}=\alpha(\bar{s})\wedge\boldsymbol{\omega}_{\bar{s}}$.
Moreover, the functions $\textup{res}_{\bD_{\bar{t},i,j}^\circ}\alpha(\bar{s})$ 
and $\textup{res}_{\bC_{\bar{t},i,j}^\circ}\alpha(\bar{s})$
are constant
with values in $\{l_{\bar{t},i,j},\ldots,l_{\bar{t},i,j}+M\}\subset k(\bar{s})$
and 
$\{m_{\bar{t},i,j},\ldots,m_{\bar{t},i,j}+M\}\subset k(\bar{s})$ respectively, and
$\bD(\bar{s}) \cdot (\boldsymbol{\beta}_{\bar{s}}^*\bH_{\bar{s}})^{\dim \bZ_{\bar{s}} -1} \le M$.
Notice that there is no $\boldsymbol{\beta}_{\bar{s}}$-exceptional prime divisor contained in $\textup{Supp}\,\bD(\bar{s})$ by construction. It follows that
\begin{multline*}
\bD(\bar{s}) \cdot \bA_{\bar{s}}^{\dim \bZ_{\bar{s}}-1} =
\bD(\bar{s}) \cdot (\boldsymbol{\beta}_{\bar{s}}^*\bH_{\bar{s}})^{\dim \bZ_{\bar{s}} -1}
-\sum_{0\le i \le \dim \bZ_{\bar{s}}-2}\bD(\bar{s}) \cdot\bE_{\bar{s}}\cdot\bA_{\bar{s}}^i\cdot 
(\boldsymbol{\beta}_{\bar{s}}^*\bH_{\bar{s}})^{\dim \bZ_{\bar{s}}-2-i} \\
\le  
 \bD(\bar{s}) \cdot (\boldsymbol{\beta}_{\bar{s}}^*\bH_{\bar{s}})^{\dim \bZ_{\bar{s}} -1} \le M, 
\end{multline*}
and hence, $\alpha(\bar{s})$ yields a closed point in $\bP$ over $\bar{t}:=[\bD(\bar{s})] \in \bT$. 

If the set of closed points $\bar{s}$ in $\bS$ such that $\sbfG_{\bar{s}}$ is not closed under $p$-th powers
is Zarsiki dense, then the image of $\bP \to \bS$ contains the generic point of $\bS$ by a theorem of Chevalley.
It follows that there exists a closed logarithmic $1$-form $\alpha$ on $Z$ such that $d\omega=\alpha\wedge\omega$. 
Let $C$ be any prime divisor $C$ on $X$.
Since the residue $\textup{res}_C\,\alpha$ of $\alpha$ at a general point of $C$ is a constant function, we must
have $\textup{res}_C\, \alpha\in \{m,\ldots,m+M\} \subset \mathbb{Z}$, where $m$ denotes the order of vanishing of $\omega$ along $C$. 

Let $(U_i)_{i\in I}$ be a covering of $X$ by analytically open sets such that $\alpha_{|U_i}=d\ln f_i + dg_i$ where $f_i$ (resp. $g_i$) is a meromorphic (resp. holomorphic) function on $U_i$. Then
$$\omega_i:=\frac{1}{f_i\exp(g_i)}\omega_{|U_i}$$ is a closed rational $1$-form with zero set of codimension at least two
and $\omega_i = c_{ij} \omega_j$ on $U_i \cap U_i$ for some $c_{ij}\in \mathbb{C}$.
This completes the proof of the proposition.
\end{proof}

We end the preparation for the proof of our main results with the following observation.

\begin{lemma}\label{lemma:bogomolov}
Let $X$ be a normal complex projective variety with klt singularities, and let $\sG$ be a codimension one foliation on $X$ with $K_\sG$ $\mathbb{Q}$-Cartier and $K_\sG\equiv 0$. Then $\nu(X)\le 1$.
\end{lemma}

\begin{proof}
Let $\beta \colon Z \to X$ be a resolution of singularities with exceptional set $E$, and suppose that $E$ is a divisor with simple normal crossings.
Let $E_1$ be the reduced divisor on $Z$ whose support is
the union of all irreducible components of $E$ that are invariant under $\beta^{-1}\sG$.
Note that $-c_1(\sN_\sG)\equiv K_X$ by assumption. By Proposition \ref{prop:numerical_dimension} and Remark \ref{rem:exceptional_set}, there exists a rational number $0 \le \varepsilon <1$ such that 
$$\nu(X)=\nu\big(-c_1(\sN_\sG)\big) =  \nu\big(-c_1(\sN_{\beta^{-1}\sG})+\varepsilon E_1\big).$$ 
The lemma then follows from \cite[Proposition 9.3]{touzet_conpsef}.
\end{proof}

\begin{proof}[Proof of Theorem \ref{thm_intro:main}] We maintain notation and assumptions of Theorem \ref{thm_intro:main}.
By \cite[Lemma 2.53]{kollar_mori} and Fact \ref{fact:quasi_etale_cover_and_singularities}, there exists a quasi-\'etale cover $f \colon X_1 \to X$ 
with $X_1$ canonical 
such that $f^*K_{\sG}\sim_\mathbb{Z} 0$. 
By Lemma \ref{lemma:canonical_quasi_etale_cover}, $f^{-1}\sG$ is canonical and $K_{f^{-1}\sG}\sim_\mathbb{Z}0$. 
Replacing $X$ by $X_1$, if necessary, we may assume that $K_\sG\sim_\mathbb{Z}0$.

By Proposition \ref{prop:p_closed_or_not} and Lemma \ref{lemma:stability_versus_uniruled},
either $\sG$ is closed under $p$-th powers for almost all primes $p$, or it is given by
a closed rational $1$-form $\omega$ with values in a flat line bundle whose
zero set has codimension at least two (see also Remark \ref{rem:p_closed_or_not}). In the latter case, the statement follows from Theorem \ref{thm:closed_rational_1_form_2}.

Suppose that $\sG$ is closed under $p$-th powers for almost all primes $p$.
By Lemma \ref{lemma:bogomolov}, we have $\nu(X)\le 1$.
If $\nu(X)=-\infty$, then Corollary \ref{cor:grothendieck_katz} says that 
$\sG$ is algebraically integrable. If $\nu(X)=1$, then $\sG$ has algebraic leaves as well by 
Corollary \ref{cor:algebraic_integrability_nu_un}. In either case, the statement follows from 
Theorem \ref{thm_intro:global_reeb_stability}. If $\nu(X)=0$, then Theorem \ref{thm_intro:main} follows from
Lemma \ref{lemma:K_torsion} and Proposition \ref{prop:nu_zero_versus_torsion}.
\end{proof}

\begin{proof}[Proof of Theorem \ref{thm_intro:abundance}] We maintain notation and assumptions of Theorem \ref{thm_intro:abundance}.
By Proposition \ref{prop:p_closed_or_not} and Lemma \ref{lemma:stability_versus_uniruled},
either $\sG$ is closed under $p$-th powers for almost all primes $p$, or it is given by
a closed rational $1$-form $\omega$ with values in a flat line bundle whose
zero set has codimension at least two (see also Remark \ref{rem:p_closed_or_not}). 

Suppose first that $\sG$ is closed under $p$-th powers for almost all primes $p$.
By Lemma \ref{lemma:bogomolov}, we have $\nu(X)\le 1$.
If $\nu(X)=-\infty$, Theorem \ref{thm:grothendieck_katz} and Corollary \ref{cor:grothendieck_katz} then imply that 
$\sG$ is algebraically integrable. If $\nu(X)=1$, then $\sG$ has algebraic leaves by Theorem \ref{thm:algebraic_integrability_nu_un} and Corollary \ref{cor:algebraic_integrability_nu_un}. In either case, 
$K_\sG$ is torsion by Proposition \ref{prop:abundance_alg_int}.
If $\nu(X)=0$, then $K_\sG$ is torsion by Lemma \ref{lemma:K_torsion} and Proposition \ref{prop:nu_zero_versus_torsion}.

Suppose now that $\sG$ is given by a closed rational $1$-form $\omega$ with values in a flat line bundle whose
zero set has codimension at least two. If $\nu(X) \ge 0$, then the statement follows from Proposition \ref{prop:closed__rational_1_form_K_pseff}. If $\nu(X)=-\infty$, then Theorem \ref{thm_intro:abundance} follows from 
Proposition \ref{prop:closed_form_abundance}.
\end{proof}

\begin{proof}[Proof of Corollary \ref{cor_intro}] We maintain notation and assumptions of Corollary \ref{cor_intro}.
Arguing as in the proof of Theorem \ref{thm_intro:main}, we see that we may assume without loss of generality that 
$K_\sG$ is Cartier. Lemma \ref{lemma:regular_versus_canonical} then implies that $\sG$ is canonical, so that Theorem \ref{thm_intro:main} applies. In particular, to prove Corollary \ref{cor_intro}, it suffices to consider the case where  
$X$ is an equivariant compactification of a commutative algebraic group $G$ of dimension at least $2$ and 
$\sG\cong \sO_X^{\, \dim X -1}$ is induced by a codimension one Lie subgroup $H \subset G$. If $X$ is not uniruled, then 
$G$ must be an abelian variety by a theorem of Chevalley, and $\sG$ is a linear foliation on $X$, so that we are in case (2) of Corollary \ref{cor_intro}. 

Suppose from now on that $X$ is uniruled. 
Let $\beta\colon Z \to X$ be an equivariant resolution of $X$ with exceptional set $E$, and assume that $E$ is a divisor with simple normal crossings and that $\beta$ induces an isomorphism over $X_\textup{reg}$. By Corollary \ref{cor:equivariant_resolution}, there is an inclusion $\beta^*\sG \subset T_Z(-\textup{log}\,E)$. In particular, we must have 
$\beta^*\sG \subset \beta^{-1}\sG$. Since $\sG$ is canonical, we conclude that $K_{\beta^{-1}\sG}\sim_\mathbb{Z}\beta^*K_\sG$ and that $\beta^*\sG \cong \beta^{-1}\sG$. 
This implies that any irreducible component of $E$ is invariant under $\beta^{-1}\sG$.
This also implies that
$\beta^{-1}\sG$ is regular by Lemma \ref{lemma:regular_bir_crepant_map}.
Since $Z$ is uniruled by assumption, there exists a $\mathbb{P}^1$-bundle structure $\phi\colon Z \to Y$ onto a complex projective manifold $Y$ with $K_Y\sim_\mathbb{Z}0$ such that $\beta^{-1}\sG$ induces a flat connection on $\phi$. This follows either from \cite{touzet} or from the proof of \cite[Proposition 5.1] {druel_bbcd2}.
Suppose that $E\neq\emptyset$, and let $E_1$ be an irreducible component of $E$. Then 
$E_1$ is smooth and $\sN_{E_1/Z}$ is flat since $E_1$ is invariant under $\beta^{-1}\sG$. But this contradicts the fact that $E_1$ is $\beta$-exceptional. It follows that $X$ is as in case (1) of Corollary \ref{cor_intro}.
This finishes the proof of the corollary.
\end{proof}


\providecommand{\bysame}{\leavevmode\hbox to3em{\hrulefill}\thinspace}
\providecommand{\MR}{\relax\ifhmode\unskip\space\fi MR }
\providecommand{\MRhref}[2]{%
  \href{http://www.ams.org/mathscinet-getitem?mr=#1}{#2}
}
\providecommand{\href}[2]{#2}

\end{document}